\documentclass[11pt,reqno]{amsproc}  \usepackage[colorlinks,linkcolor=Green,citecolor=blue]{hyperref}
\usepackage{lineno}
\usepackage{times}
\usepackage{amsmath}
\usepackage{amssymb}
\usepackage{mathrsfs}
\usepackage{amsthm}
\usepackage{bm}
\usepackage{subcaption}
\usepackage{graphicx}
\usepackage{algorithm} 
\usepackage{algpseudocode}
\usepackage{color}
\definecolor{Green}{RGB}{0,128,0}
\usepackage{enumitem}
\usepackage{tikz}
\usetikzlibrary{positioning, arrows.meta, calc, shadows}
\usepackage{amsmath, amssymb} 
\usepackage{float}
\usepackage{geometry}
\newtheorem{Def}{Definition}[section]
\newtheorem{lemma}[Def]{Lemma}
\newtheorem{assumption}[Def]{Assumption}
\newtheorem{corollary}[Def]{Corollary}

\newtheorem{theorem}[Def]{Theorem}
\newtheorem{example}[Def]{Example}

\newtheorem{remark}[Def]{Remark}

\usepackage{color,xcolor}
\definecolor{Green}{RGB}{46 139 87} 

\newcommand{\ud}{\mathrm d}
\newcommand{\R}{\mathbb{R}}
\newcommand{\C}{\mathbb{C}}
\numberwithin{equation}{section}

\newcommand{\E}{\mathbb{E}}\allowdisplaybreaks[4]

\begin{document}

\title[FEM for SPDE on graph]{
A regularized truncated finite element method
for degenerate parabolic stochastic PDE on non-compact graph}

\author{Jianbo Cui}
\address{Department of Applied Mathematics, The Hong Kong Polytechnic University, Hung Hom, Kowloon, Hong Kong}
\email{jianbo.cui@polyu.edu.hk}
\author{Mih\'aly Kov\'acs}
\address{Department of Mathematical Sciences, Chalmers, Gothenburg, SE-41296, Sweden;
Department of Analysis and Operations Research, University of Technology and Economics, Műegyetem rkp. 3-9, Budapest, H-1111, Hungary;
Faculty of Information Technology and Bionics, Pázmány Péter Catholic University, Práter u. 50/a, Budapest, H-1083, Hungary}
\email{mkovacs@math.bme.hu}

\author{Derui Sheng}
\address{Department of Applied Mathematics, The Hong Kong Polytechnic University, Hung Hom, Kowloon, Hong Kong}
\email{sdr@lsec.cc.ac.cn}

\email{}

\thanks{The research  is partially supported by NSFC grant 12522119, NSFC grant 12301526, MOST National Key \& Program No.
2024YFA1015900, the Hong Kong Research Grant Council GRF grant 15302823, NSFC/RGC Joint Research Scheme N\_PolyU5141/24, the internal
grants (P0045336, P0046811) from the Hong Kong Polytechnic University, and the CAS AMSS-PolyU Joint Laboratory of Applied Mathematics.}


\begin{abstract}
We study the numerical approximation of a class of degenerate parabolic stochastic partial differential equations on non-compact metric graphs, which naturally arise in the asymptotic analysis of Hamiltonian flows under small noise perturbations. The numerical discretization of these equations faces several challenges, including the non-compactness of the graph, the degeneracy of the differential operator near vertices, and the non-symmetry of the associated bilinear form. To address these issues, we propose a multi-step numerical strategy combining graph truncation, localized coefficient regularization, and finite element spatial discretization. By incorporating localization techniques, tightness arguments and resolvent estimates, we establish the strong  convergence of the proposed scheme in a weighted $L^2$-space. Our results provide a systematic methodology that is potentially extensible to more general non-compact graphs and degenerate operators.
\end{abstract}

\maketitle
\section{Introduction}
A graph $\Gamma$ consists of a finite set of vertices $\{O_i\}$ and a finite set of edges $\{I_k\}$. Each edge $I_k$ connects a pair of vertices and can be isometrically identified with an interval $[a_k,b_k]$ with $-\infty < a_k\le b_k\le\infty$. A location $(z,k)\in\Gamma$ on the graph represents a position $z\in [a_k,b_k]$ on the edge $I_k$. We assume that the graph $\Gamma$ is connected, so that there is a path between any two vertices on the graph. Equipped with the shortest path distance, defined as the length of the shortest path connecting any two locations, $\Gamma$ becomes a metric graph.
Stochastic partial differential equations (PDEs) on metric graphs have recently emerged as fundamental modeling tools in neuroscience, physics, and network dynamics. They provide a continuum formulation of stochastic PDEs on discrete graphs (see, e.g., \cite{MR4612606, CLZ23}) and can describe, for example, the propagation of electrical signals along dendritic trees, where stochastic impulsive inputs account for excitatory and inhibitory influences from neighboring neurons \cite{BMZ08, BM10}. Beyond neuroscience, stochastic PDEs on graphs appear in a wide range of applications, including free-electron models for organic molecules \cite{LP36}, superconductivity in granular and engineered materials \cite{AS83}, wave propagation in acoustic and electromagnetic networks \cite{CRH87}, Anderson transitions in disordered wires \cite{AMR20, SB82}, quantum chaos \cite{BK13}, and statistical modeling \cite{BSW24}.

A particularly important class of stochastic PDEs on graphs arises from the averaging principle for stochastic perturbations of Hamiltonian flows. This principle, initially introduced in \cite{FW94} for small noise perturbations, has been extensively developed in subsequent works (see, e.g., \cite{FM98, FM01, FW04, FW12}). As the perturbation scale tends to zero, the effective dynamics is no longer confined to the original Euclidean space but instead evolves on a graph $\Gamma$, obtained by collapsing each connected component of a level set of the Hamiltonian $H$ into a single point. More precisely, the limiting dynamics along the edges $\{I_k\}_{k=1}^m$ of $\Gamma$ is governed by a second-order differential operator
\begin{equation}
	\mathcal{L} f(z,k)
	= \frac{1}{2\beta_k(z)} \frac{\ud}{\ud z}
	\Bigl( \alpha_k(z) \frac{\ud f}{\ud z}(z,k) \Bigr),
	\qquad  \text{for } z\in (a_k,b_k),k=1,2,\ldots,m
\end{equation}
with Kirchhoff-type gluing conditions at interior vertices (see \eqref{eq:glue}). Here, $\alpha_k$ and $\beta_k$ are functions determined by the Hamiltonian $H$ (see \eqref{eq:AT}).

In this paper, we study the following stochastic PDE on the graph $\Gamma$:
\begin{equation}\label{eq:SPDE}
	\partial_t u(t,z,k)
	= \mathcal{L} u(t,z,k)
	+ b(u(t,z,k))
	+ g(u(t,z,k))\partial_t W(t,z,k),
\end{equation}
for $t\in(0,T]$ and $(z,k)\in\Gamma$ with the initial condition $u(0,z,k)=u_0(z,k)$. Here, the functions $b, g : \R \to \R$ are Lipschitz continuous, and $W$ denotes a $Q$-Wiener process on an appropriate Hilbert space on the graph $\Gamma$ (see Section \ref{S:Pr} for further details). The model \eqref{eq:SPDE}, first proposed in \cite{CF19}, describes the
asymptotic behavior of particles moving along an incompressible flow in $\R^2$, incorporating the effects
of small viscosity, the flow pattern, and a slow chemical reaction involving the particles; see also \cite{CX21}. We also refer to \cite{CH24,CS25} for the sample path large and moderate deviation principles of stochastic PDEs on graphs with small noise.

The absence of analytic solutions for stochastic PDEs on graphs, particularly in the nonlinear regime,  necessitates the development of robust computational methods. However, the nontrivial topology of the underlying graph and the diversity of vertex gluing conditions render the discretization of stochastic PDEs on graphs substantially more intricate than those in the Euclidean setting. To the best of our knowledge, the only existing works in this direction are \cite{BKKS24,CS24}. For a class of stochastic fractional elliptic PDEs on compact metric graphs, the authors in \cite{BKKS24} proposed and analyzed a numerical method which  combines  finite element approximation with a rational approximation of the fractional power. In the special case where the Hamiltonian $H$ admits a unique critical point, \cite{CS24} analyzed the convergence and asymptotic-preserving properties of a temporal semi-discretization for \eqref{eq:SPDE}  based on the exponential Euler scheme.

In the present work, we address the spatial discretization of \eqref{eq:SPDE}, which raises several challenges. First, in many physically relevant settings, the Hamiltonian acts as a confining potential, leading to a graph $\Gamma$ that contains an unbounded edge. Second, the operator $\mathcal{L}$ exhibits qualitatively distinct behavior at different types of vertices, depending on whether the corresponding critical point of the Hamiltonian is a saddle point or a local extremum. In particular, the coefficient $\alpha_k(z)$ degenerates linearly near exterior vertices, while $\beta_k(z)$ diverges logarithmically near interior vertices (see Remark \ref{rem:AT}), and thus  $\mathcal{L}$ degenerates in the neighborhood of each vertex on the graph $\Gamma$. Last but not least, the equation \eqref{eq:SPDE} is well-posed in a weighted $L^2$-space, where the bilinear form associated with $\mathcal{L}$ is not symmetric, which further complicates the numerical analysis.

Using the asymptotic behavior of Hamiltonian flows under small perturbations, \cite{CF19} showed that $\mathcal{L}$ generates a strongly continuous semigroup on a weighted $L^2$-space $L^2_{\beta\gamma}(\Gamma)$ for a class of edge-independent weight functions $\gamma : \Gamma \to \mathbb{R}^+$. In Section \ref{S:WR}, we extend this result by operator semigroup theory to accommodate more general edge-dependent weight functions $\gamma$. We show that $\mathcal{L}$ not only generates a strongly continuous semigroup but also an analytic semigroup on $L^2_{\beta\gamma}(\Gamma)$ (see Theorem \ref{lem:operatorL}), a property essential for the numerical analysis of \eqref{eq:SPDE}. Moreover, we establish the well-posedness and regularity estimates of the mild solution to \eqref{eq:SPDE} in the weighted $L^2$-space for edge-dependent weights, thereby generalizing existing well-posedness results stated in \cite[Section 5]{CF19} and \cite[Section 2.4]{CS24}.

The non-compactness of the graph $\Gamma$ and the degeneracy of the operator $\mathcal{L}$ present significant challenges for the direct numerical treatment of \eqref{eq:SPDE}. To overcome these difficulties, we propose a three-step numerical strategy  in Section \ref{S:main} (see Fig.~\ref{Fig:map}). 
\begin{figure}[htbp]
    \centering
    \resizebox{0.7\textwidth}{!}{
        \begin{tikzpicture}[
            node distance=1.6cm, 
            block/.style={
                rectangle, 
                draw=blue!60!black, 
                thick, 
                fill=blue!5,        
                text width=11cm,    
                text centered, 
                minimum height=1cm, 
                rounded corners=3pt,
            },
            arrow_down/.style={-Stealth, thick, color=blue!70!black},
            arrow_up/.style={-Stealth, thick, color=red!70!black}, 
            description/.style={align=left, font=\footnotesize\sffamily, text width=5cm}
        ]
        \node [block] (original) {Degenerate stochastic PDE \eqref{eq:SPDE} on non-compact graph $\Gamma$};
        
        \node [block, below=of original, fill=blue!8] (step1) {Degenerate stochastic PDE \eqref{eq:local} on compact graph $\Gamma^R$};
        
        \node [block, below=of step1, fill=blue!11] (step2) {Non-degenerate stochastic PDE \eqref{eq:reguarlized} on compact graph $\Gamma^R$};
        
        \node [block, below=of step2, fill=blue!14] (step3) {Regularized truncated finite element approximation $u^{R,\delta}_h$ for \eqref{eq:SPDE}};

        \draw [arrow_down] ([xshift=2cm]original.south) -- node[right=0.2cm, description] {
            \textbf{Step 1}. Graph truncation  \\ \phantom{Step 2. }with parameter $R>1$
        } ([xshift=2cm]step1.north);

        \draw [arrow_up, dashed] ([xshift=-2cm]step1.north) -- node[midway, left=0.2cm, font=\small\bfseries\color{red!70!black}] {
            $R \to \infty$
        } ([xshift=-2cm]original.south);

        \draw [arrow_down] ([xshift=2cm]step1.south) -- node[right=0.2cm, description] {
            \textbf{Step 2}. Local regularization \\\phantom{Step 2. } with parameter $\delta>0$
        } ([xshift=2cm]step2.north);

        \draw [arrow_up, dashed] ([xshift=-2cm]step2.north) -- node[midway, left=0.2cm, font=\small\bfseries\color{red!70!black}] {
            $\delta\to0$
        } ([xshift=-2cm]step1.south);

        \draw [arrow_down] ([xshift=2cm]step2.south) -- node[right=0.2cm, description] {
           \textbf{Step 3}. Finite element discretiza-\\\phantom{Step 2. }tion with step size $h>0$
        } ([xshift=2cm]step3.north);

        \draw [arrow_up, dashed] ([xshift=-2cm]step3.north) -- node[midway, left=0.2cm, font=\small\bfseries\color{red!70!black}] {
          $h\to 0$
        } ([xshift=-2cm]step2.south);

        \end{tikzpicture}
    } 
    \caption{The three-step numerical strategy for approximating the solution to \eqref{eq:SPDE}.}
    \label{Fig:map}
\end{figure}
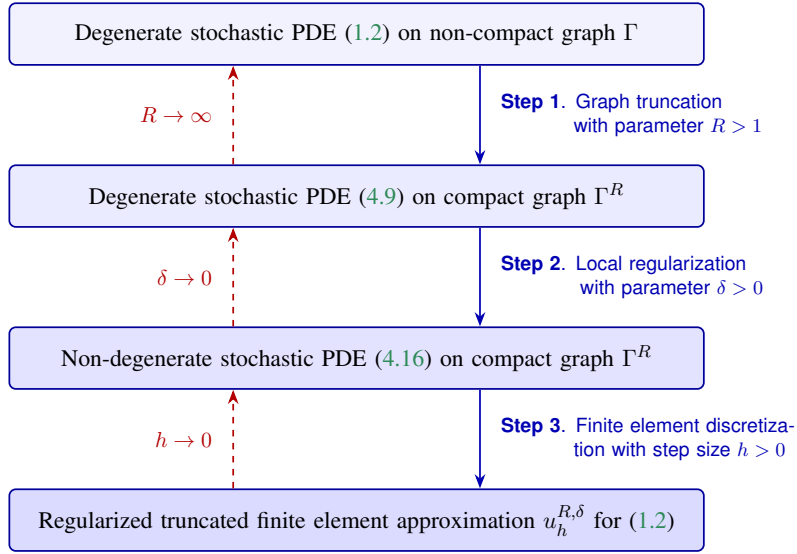
First, using a cut-off function $\eta_R:\Gamma\to\R^+$, we introduce a truncated problem (see \eqref{eq:uR}) in which the dynamics is frozen on the unbounded subgraph $\Gamma/\Gamma^R$, where $\Gamma^R = \textup{supp}(\eta_R)$ is a compact subgraph. Restricted to the compact graph $\Gamma^R$, the truncated problem is governed by a truncated operator $\check{\mathcal{L}}$ preserving the Kirchhoff-type gluing conditions at interior vertices (see \eqref{eq:checkL}). 
Second, to address the degeneracy of the coefficients $\alpha_k^R=\alpha_k\cdot\eta_R$ and $\beta_k^{-1}$ near the vertices of $\Gamma^R$, we introduce a localized regularization for the truncated operator $\check{\mathcal{L}}$, yielding a uniformly elliptic operator $\check{\mathcal{L}}^\delta$ on $\Gamma^R$ for each fixed regularization parameter $\delta>0$.
Third, for fixed $R$ and $\delta$, the regularized truncated problem (see \eqref{eq:reguarlized}) constitutes a non-degenerate stochastic parabolic PDE on the compact graph $\Gamma^R$, which is discretized in space using the finite element method with a mesh size $h>0$ (see \eqref{eq:udeltah}). This procedure yields a regularized truncated finite element approximation $u^{R,\delta}_h$ of the mild solution $u$ to \eqref{eq:SPDE}, and our main goal is to establish the following convergence result
\begin{equation}\label{eq:main}
	\lim_{R\to\infty}\lim_{\delta\to0}\lim_{h\to 0}\E\left[\|u(t)-u^{R,\delta}_h(t)\|^2_{L^2_{\beta\gamma}(\Gamma)}\right]=0,\quad\forall~ t\in(0,T].
\end{equation}

The proof of \eqref{eq:main} relies on Theorems \ref{tho:truncate}, \ref{tho:con-in-pro} and \ref{tho:FEM}, each of which uses the  properties of the generator $\mathcal{L}$ established in Theorem~\ref{lem:operatorL}.
In Theorem \ref{tho:truncate},
we derive an upper bound for the truncation error in terms of the diffusion coefficients of $\mathcal{L}$ and the graph weight function $\gamma$, whose proof is presented in Section \ref{S:Trun}. In particular, the truncation error tends to $0$ as the truncation parameter $R$ goes to infinity, provided that $\gamma$ decays  at least quadratically at infinity.

The most delicate part in proving \eqref{eq:main} is the convergence analysis of the regularization procedure (Theorem \ref{tho:con-in-pro}), which is carried out in  Section \ref{S:Reg}
by a tightness argument.
First, based on the specific constructions of the regularized coefficients
$
\alpha^{R,\delta}=\{\alpha_k^{R,\delta}\}_{k=1}^m$
and $
\beta^\delta=\{\beta_k^{\delta}\}_{k=1}^m,
$
we establish uniform energy estimates for the solution $\check{u}^\delta$ of the regularized truncated equation \eqref{eq:reguarlized} in a weighted Sobolev space
$
W^{1,2}_{\alpha^{R,\delta}\gamma,\beta^\delta\gamma}(\Gamma^R) 
$ 
independent of $\delta$ (see \eqref{eq:udeltaee}). A major difficulty arises from the non-symmetry of the regularized operator $\check{\mathcal{L}}^\delta$ on $L^2_{\beta^\delta\gamma}(\Gamma^R)$, which complicates control of the stochastic convolution.
To overcome this issue, we exploit the positivity and continuity of $\gamma$ on the compact graph $\Gamma^R$ and reduce the analysis to the case $\gamma \equiv 1$ where the fractional norm
$
\|(-\mathcal{L}^\delta)^{1/2}\cdot\|_{L^2_{\beta^\delta\gamma}(\Gamma^R)}
$
provide an equivalent norm of the energy space
$W^{1,2}_{\alpha^{R,\delta}\gamma,\beta^\delta\gamma}(\Gamma^R)$ (see \eqref{eq:self-adj} for more details). This enables 
us to control the stochastic convolution uniformly in $\delta$ (see Lemma \ref{lem:rHolder}), by making full use of the smoothing properties of the associated semigroup (see Lemma \ref{lem:smo}).
Using the integrability of the possible logarithmic singularity of $\beta_k$, we further refine the estimate to show that $\check{u}^\delta$ is uniformly bounded in
$
W^{1,2}_{\alpha^{R}\gamma,\beta\gamma}(\Gamma^R)
$
with respect to $\delta$. The tightness argument then reduces to proving  the compactness of the Sobolev embedding
$
W^{1,2}_{\alpha^R\gamma,\beta\gamma}(\Gamma^R)
\hookrightarrow
L^2_{\beta\gamma}(\Gamma^R),
$
which is nontrivial due to the potential degeneracy of $\alpha_k^R$ and the possible blow-up of $\beta_k$.
To control the decay rate of $\alpha_k^R=\alpha_k\cdot\eta_R$ near the boundary of $\mathrm{supp}(\eta_R)$, we choose a linearly decaying cut-off function $\eta_R$ rather than using a standard $\mathcal{C}_c^\infty$ cut-off function. This allows us to apply \cite[Claim 1 in Lemma 5.6]{CS25} and prove the tightness of the law of $\check{u}^\delta$ in
$
\mathcal{C}([0,T];L^2_{\beta\gamma}(\Gamma^R))
$ (see Lemma \ref{lem:tight}).
Hence, we can prove the convergence of the regularized truncated equation \eqref{eq:reguarlized},
 by combining tightness with the Gy\"ongy--Krylov characterization of convergence in probability (see \cite[Lemma 1.1]{GK96} for more details).

Section \ref{S:FEM} is devoted to the proof of Theorem \ref{tho:FEM} on the error of the finite element approximation of the regularized truncated equation \eqref{eq:reguarlized}. 
To handle the non-symmetry of the bilinear form associated with $\check{\mathcal{L}}^\delta$, we analyze the convergence of the discretization error in $L^2_{\beta^\delta\gamma}(\Gamma^R)$ using a resolvent-based approach, inspired by  \cite{JLZ13}.
Another key contribution of this part consists in interpolation error estimates in the energy norm (see  Lemma \ref{lem:int-L}) and in introducing a suitable shifting argument to treat the case where the operator $\check{\mathcal{L}}^\delta$ is merely bounded from above, rather than being negative semidefinite.
Combining an interpolation result (see \eqref{eq:L2beta}) with a priori estimates for the solutions, we establish the convergence rate of the finite element method in
$L^2_{\beta\gamma}(\Gamma^R)$. 

Finally, we summarize the main contributions of this work as follows. 
To numerically solve the degenerate parabolic stochastic PDE \eqref{eq:SPDE} on non-compact graph $\Gamma$, we propose a numerical strategy combining truncation, regularization and the finite element discretization, and establish the convergence of the resulting numerical solution. Although our analysis focuses on the model \eqref{eq:SPDE}, we believe that the numerical strategies and analytical techniques presented here can be potentially extended to handle more general non-compact graphs and degenerate operators.
Since the main focus of this work is the construction and convergence analysis of the spatial discretization of \eqref{eq:SPDE}, the development of fully discrete schemes and the corresponding numerical experiments are left for future work. For readers interested in the implementation of finite element discretizations for (stochastic) PDEs on graphs, we refer to \cite{AB18,BKKS24} for further details. In particular, the authors in \cite{AB18} proposed a representation of the stiffness matrix associated with the graph Laplacian based on the incidence matrix of the so-called extended graph, where spatial grid points are interpreted as additional vertices. This yields a structured algebraic formulation that facilitates the implementation of finite element discretization on graphs.

\section{Preliminaries}\label{S:Pr}
In this section, we provide a brief overview of the background and mathematical setting of the stochastic PDE on the graph $\Gamma$.
\subsection{Notation}
Let $\Gamma$ be a connected metric graph consisting of a finite set of vertices $\{O_i\}$ and a finite set of edges $\{I_k\}$, equipped with the shortest path distance.
Each edge $I_k$ is identified with an interval $[a_k,b_k]$ with $-\infty\le a_k<b_k\le \infty$, denoted by $I_k\cong [a_k,b_k]$. A location $(z,k)\in  \Gamma$ on the graph represents a position $z\in [a_k,b_k]$ on the edge $I_k$. Since $I_k\cong [a_k,b_k]$,  we sometimes write $(z,k)\in I_k$ (resp.~$(z,k)\in \mathring{I}_k$) to indicate  a location on $I_k$ with $z\in [a_k,b_k]$ (resp.~$z\in(a_k,b_k)$). A function $f$ defined on the graph $\Gamma$ is said to be continuous at a vertex $O_i$ if for all edges $\{I_k\}$ incident to $O_i$, the corresponding coordinate functions $\{f(\cdot,k)\}$ attain the same value at $O_i$.

Let $\mu=\{\mu_k\}_{k=1}^m$ and $\nu=\{\nu_k\}_{k=1}^m$ be two weight functions on $\Gamma$, that is,
$\mu_k,\nu_k : I_k \to [0,\infty)$ are measurable functions for each $k=1,2,\ldots,m$.
For any subinterval $J \subset I_k$, we define $L^2_{\nu_k}(J)$ as the space of complex-valued functions on $J$
equipped with the inner product
$$
\langle w,v\rangle_{L^2_{\nu_k}(J)}
:= \int_J w(z)\overline{v(z)}\nu_k(z)\ud z ,
\qquad w,v \in L^2_{\nu_k}(J),
$$
where $\overline{v(z)}$ denotes the complex conjugate of $v(z)$.
The weighted Sobolev space $W^{1,2}_{\mu_k,\nu_k}(J)$ consists of all functions
$w \in L^2_{\nu_k}(J)$ whose weak derivative
$w':=\frac{\ud}{\ud z}w\in L^2_{\mu_k}(J)$, endowed with the inner product
$$
\langle w,v\rangle_{W^{1,2}_{\mu_k,\nu_k}(J)}
:=
\langle w,v\rangle_{L^2_{\nu_k}(J)}
+
\langle w',v'\rangle_{L^2_{\mu_k}(J)},\quad w,v \in W^{1,2}_{\mu_k,\nu_k}(J).$$ 
We then define the graph-weighted spaces by the direct sums
$$
L^2_{\nu}(\Gamma)
:= \bigoplus_{k=1}^m L^2_{\nu_k}(I_k),
\qquad
W^{1,2}_{\mu,\nu}(\Gamma)
:= \bigoplus_{k=1}^m W^{1,2}_{\mu_k,\nu_k}(I_k).
$$
In the special case $\mu=\nu\equiv 1$, we omit the weights and write
$L^2(J)$ and $W^{1,2}(J)$ for the spaces on each edge, and
$L^2(\Gamma)$ and $W^{1,2}(\Gamma)$ for the corresponding function spaces on the graph $\Gamma$. 
Besides, we denote by $W^{1,2,c}_{\mu,\nu}(\Gamma)$ the subspace of $W^{1,2}_{\mu,\nu}(\Gamma)$ consisting of functions that are continuous at each interior vertex on the graph $\Gamma$.

Two functions $\varphi$ and $\psi$ are said to be asymptotically equivalent as $x \to x_0$ if
$\lim_{x \to x_0} \varphi(x)/\psi(x) = 1$, in which case we write $\varphi(x) \sim \psi(x)$ as $x \to x_0$.
 Given two Hilbert spaces $V_1$ and $V_2$, we denote by $\mathscr{L}(V_1)$ the space of bounded linear operators from $V_1$ to itself, equipped with the operator norm 
 \begin{align*}
 \|A_1\|_{\mathscr{L}(V_1)}:=\sup_{\|x\|_{V_1}\le 1}\|A_1x\|_{V_1},\quad A_1\in \mathscr{L}(V_1),
 \end{align*} and by $\mathscr{L}_2(V_1, V_2)$ the space of Hilbert--Schmidt operators from $V_1$ to $V_2$ endowed with the Hilbert--Schmidt norm
 \begin{align*}
 \|A_2\|_{\mathscr{L}_2(V_1, V_2)}:=\Big(\sum_{i=0}^\infty\|A_2f_i\|_{V_2}^2\Big)^{\frac12},\quad A_2\in\mathscr{L}_2(V_1, V_2),
 \end{align*}
 where $\{f_i\}_{i\in\mathbb{N}}$ is any complete orthonormal basis of $V_1$
  (see \cite{DZ14} for more details). 
  
\subsection{Mathematical setting}
Let $-H :\R^2\to \R$ be a stream function that describes the flow pattern of a fluid or gas in two-dimensional space, and let $\mathscr{W}$ a spatially homogeneous Wiener process on $\R^2$ with finite spectral measure $\mu$. For $\epsilon>0$,
consider a stochastic PDE 
\begin{align}\label{eq:SRDA}
    \partial_t u_\epsilon(t,x)=\mathcal{L}_\epsilon u_\epsilon(t,x)+b(u_\epsilon(t,x))+g(u_\epsilon(t,x))\partial_t \mathscr{W}(t,x),\quad t\in[0,T],
\end{align}
 subject to a suitable initial condition.
Let $T>0,$ $b,g:\R^2\to\R$ be Lipschitz continuous, 
assume that $ H :\R^2\to \R$ satisfies the following conditions (see {\cite[Hypothesis 1]{CF19}}).
\begin{enumerate}
\item[(i)] $ H \in \mathcal{C}^4(\R^2)$ with a bounded second derivative and
$\min _{x \in \mathbb{R}^2}  H (x)=0;
$
\item[(ii)] $ H $ has only a finite number of critical points $ \mathbf{x}_1, \ldots, \mathbf{x}_{n}$. 
The Hessian matrix $\nabla^2  H (\mathbf{x}_i)$ is non-degenerate for every $i=1,2, \ldots, {n}$, and $ H (\mathbf{x}_i) \neq H (\mathbf{x}_j)$ if $i \neq j$;
\item[(iii)] There exist $\mathfrak{a}_1, \mathfrak{a}_2, \mathfrak{a}_3>0$ such that
 for all $x \in \mathbb{R}^2$ with $|x|$ large enough,
\begin{equation*}
 H (x) \ge \mathfrak{a}_1|x|^2,\quad|\nabla  H (x)| \ge \mathfrak{a}_2|x|,\quad\Delta  H (x) \ge \mathfrak{a}_3.
\end{equation*}
\end{enumerate}

 The linear operator
$\mathcal{L}_\epsilon=\frac1\epsilon\Delta+\nabla^\perp H\cdot\nabla$ is the infinitesimal generator of the diffusion process
\begin{equation*}
    \ud X_\epsilon(t)=\frac{1}{\epsilon}\nabla^\perp  H (X_\epsilon(t))\ud t+\ud \mathrm{B}(t),\quad t\in[0,T],
\end{equation*}
where $\mathrm{B}$ is a two dimensional Brownian motion, and $\nabla^\perp:=(-\partial_{x_2},\partial_{x_1})$. The multiscale reaction-diffusion-advection equation \eqref{eq:SRDA} models the motion of particles over the time interval $[0, T/\epsilon]$, transported by an incompressible flow in $\mathbb{R}^2$ with the stream function $-H$, incorporating the effects of small viscosity, the flow pattern, and a slow chemical reaction affecting the particles.

 For every connected component $\textup{C}_k(z)$ of a given level set $\textup{C}(z):=\{x\in\R^2: H (x)=z\}$ of the function $ H $, there is an invariant measure $\ud\mu_{z,k}(x)\propto|\nabla  H (x)|^{-1}\ud x$ for the Hamiltonian system $
 \dot{X}(t)=\nabla^\perp  H (X(t)).$  By  identifying all points in $\mathbb{R}^2$ in the same connected component $\textup{C}_k(z)$, we obtain a graph $\Gamma$ consisting of a finite set of vertices $\{O_i\}_{i=1}^{n}$ and a finite set of edges $\{I_k\}_{k=1}^m$. 
 Each vertex $O_i$ corresponds to a critical point $\mathbf{x}_i$ of the function $ H $, where interior vertices correspond to saddle points of $ H $, and exterior vertices correspond to local extremum points of $ H $. Each interior and exterior vertex is incident to three edges and a single edge, respectively.
 We also include $O_{\infty}$ among exterior vertices, the endpoint of the only unbounded interval $I_m$ on the graph, corresponding to the point at infinity (see Fig.~\ref{fig:Hamilton} for an illustration).  
 \begin{figure}[!htb]
	\centering
	\includegraphics[width=0.7\linewidth]{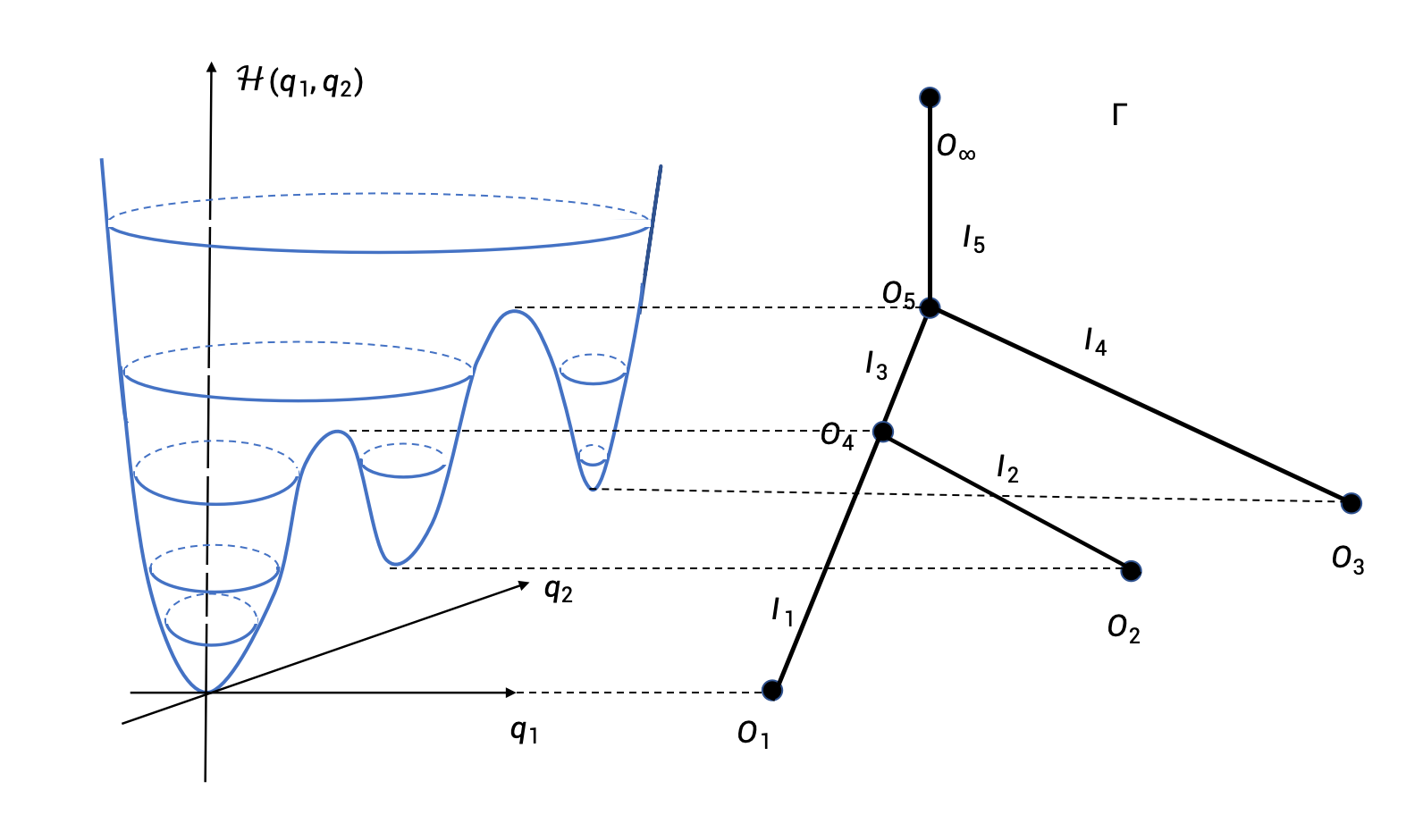}
	\caption{An example for the Hamiltonian $ H $  with three local minimizers ($O_1$, $O_2$, and $O_3$) and two saddle points ($O_4$ and $O_5$), and the graph $\Gamma$.}\label{fig:Hamilton}
\end{figure}

In what follows, we shall denote by $\Pi:\R^2\to\Gamma$ the identification map associating every $x\in\R^2$ to the corresponding point $\Pi(x)$ on the graph $\Gamma$. Then $\Pi(x)= ( H (x), k(x)),$ where $k(x) \in \{1,2,\ldots,m\}$ denotes the index of the edge $I_{k(x)}$ to which $\Pi(x)$ belongs.
Moreover, we have $I_m\cong[H_0,\infty]$, where $H_0:=\max_{1\le i\le {n}} H (\mathbf{x}_i)$; for $1\le k\le m-1$, $I_k\cong[H(\mathbf{x}_i),H(\mathbf{x}_j)]$, where $O_i$ and $O_j$ are the two vertices incident to $I_k$ ordered so that $H(\mathbf{x}_i)<H(\mathbf{x}_j)$. Endowed with the shortest-path distance, $\Gamma$ is a metric space.
By an averaging procedure with respect to the invariant measure $\mu_{z,k}$, it has been shown in \cite[Chapter 8]{FW12} that,
for any non-random initial value $X_\epsilon(0)\in\mathbb{R}^2$, the projected process
$\{\Pi(X_\epsilon(t))\}_{t\in[0,T]}$
converges in distribution to a Markov process
$\{\bar{Y}(t)\}_{t\in[0,T]}$ on the graph $\Gamma$.

For a function $\phi:\R^2\to \R$, we define its projection $\phi^\wedge:\Gamma\to\R$ by $(z,k)\mapsto\phi^\wedge(z,k)=\oint_{\textup{C}_k(z)}\phi(x)\ud \mu_{z,k}$.
  It has been shown in \cite{CF19} that as $\epsilon \to 0$, the asymptotic behavior of the projection $(u_\epsilon)^\wedge$ of the solution to  \eqref{eq:SRDA} is captured by the stochastic PDE \eqref{eq:SPDE}. The initial value $u(0)=u_0$ (resp. the driving process $W$) of \eqref{eq:SPDE} is the projection of the initial value (resp. the driving process $\mathscr{W}$) of \eqref{eq:SRDA}. 
The linear operator $\mathcal{L}$ in \eqref{eq:SPDE} is the infinitesimal generator of the limiting process $\bar{Y}$ given by 
\begin{equation}\label{eq:barL}
	\mathcal{L} f(z,k)
	= \frac{1}{2\beta_k(z)} \frac{\ud}{\ud z}
	\Bigl( \alpha_k(z) \frac{\ud f}{\ud z}(z,k) \Bigr),
	\qquad \text{if }(z,k)\in\mathring{I}_k,\quad k=1,\ldots,m.
\end{equation}
The coefficients in \eqref{eq:barL} are given by
\begin{equation}\label{eq:AT}
\alpha_k(z)=\oint_{\textup{C}_k(z)} |\nabla  H (x)| \ud l_{z, k},\quad \beta_k(z)=\oint_{\textup{C}_k(z)} \frac{1}{|\nabla  H (x)|}\ud l_{z,k},
\end{equation}
where $\ud l_{z,k}$ is the length element on $\textup{C}_k(z)$, and
$\beta_k(z)$ is the
 period of the motion on $\textup{C}_k(z)$.
The operator $\mathcal{L}$ is subject to gluing conditions, which require that 
$f$ is continuous at each interior vertex $O_i $ and satisfies the following Kirchhoff condition:
\begin{equation}\label{eq:glue}
	\sum_{k: I_k \sim O_i} 
	\alpha_k(H (\mathbf{x}_i))\ud_k f(H (\mathbf{x}_i),k) = 0 
	\quad \text{for each interior vertex } O_i,
\end{equation}
where $\ud_kf(H (\mathbf{x}_i),k):=\Phi_{k,i}\frac{\ud}{\ud z}f(H (\mathbf{x}_i),k)$ with
\begin{equation*}
\Phi_{k,i} =
\begin{cases}
+1, & \text{if } I_k  \text{ is incident to } O_i  \text{ and the $z$-coordinate decreases along $I_k$ towards $O_i$} ,\\[1mm]
-1, & \text{if } I_k  \text{ is incident to } O_i \text{ and the $z$-coordinate increases along $I_k$ towards $O_i$}.
\end{cases}
\end{equation*}
   The metric graph $\Gamma$ equipped with the operator $\mathcal{L}$ defined by \eqref{eq:barL} is also called a quantum graph \cite{AB18,BK13}. 

 \begin{remark}\label{rem:AT}
 	By \cite[Lemma 1.1]{FW12}, the functions $\beta_k$ and $\alpha_k$ are continuously differentiable in the interior $\mathring{I}_k$ of $I_k$. 
 	In contrast, $\alpha_k$ and $\beta_k$ possibly exhibit {singularities or degeneracy} near the vertices.
 	Indeed, if $(z,k)$ approaches an endpoint of an edge $I_k$, corresponding to a vertex $O_i=(H (\mathbf{x}_i),k)$, then by \cite[pp. 266-267]{FW12} and \cite[pp. 501]{CF19}, as $(z,k)\to O_i$, 
 	\begin{align}\label{eq:A}
 		&\alpha_k(z)\sim\begin{cases} \textup{const}\cdot|z-H (\mathbf{x}_i)|,\quad &\text{if $\mathbf{x}_i$ is a local extremum point},\\
 			\textup{const},\quad& \text{if $\mathbf{x}_i$ is a saddle point},\\
 			\textup{const}\cdot z,\quad &\text{if $O_i=O_\infty$};
 		\end{cases}\\\label{eq:T}
 		&\beta_k(z)\sim\begin{cases}\textup{const},&\text{if $\mathbf{x}_i$ is a local extremum point},\\
 			\textup{const}\cdot|\log|z- H (\mathbf{x}_i)||,& \text{if $\mathbf{x}_i$ is a saddle point},\\
 			\textup{const},& \text{if $O_i=O_\infty$}.
 		\end{cases}
 	\end{align}
 	Here \textup{const} denotes some positive constant depending on $k$ and $O_i$, and may differ from line to line. 
 \end{remark}

Recall that $\mathscr{W}$ a spatially homogeneous Wiener process on $\R^2$ with finite spectral measure $\mu$, i.e., 
\begin{align*}
\E\left[\mathscr{W}(t,x)\mathscr{W}(t,y)\right]=(t\wedge s)\Lambda(x-y),\quad t,s\in [0,T],~x,y\in \R^2,
\end{align*}
where $\Lambda$ is the Fourier transform of $\mu$ given by $\Lambda(x)=\int_{\R^2}e^{\mathbf{i} x\cdot\xi}\mu(\ud\xi) $ for $x\in \R^2$ with $\mathbf{i}$ being the imaginary unit. 
Let $\mathcal S(\R^2)$ be the set of smooth functions on $\R^2$ with rapid decrease.
The operator $\mathscr{Q}:\mathcal S(\R^2)\times\mathcal S(\R^2)\to\R$ defined by
$$\mathscr{Q}(\varphi_1,\varphi_2):=\int_{\R^2}\Lambda(x)\int_{\R^2}\varphi_1(y)\varphi_2(y-x)\ud y\ud x,\quad\varphi_1,\varphi_2\in\mathcal S(\R^2),$$
is called the covariance form of $\mathscr W$. 
By \cite[Section 1.2]{PZ97}, the reproducing kernel Hilbert space $(\mathcal{U},\|\cdot\|_{\mathcal{U}})$ of $\mathscr W$ can be identified with the dual of $\mathcal{S}_q $, where $\mathcal{S}_q $ is the completion of the set $\mathcal S(\R^2)/\ker \mathscr{Q}$ with respect to the norm $q([\varphi]):=\sqrt{\mathscr{Q}(\varphi,\varphi)}$.
Denote by $L^2_{(s)}(\R^2,\ud \mu)$ the subspace of the Hilbert space $L^2(\R^2,\ud \mu;\mathbb{C})$ consisting of all functions $\varphi$ satisfying $ \overline{\varphi(-x)}=\varphi(x)$ for $x \in\R^2$. Due to \cite[Proposition 1.2]{PZ97},
 an orthonormal basis for $\mathcal{U}$ is given by $\{\widehat{\mathfrak{u}_j\mu}\}_{j\in\mathbb{N}}$, where $\{\mathfrak{u}_j\}_{j\in\mathbb{N}}$ is a complete orthonormal basis of the Hilbert space $L^2_{(s)}(\R^2,\ud \mu)$ and
$\widehat{\mathfrak{u}_j\mu}(x)=\int_{\R^2}e^{\textup{i}x\cdot\xi} \mathfrak{u}_j(\xi)\mu(\ud\xi)$ for $x\in\R^2$.
This implies that $\mathscr W$ has the following Karhunen--Loève expansion
$\mathscr W(t,x)=\sum_{j\in\mathbb{N}}\widehat{\mathfrak{u}_j\mu}(x)\boldsymbol{\beta}_j(t)$ for $t \in[0,T]$ and $x\in\R^2$, where $\{\boldsymbol{\beta}_j\}_{j\in\mathbb{N}}$ is a sequence of independent Brownian motions defined on the stochastic basis $(\Omega,\mathscr F,\{\mathscr F_t\}_{t\in[0,T]},\mathbb P)$. By projection,  the driving process $W$ in \eqref{eq:SPDE}
is a Wiener process with the Karhunen--Loève expansion
$$
W(t,z,k)=\sum_{j\in\mathbb{N}}(\widehat{\mathfrak{u}_j\mu})^\wedge(z,k)\boldsymbol{\beta}_j(t),\quad (z,k)\in\Gamma.
$$
The sequence $\{(\widehat{\mathfrak{u}_j\mu})^\wedge\}_{j\in\mathbb{N}}$
forms a complete orthonormal basis of the Hilbert space
$\mathcal U_0:=\{\varphi^\wedge|\varphi\in\mathcal{U}\}$ endowed with the norm $\|\varphi^\wedge\|_{\mathcal{U}_0}:=\|\varphi\|_{\mathcal{U}}$. By the Parseval identity and the fact that $\{\mathfrak{u}_j\}_{j\in\mathbb{N}}$ forms a complete orthonormal basis of the Hilbert space $L^2_{(s)}(\mathbb{R}^2,\mathrm{d}\mu)$, and since $\mu$ is a finite measure, for any $x\in\R^2$ we have
\begin{equation}
\sum_{j\in\mathbb{N}} \left|\widehat{\mathfrak{u}_j \mu}(x)\right|^2
= \sum_{j\in\mathbb{N}} \left| \int_{\mathbb{R}^2} e^{\mathrm{i} x \cdot \xi} \mathfrak{u}_j(\xi)\mu(\mathrm{d}\xi) \right|^2
= \int_{\mathbb{R}^2} | e^{\mathrm{i} x \cdot \xi} |^2 \mu(\mathrm{d}\xi)
= \mu(\mathbb{R}^2) < \infty.
\end{equation}
Furthermore, by the Cauchy--Schwarz inequality, for any $(z,k)\in\Gamma$,
\begin{equation}\label{eq:e_i}
	\sum_{j\in\mathbb{N}} |(\widehat{\mathfrak{u}_j\mu})^\wedge(z,k)|^2 = \sum_{j\in\mathbb{N}} \Big|\oint_{\mathrm{C}_k(z)}\widehat{\mathfrak{u}_j\mu}(x)\ud\mu_{z,k}\Big|^2
	\le \sum_{j\in\mathbb{N}} \oint_{\mathrm{C}_k(z)}|\widehat{\mathfrak{u}_j\mu}(x)|^2\ud\mu_{z,k}\le \mu(\R^2).
\end{equation}

\section{Well-posedness and regularity of stochastic PDE on graph}\label{S:WR}
 Throughout the paper, we denote by $C$ a generic constant which may depend on several
parameters but never on the stepsize $h$ and may change from occurrence to occurrence. When
needed, we will explicitly write $C(a,b,\ldots)$ to emphasize the dependence of the constant $C$ upon
parameters $a,b,\ldots$ In particular, when $C$ depends on the truncation parameter $R$ or the regularization parameter $\delta$, we always explicitly write $C(R)$ or $C(\delta)$ to emphasize this dependence.

In this section, we establish the well-posedness of the stochastic PDE \eqref{eq:SPDE} on graph, under a weighted $L^2$-space $L^2_{\beta\gamma}(\Gamma)$, where $\gamma:\Gamma\to(0,\infty)$, $(z,k)\mapsto\gamma_k(z)$, is continuous on $\Gamma$ and continuously differentiable on the interior of each edge. 
This will rely on the properties of the operator $\mathcal{L}$ and its associated semigroup.

\subsection{Operator and semigroup}
 In this subsection, we show that the operator $\mathcal{L}$ generates an analytic semigroup 
 in the weighted $L^2$-space $L^2_{\beta\gamma}(\Gamma)$, under the following assumption. 
\begin{assumption}\label{asp:aT}
There exists a constant $\kappa>0$ such that 
\begin{equation}\label{eq:theta}
\alpha_k(z)|\gamma_k^\prime(z)|^2\le \kappa \beta_k(z)\gamma_k^2(z)\quad \text{for }(z,k)\in\mathring{I}_k,~ k=1,\ldots,m.
\end{equation}
\end{assumption}
The continuity and positivity of $\gamma$ implies the existence of a positive constant $C>0$ such that $C^{-1} \le \gamma_k \le C$ for all $k=1,2,\ldots,m-1$. 
In view of \eqref{eq:A} and \eqref{eq:T}, Assumption~\ref{asp:aT} will be satisfied provided that $\gamma$ satisfies the following condition
\begin{equation}\label{eq:gammader}
|\gamma_k'(z)| \le C,~ \text{if }(z,k)\in\mathring{I}_k,1\le k\le m-1;
~\text{and}~
\sqrt{z}|\gamma_m'(z)| \le C\gamma_m(z),~ \text{if }(z,m)\in\mathring{I}_m,
\end{equation}
where $C>0$ is a constant independent of $(z,k)$. Next we present some concrete examples of the graph weight function $\gamma$.

\begin{example}\label{ex:gammafunction}
For $(z,k)\in \Gamma$ with $k=1,\ldots,m-1$, set $\gamma_k(z)=1$.  
On the unbounded edge $I_m\cong[H_0,\infty)$, where $H_0=\max_{1\le i\le {n}} H (\mathbf{x}_i)$, 
we consider the following two choices:
\[
\text{(i)}\quad \gamma_m(z)=e^{-\rho_1 (z-H_0)^{\rho_2}}, 
\qquad
\text{(ii)}\quad \gamma_m(z)=(z-H_0+1)^{-\rho_3},
\]
where $\rho_1>0,0<\rho_2<\tfrac12$ and $\rho_3>1$.
One can readily verify that the inequality \eqref{eq:gammader} (and thus Assumption~\ref{asp:aT}) as well as $1_\Gamma\in L^2_{\beta\gamma}(\Gamma)$ hold in both cases.
\end{example}

In the special case where $\gamma_k(z) = \vartheta(z)$ is independent of $k$, and $\vartheta:[0,\infty)\to(0,\infty)$ is a bounded, twice continuously differentiable function satisfying
	\begin{equation*}
z|\vartheta''(z)| + |\vartheta'(z)| \le C\vartheta(z) \quad \forall~ z \ge 0, 
	\end{equation*}
	it has been shown in \cite{CF19, CS24} that the operator $\mathcal{L}$ generates a strongly continuous semigroup on $L^2_{\beta\gamma}(\Gamma)$.
Next, we extend this result to the more general setting of edge-dependent weight functions $\gamma$, and show that the corresponding operator generates not only a strongly continuous semigroup but also an analytic semigroup on $L^2_{\beta\gamma}(\Gamma)$. This property will be crucial for the numerical analysis of finite element methods.

\begin{theorem}\label{lem:operatorL}
Under Assumption \ref{asp:aT}, for any $\lambda\ge\frac18\kappa$, the operator $-\lambda I+\mathcal{L}$ generates an analytic contraction semigroup on $L^2_{\beta\gamma}(\Gamma)$, where the domain $D(\mathcal{L})$ of $\mathcal{L}$ is given by 
\begin{equation}\label{eq:DL}
D(\mathcal{L})=\bigl\{
f \in W^{1,2,c}_{\alpha\gamma,\beta\gamma}(\Gamma)
:\ f \text{ satisfies } \eqref{eq:glue},\ 
\alpha_k \tfrac{\ud f}{\ud z} \text{ is differentiable on each } I_k, \text{ and }\ 
\mathcal{L}f \in L^2_{\beta\gamma}(\Gamma)
\bigr\}. 
\end{equation}
\end{theorem}\begin{proof}
Recall that $W^{1,2,c}_{\alpha\gamma,\beta\gamma}(\Gamma)$ denotes the subspace of $W^{1,2}_{\alpha\gamma,\beta\gamma}(\Gamma)$ consisting of functions that are continuous at each interior vertex on the graph $\Gamma$.
First, we claim that $W^{1,2}_{\mathrm{cont}}(\Gamma):=W^{1,2,c}_{\alpha\gamma,\beta\gamma}(\Gamma)$ is a closed subset of $W^{1,2}_{\alpha\gamma,\beta\gamma}(\Gamma)$. To illustrate this,
suppose a sequence $\{f_n\}_{n=1}^\infty \subset W^{1,2}_{\mathrm{cont}}(\Gamma)$ 
converging to some $f$ in $W^{1,2}_{\alpha\gamma,\beta\gamma}(\Gamma)$,
and we will show that $f$ is continuous at each interior vertex. 
Let $O_i = (H (\mathbf{x}_i), k_1) = (H (\mathbf{x}_i), k_2) = (H (\mathbf{x}_i), k_3)$ be an interior vertex. 
Divide each incident edge $I_{k_j}$, $j=1,2,3$, into two equal subintervals, 
and denote by $\tilde{I}_{k_j}$ the subinterval of $I_{k_j}$ incident to $O_i$.
Since $f_n \to f$ in $W^{1,2}_{\alpha\gamma,\beta\gamma}(\Gamma)$, 
we have $f_n\mid_{\tilde{I}_{k_j}} \to f\mid_{\tilde{I}_{k_j}}$ in $W^{1,2}_{\alpha_{k_j}\gamma_{k_j},\beta_{k_j}\gamma_{k_j}}(\tilde{I}_{k_j})$ for each $j=1,2,3$.
By \eqref{eq:A} and \eqref{eq:T}, it can be verified that $W_{\alpha_{k_j}\gamma_{k_j},\beta_{k_j}\gamma_{k_j}}^{1,2}(\tilde{I}_{k_j})\subset W^{1,2}(\tilde{I}_{k_j})$. In combination with
the Sobolev embedding $W^{1,2}(\tilde{I}_{k_j})\hookrightarrow\mathcal{C}(\tilde{I}_{k_j})$,  this implies that $f_n\mid_{\tilde{I}_{k_j}} \to f\mid_{\tilde{I}_{k_j}}$ in $\mathcal{C}(\tilde{I}_{k_j})$ for each $j=1,2,3$. 
 Recalling that $f_n$'s are continuous at the interior vertex $O_i$, it follows that for $j=2,3$,
$$
f(H (\mathbf{x}_i),k_j)=\lim_{n\to\infty}f_n(H (\mathbf{x}_i),k_j)=\lim_{n\to\infty}f_n(H (\mathbf{x}_i),k_1)=f(H (\mathbf{x}_i),k_1).
$$
This proves the continuity of $f$ at the interior vertex $O_i$, and thus 
$W^{1,2}_{\mathrm{cont}}(\Gamma)$ is a Hilbert space endowed with the inner product inherited from $W^{1,2}_{\alpha\gamma,\beta\gamma}(\Gamma)$. 

For $\lambda\in\R$, consider
 the bilinear form $a_{\lambda}: W^{1,2}_{\mathrm{cont}}(\Gamma)\times W^{1,2}_{\mathrm{cont}}(\Gamma)\to \C$ defined by 
\begin{equation}\label{eq:ap}
a_{\lambda}(\phi,\psi)=\lambda \sum_{k=1}^m\int_{I_k}\phi(z,k)\overline{\psi(z,k)}\beta_k(z)\gamma_k(z)\ud z+\frac12\sum_{k=1}^m\int_{I_k}\alpha_{k}(z)\frac{\ud }{\ud z}\phi(z,k)\frac{\ud }{\ud z}\left(\overline{\psi(z,k)}\gamma_k(z)\right)\ud z.
\end{equation}
\emph{Claim 1. For $\lambda\ge\frac18 \kappa$, $a_\lambda$ is densely defined, continuous, closed and accretive.}\\
\emph{Proof of Claim 1.}
(i) $a_\lambda$ is densely defined since $ \oplus_{k=1}^m\mathcal{C}_c^\infty(I_k)\subset W^{1,2}_{\mathrm{cont}}(\Gamma)$ is dense in $L_{\beta\gamma}^2(\Gamma)$. Here, $\mathcal{C}_c^\infty(I_k)$ denotes the space of all infinitely differentiable functions with compact support in $I_k$.

(ii) By Assumption \ref{asp:aT} and  Young's inequality, for any $\epsilon>0$ and $\phi,\psi\in W^{1,2}_{\mathrm{cont}}(\Gamma)$, we have 
\begin{align}\label{eq:lowphipsi}\notag
   & \bigg|\sum_{k=1}^m\int_{I_k}\alpha_{k}\frac{\ud }{\ud z}\phi\overline{\psi}\gamma_k^\prime\ud z\bigg|\le \frac{1}{2}\epsilon\sum_{k=1}^m\int_{I_k}\alpha_{k}|\frac{\ud }{\ud z}\phi|^2\gamma_k\ud z+\frac{1}{2\epsilon}\sum_{k=1}^m\int_{I_k}\alpha_{k}|\overline{\psi}\gamma_k^\prime|^2\gamma_k^{-1}\ud z\\
  &\le \frac{1}{2}\epsilon\sum_{k=1}^m\int_{I_k}\alpha_{k}|\frac{\ud }{\ud z}\phi|^2\gamma_k\ud z+\frac{\kappa}{2\epsilon}\sum_{k=1}^m\int_{I_k}|\psi|^2\beta_k\gamma_k\ud z.
\end{align}
 In view of \eqref{eq:lowphipsi} with $\psi=\phi$, one has that for any $\phi\in W^{1,2}_{\mathrm{cont}}(\Gamma)$,
\begin{align}\label{eq:lower}
\Re a_\lambda(\phi,\phi)
&\ge (\frac12-\frac{1}{4}\epsilon)\sum_{k=1}^m\int_{I_k}\alpha_{k}|\frac{\ud}{\ud z}\phi|^2\gamma_k\ud z+(\lambda-\frac{\kappa}{4\epsilon} )\sum_{k=1}^m\int_{I_k}|\phi|^2\beta_k\gamma_k\ud z.
\end{align}
For $\lambda\ge\frac{1}{8}\kappa$, we take $\epsilon\in[\frac{\kappa}{4\lambda},2]$ such that
 $\Re a_\lambda(\phi,\phi)\ge0$ for all $\phi\in W^{1,2}_{\mathrm{cont}}(\Gamma)$, i.e., $a_\lambda$  is accretive.

(iii) By the H\"older inequality and \eqref{eq:lowphipsi}, one can show that
\begin{equation}\label{eq:aconti}
  |a_{\lambda}(\phi,\psi)|\le C(\lambda,\kappa)\|\phi\|_{W^{1,2}_{\alpha\gamma,\beta\gamma}(\Gamma)}\|\psi\|_{W^{1,2}_{\alpha\gamma,\beta\gamma}(\Gamma)},\quad \phi,\psi\in W^{1,2}_{\mathrm{cont}}(\Gamma).  
\end{equation}
Taking $\epsilon=\frac{\kappa}{4\lambda+2}$ in \eqref{eq:lower} yields  $\|\phi\|_{a_\lambda}^2:=\Re a_\lambda(\phi,\phi)+\|\phi\|^2_{L_{\beta\gamma}^2(\Gamma)}\ge \frac{1}{4\lambda+2}\|\phi\|_{W^{1,2}_{\alpha\gamma,\beta\gamma}(\Gamma)}^2$ for all $\phi\in  W^{1,2}_{\mathrm{cont}}(\Gamma)$. This proves the continuity of $a_{\lambda}$, i.e., 
\begin{equation}\label{eq:cont}
|a_{\lambda}(\phi,\psi)|\le  \mathcal{M}\|\phi\|_{a_\lambda}\|\psi\|_{a_\lambda},\quad \phi,\psi\in W^{1,2}_{\mathrm{cont}}(\Gamma).  
\end{equation}
where  $\mathcal M$ is a positive constant depending on $\kappa$ and $\lambda$.

(iv) $a_\lambda$ is closed since for any $\lambda\ge\frac{1}{8}\kappa$, $\|\cdot\|_{a_\lambda}$
is an equivalent norm of the Hilbert space $W^{1,2}_{\mathrm{cont}}(\Gamma)$.

The operator associated with the form $a_\lambda$ is defined by (see \cite[Definition 1.21]{OE05})
\begin{align}\label{eq:DA}D(A)&:=\{ \phi\in W^{1,2}_{\mathrm{cont}}(\Gamma):\exists f\in L^2_{\beta\gamma}(\Gamma) ~s.t.~ a_\lambda(\phi, \psi)=\langle f, \psi\rangle_{L^2_{\beta\gamma}(\Gamma)}\quad \forall~\psi\in W^{1,2}_{\mathrm{cont}}(\Gamma)\},\\
A\phi&:=-f.
\end{align}

\noindent\emph{Claim 2. The operator $A$ associated with the form $a_\lambda$ is $-\lambda I + \mathcal{L}$, where $D(\mathcal{L})$ is defined in \eqref{eq:DL}.}
\emph{Proof of Claim 2}. (I) 
Denote by $W^{1,2}_{c,\mathrm{cont}}(\Gamma)$ the space of functions in $W^{1,2}_{\mathrm{cont}}(\Gamma)$ with compact support, namely, for any $\psi\in W^{1,2}_{c,\mathrm{cont}}(\Gamma)$, there exists a constant $z^{\psi}\ge H _0$ such that $\psi(z,m)=0$ for all $z\ge z^\psi$. 
If $\phi\in D(\mathcal{L})$, then integrating by parts yields that for any $\psi\in W^{1,2}_{c,\mathrm{cont}}(\Gamma)$, 
\begin{align*}
a_\lambda(\phi,\psi)
&=\lambda\sum_{k=1}^m\int_{I_k}\phi(z,k)\overline{\psi(z,k)}\beta_k(z)\gamma_k(z)\ud z-\sum_{k=1}^m\int_{I_k}\mathcal{L}\phi(z,k)\overline{\psi(z,k)}\beta_k(z)\gamma_k(z)\ud z\\
&\quad+\frac12\sum_{k=1}^{m-1}\sum_{O_i\sim I_k}\alpha_{k}(z)\frac{\ud }{\ud z}\phi(z,k)\overline{\psi(z,k)}\gamma_k(z)\Phi_{k,i}+\frac12\alpha_{m}(z)\frac{\ud }{\ud z}\phi(z,m)\overline{\psi(z,m)}\gamma_m(z)\Big|^{z=z^\psi}_{z=H_0}\\
&=\left\langle (\lambda I-\mathcal{L})\phi,\psi\right\rangle_{L^2_{\beta\gamma}(\Gamma)}\\
&\quad+\frac12\sum_{O_i\in\mathcal{O}_{\textup{int}}}\sum_{k:I_k\sim O_i}\alpha_{k}(H (\mathbf{x}_i))\ud_k\phi(H (\mathbf{x}_i),k)\overline{\psi(H (\mathbf{x}_i),k)}\gamma_k(H (\mathbf{x}_i)),
\end{align*}
where  we have used the fact that $\alpha_k$ vanishes at exterior vertices and $\psi(z^\psi,m)=0$. Throughout, $\mathcal{O}_{\textup{int}}$ denotes the set of interior vertices of $\Gamma$.
Furthermore,
 the continuity of $\overline{\psi}\cdot\gamma$ 
 at interior vertices and the Kirchhoff condition \eqref{eq:glue} lead to 
 \begin{align}\label{eq:aphipsi}
a_\lambda(\phi,\psi)=\langle (\lambda I-\mathcal{L})\phi,\psi\rangle_{L^2_{\beta\gamma}(\Gamma)}\quad\forall~\psi\in W^{1,2}_{c,\mathrm{cont}}(\Gamma).
 \end{align}

We next show that $W^{1,2}_{c,\mathrm{cont}}(\Gamma)$ is dense in $W^{1,2}_{\mathrm{cont}}(\Gamma)$. 
For each $R>0$, set $\varrho(\cdot)=\varrho(\frac{\cdot}{R})$, where $\varrho:[0,\infty)\to [0,1]$ is a smooth function satisfying 
\begin{equation}\label{eq:varphi}
\varrho(z)=1, \quad\text{if}~ 0\le z\le1;\qquad
\varrho(z)=0,\quad\text{if}~ z\ge2.
\end{equation}
Given $\psi\in W^{1,2}_{\mathrm{cont}}(\Gamma)$, we have $\psi\cdot\varrho_R\in W^{1,2}_{c,\mathrm{cont}}(\Gamma)$ and that for any $R> H _0$,
\begin{align*}
\|\psi-\psi\cdot\varrho_R\|^2_{W^{1,2}_{\alpha\gamma,\beta\gamma}(\Gamma)}&=\int_{I_m}|\frac{\ud}{\ud z}(\psi-\psi\varrho_R)|^2\alpha_{m}\gamma_m\ud z+\int_{I_m}|\psi-\psi\varrho_R|^2\beta_m\gamma_m\ud z\\
&\le \int_{R}^\infty|\frac{\ud}{\ud z}\psi|^2\alpha_{m}\gamma_m\ud z+\frac{C}{R^2}\int_{R}^{2R}|\psi|^2\alpha_{m}\gamma_m\ud z+\int_{R}^\infty|\psi|^2\beta_{m}\gamma_m\ud z.
\end{align*}
Since  $\textup{supp}(\varrho_R^\prime)\subset[R,2R]$, $\alpha_m(z)\sim  z$ and $\beta_m(z)\sim C>0$ as $z\to \infty$, for sufficiently large $R$, we have
$$\frac{C}{R^2}\int_{R}^{2R}|\psi|^2\alpha_{m}\gamma_m\ud z
\le C\frac{1}{R}\int_{R}^{2R}|\psi|^2\beta_{m}\gamma_m\ud z\le 
 C\frac{1}{R}\int_{R}^{\infty}|\psi|^2\beta_{m}\gamma_m\ud z.$$
Hence $\psi\cdot\varrho_R$ converges to $\psi$  in $W^{1,2}_{\mathrm{cont}}(\Gamma)$,
which means that $W^{1,2}_{c,\mathrm{cont}}(\Gamma)$ is dense in $W^{1,2}_{\mathrm{cont}}(\Gamma)$.

For any $\psi\in  W^{1,2}_{\mathrm{cont}}(\Gamma)$, there is a sequence $\{\psi_n\}_{n=1}^\infty\subset W^{1,2}_{c,\mathrm{cont}}(\Gamma)$ converging in $W^{1,2}_{\mathrm{cont}}(\Gamma)$ to $\psi$, which, along with the continuity of $a_\lambda$ and \eqref{eq:aphipsi}, indicates 
$$a_\lambda(\phi,\psi)=\lim_{n\to\infty}a_\lambda(\phi,\psi_n)=\lim_{n\to\infty}\langle (\lambda I-\mathcal{L})\phi,\psi_n\rangle_{L^2_{\beta\gamma}(\Gamma)}=\langle (\lambda I-\mathcal{L})\phi,\psi\rangle_{L^2_{\beta\gamma}(\Gamma)}.$$
Thus we have proved that $\phi\in D(A)$ and $A \phi=(-\lambda I+\mathcal{L})\phi$.

(II) On the other hand, if $\phi\in D(A)$, then by \eqref{eq:DA} and the definition of weak derivative, we obtain that for each $k=1,\ldots,m$,
\begin{align*}
\frac12\frac{\ud}{\ud z}\left(\alpha_{k}(z)\frac{\ud }{\ud z}\phi(z,k)\right)=(\lambda\phi(z,k)-f(z,k))\beta_k(z)\quad \text{in weak sense},
\end{align*}
i.e., $(\lambda I-\mathcal{L})\phi=f\in L^2_{\beta\gamma}(\Gamma)$.
 Then integrating by parts implies that for any $\psi\in W^{1,2}_{c,\mathrm{cont}}(\Gamma)$,
\begin{align*}
a_\lambda(\phi,\psi)&=\langle f,\psi\rangle_{L^2_{\beta\gamma}(\Gamma)}+\frac12\sum_{O_i\in\mathcal{O}_{\textup{int}}}\Big(\sum_{k:I_k\sim O_i}\alpha_{k}(H (\mathbf{x}_i))\ud_k\phi(H (\mathbf{x}_i),k))\Big)\overline{\psi(H (\mathbf{x}_i),k)}\gamma_k(H (\mathbf{x}_i)),
\end{align*}
due to $(\lambda I-\mathcal{L})\phi=f$, the fact that $\alpha_k$ vanishes at exterior vertices, $\psi(z^\psi,m)=0$ and the continuity of $\gamma$ at interior vertices.
The arbitrariness of $\psi$ in $W^{1,2}_{c,\mathrm{cont}}(\Gamma)$ implies that $\phi$ satisfies the Kirchhoff condition \eqref{eq:glue}, and hence $\phi \in D(\mathcal{L})$.

Combining Claim 1 and Claim 2,
according to \cite[Theorem 1.51 \& Theorem 1.52]{OE05}, 
 for any $\lambda\ge\frac18\kappa$, the linear operator $\mathcal{L}-\lambda I$  generates a strongly continuous contraction semigroup on $L^2_{\beta\gamma}(\Gamma)$ which is analytic on the sectorial $\{y\in\C,y\neq0, |\arg y|<\frac{\pi}{2}-\arctan \mathcal M\}$, where $\mathcal M$ is the constant in the continuity property \eqref{eq:cont} of $a_\lambda$.
 \end{proof}
\subsection{Well-posedness and regularity of the solution}
Let $B$ and $G$ be the Nemytskii operators associated with $b$ and $g$, respectively, i.e., for $f\in L^2_{\beta\gamma}(\Gamma)$ and $\psi\in \mathcal{U}_0$, 
\begin{align}\label{eq:BG}
	B(f)(z,k):=b(f(z,k)),\quad G(f)(\psi)(z,k):=g(f(z,k))\psi(z,k),\quad (z,k)\in\Gamma.
\end{align}
By \eqref{eq:e_i} and the Lipschitz continuity of $b$ and $g$, the mappings $B:L^2_{\beta\gamma}(\Gamma)\to L^2_{\beta\gamma}(\Gamma)$ and $G:L^2_{\beta\gamma}(\Gamma)\to \mathscr{L}_2(\mathcal{U}_0, L^2_{\beta\gamma}(\Gamma))$ are Lipschitz continuous, i.e., for any $f_1,f_2\in L^2_{\beta\gamma}(\Gamma)$,
\begin{gather}\label{eq:BLip}
     \|B(f_1)-B(f_2)\|_{L^2_{\beta\gamma}(\Gamma)}+ \|G(f_1)-G(f_2)\|_{\mathscr{L}_2(\mathcal{U}_0,L^2_{\beta\gamma}(\Gamma))}\le  C\|f_1-f_2\|_{L^2_{\beta\gamma}(\Gamma)}.
\end{gather}
Notice that $L^2_{\beta\gamma}(\Gamma)=L^2(\Gamma;\nu_\gamma)$, where $\nu_\gamma$ is a measure on the measurable space $(\Gamma,\mathscr{B}(\Gamma))$ defined by $\nu_{\gamma}(E):=\sum_{k=1}^m \int_{I_k\cap E} \gamma_k(z) \beta_k(z) \ud z$ for $E\in\mathscr{B}(\Gamma)$.
Denote by $1_\Gamma:\Gamma\to \R$ the constant mapping $(z,k)\mapsto 1_\Gamma(z,k)=1$. When $1_{\Gamma}\in L^2_{\beta\gamma}(\Gamma)$,  $\nu_{\gamma}$ is a finite measure on $(\Gamma,\mathscr{B}(\Gamma))$.
 By \eqref{eq:BLip}, if 
	$1_{\Gamma}\in L^2_{\beta\gamma}(\Gamma)$, 
	the mappings $B$ and $G$ are of linear growth, i.e., for any $ f_1\in L^2_{\beta\gamma}(\Gamma)$,
    \begin{align*}
         \|B(f_1)\|_{L^2_{\beta\gamma}(\Gamma)}+ \|G(f_1)\|_{\mathscr{L}_2(\mathcal{U}_0,L^2_{\beta\gamma}(\Gamma))}\le  C(1+\|f_1\|_{L^2_{\beta\gamma}(\Gamma)}).
    \end{align*}
 
Due to Theorem \ref{lem:operatorL}, $\mathcal L$ generates a strongly continuous analytic semigroup 
	$\{e^{t \mathcal{L}},t>0\}$ on $L^2_{\beta\gamma}(\Gamma)$, and 
	$\|e^{t \mathcal{L}}\|_{\mathscr{L}(L^2_{\beta\gamma}(\Gamma))}\le e^{\frac{1}{8}\kappa t}$ for all $t>0.$
Invoking the Lipschitz continuity and linear growth of both $B$ and $G$, we can deduce the following well-posedness result by applying \cite[Theorem 7.5]{DZ14}.
\begin{lemma}\label{lem:L2}
Let Assumption \ref{asp:aT} hold and $\{1_{\Gamma},u_0\}\subset L^2_{\beta\gamma}(\Gamma)$. Then there exists a unique mild solution $u=\{u(t)\}_{t\in[0,T]}$ to \eqref{eq:SPDE}, i.e., for any $t\in[0,T]$,
$$u(t)=e^{t\mathcal{L}} u_0+\int_0^t e^{(t-s)\mathcal{L}}B(u(s))\ud s
+\int_0^t e^{(t-s)\mathcal{L}} G(u(s))\ud W(s).$$
	 Moreover, for any $p\ge 1$, there exists $C:=C(p,\kappa,T)>0$ such that
\begin{align}\label{eq:L2}
	\E\bigg[\sup_{t\in[0,T]}\|u(t)\|^p_{L^2_{\beta\gamma}(\Gamma)}\bigg]\le C\Big(1+\|u_0\|^p_{L^2_{\beta\gamma}(\Gamma)}\Big).
\end{align}
\end{lemma}

\begin{lemma}\label{lem:H1}
Under the assumptions of Lemma \ref{lem:L2}, there exists $C:=C(T,\kappa,p)>0$ such that
\begin{align}\label{eq:utgamma}
\E\int_0^T\sum_{k=1}^m\int_{I_k}|\frac{\partial}{\partial z}u(t,z,k)|^2\alpha_k(z)\gamma_k(z)\ud z \ud t\le C\Big(1+\|u_0\|_{L^2_{\beta\gamma}(\Gamma)}^2\Big).
\end{align}
\end{lemma}
\begin{proof}
For any $\phi \in W_{\alpha\gamma,\beta\gamma}^{1,2}(\Gamma)$, we define the operator $\mathcal{L}$ by $
(\mathcal{L}\phi)(\psi) := - a_0(\phi,\psi)$ for all $\psi \in W_{\alpha\gamma,\beta\gamma}^{1,2}(\Gamma),
$
where $a_0$ is the bilinear form defined in \eqref{eq:ap} with $\lambda = 0$. 
Then $\mathcal{L}\phi$ defines a continuous linear functional on 
$W_{\alpha\gamma,\beta\gamma}^{1,2}(\Gamma)$, that is,
$
\mathcal{L}\phi \in \bigl(W_{\alpha\gamma,\beta\gamma}^{1,2}(\Gamma)\bigr)^*.
$
Applying the It\^o formula (see \cite[Theorem 4.2.5]{LR15}) to $\|u(t)\|_{L^2_{\beta\gamma}}^2$, we obtain that for any $t\in(0,T]$,
\begin{align}\label{eq:Psi0R}\notag
\ud \|u(t)\|^2_{L^2_{\beta\gamma}(\Gamma)}&=-2a_0(u(t),u(t))\ud t+2\langle u(t),B(u(t))\rangle_{L^2_{\beta\gamma}(\Gamma)}\ud t\\
&\quad+2\langle u(t),G(u(t))\ud W(t)\rangle_{L^2_{\beta\gamma}(\Gamma)}+\|G(u(t))\|_{\mathscr{L}_2(\mathcal{U}_0,L^2_{\beta\gamma}(\Gamma))}^2\ud t.
\end{align}
The inequality \eqref{eq:lower} with $\epsilon=1$ reads
\begin{align}\label{eq:Lut}
-a_0(u(t),u(t))
\le -\frac14\sum_{k=1}^m\int_{I_k}\alpha_{k}|\frac{\partial }{\partial z}u(t)|^2\gamma_k\ud z+
\frac{\kappa}{4}\sum_{k=1}^m\int_{I_k}|u(t)|^2\beta_k\gamma_k\ud z.
\end{align}
The linear growth of $B:L^2_{\beta\gamma}(\Gamma)\to L^2_{\beta\gamma}(\Gamma)$ and $G:L^2_{\beta\gamma}(\Gamma)\to\mathscr{L}_2(\mathcal{U}_0,L^2_{\beta\gamma}(\Gamma))$ leads to 
\begin{align}\label{eq:BG1}
\langle u(t),B(u(t))\rangle_{L^2_{\beta\gamma}(\Gamma)}+\|G(u(t))\|_{\mathscr{L}_2(\mathcal{U}_0,L^2_{\beta\gamma}(\Gamma))}^2
\le C(1+\|u(t)\|^2_{L^2_{\beta\gamma}(\Gamma)}).
\end{align}
 Substituting the inequalities \eqref{eq:Lut} and \eqref{eq:BG1} into \eqref{eq:Psi0R}, and then integrating over $[0,T]\times\Omega$, we obtain \eqref{eq:utgamma} from \eqref{eq:L2}.
 \end{proof}

 \begin{remark}
The coefficients $b$ and $g$ in \eqref{eq:SPDE} originate from the reaction terms in the multiscale model \eqref{eq:SRDA}. 
Regarding the well-posedness of stochastic PDEs on graphs, it is natural to include more general nonlinearities that may depend on the edges; see, e.g., \cite{KS21}. 
Specifically, we can consider the following generalized form of \eqref{eq:SPDE}:
\begin{equation}\label{eq:SPDE general}
	\partial_t u(t,z,k)
	= \mathcal{L} u(t,z,k)
	+ b_k(u(t,z,k))
	+ g_k(u(t,z,k))\,\partial_t W(t,z,k),
\end{equation}
for $(t,z,k)\in (0,T]\times \Gamma$, where $b_k, g_k : \mathbb{R} \to \mathbb{R}$ are Lipschitz continuous for $k=1,2,\ldots,m$.  
Accordingly, the associated Nemytskii operators $B$ and $G$ are defined by
\begin{align*}
	B(f)(z,k) := b_k(f(z,k)), \qquad
	G(f)(\psi)(z,k) := g_k(f(z,k))\,\psi(z,k), \quad (z,k)\in \Gamma,
\end{align*}
where $f\in L^2_{\beta\gamma}(\Gamma)$ and $\psi\in \mathcal{U}_0$. Since the mappings $B$ and $G$ defined above remain Lipschitz continuous, Lemma~\ref{lem:L2} continues to hold for \eqref{eq:SPDE general}.
 \end{remark}

\section{Approximation strategy and convergence result}\label{S:main}
In this section, we describe our approach for the numerical approximation of the degenerate stochastic PDE \eqref{eq:SPDE} on the unbounded graph $\Gamma$. The non-compactness of the graph $\Gamma$ and the degeneracy of the operator $\mathcal{L}$
pose significant challenges for direct numerical treatment. To address these challenges, we adopt a three-step strategy: first, we truncate the unbounded graph 
$\Gamma$ to obtain a compact subgraph; next, we apply a regularization technique to manage the degeneracy; and finally, we perform a numerical discretization of the resulting regularized and truncated problem (see Fig.~\ref{Fig:map}).  We further present the convergence of the resulting regularized truncated finite element approximation of \eqref{eq:SPDE}, which constitutes the main numerical result of this work.

\subsection{Truncated approximation} 
To solve \eqref{eq:SPDE} numerically, we first consider its truncated approximation.
For each $R>0$, let $\eta_R:\Gamma\to [0,1]$ be a continuous cut-off function such that 
\begin{equation*}
    \eta_R(z,k)=\begin{cases}
        1,\quad \text{if }0\le z\le R,\\
        0,\quad \text{if }z\ge R+1.
    \end{cases}
\end{equation*}
Since $\eta_R$ does not depend on the edge index $k$, we write $\eta_R(z)$ in place of $\eta_R(z,k)$ for simplicity. 
In what follows, suppose that $R> H _0$ so that $1-\eta_R$ vanishes identically on $\cup_{k=1}^{m-1}I_k$.
We propose the following truncated problem
\begin{equation}\label{eq:uR}
\partial_t u^R(t,z,k)=\mathcal{L}^Ru^R(t,z,k)+b(u^R(t,z,k))\eta_R(z)+g(u^R(t,z,k))\eta_R(z)\partial_t W(t,z,k)
\end{equation}
subject to the initial value $u^R(0)=u_0$, where 
\begin{equation}
\mathcal{L}^R f(z, k)=\frac{1}{ 2\beta_k(z)} \frac{\ud}{\ud z}\left(\alpha_k^R(z) \frac{\ud f}{\ud z}(z,k)\right), \quad\text{ if } (z,k)\in\mathring{I}_k,\quad k=1,\ldots,m
\end{equation}
with $\alpha_k^R:=\alpha_k\cdot \eta_R$. The operator $\mathcal{L}^R$ is subject to the same gluing conditions as $\mathcal{L}$ (see \eqref{eq:glue}). 
The well-posedness of the truncated equation \eqref{eq:uR} will be shown in Lemma \ref{lem:wp-trun}. 
We now state a theorem estimating the truncation error of \eqref{eq:uR} for approximating \eqref{eq:SPDE}, which provides an upper bound for the truncation error in terms of $\alpha_k$ and $\gamma$.
\begin{theorem}\label{tho:truncate}
Let  Assumption \ref{asp:aT} hold, and $\{1_\Gamma,u_0\}\subset L_{\beta\gamma}^2(\Gamma)\cap L_{\beta\sqrt\gamma}^2(\Gamma)$. Then  there exists a constant $C:=C(T,u_0)>0$ such that for any $t\in[0,T]$ and $R\ge H _0+1$,
\begin{align*}
\E\left[\|u(t)-u^R(t)\|_{L^2_{\beta\gamma}(\Gamma)}^2\right]
\le  C\sup_{z\ge R-1}\left\{ \alpha_{m}(z)\sqrt{\gamma_m(z)}\right\}.
\end{align*}
\end{theorem}
We postpone the proof of
Theorem \ref{tho:truncate} to Section \ref{S:Trun}. Theorem \ref{tho:truncate}
 implies the convergence of the truncated problem \eqref{eq:uR} to \eqref{eq:SPDE}, provided that 
\begin{align}\label{eq:convercond}
 1_\Gamma\subset L_{\beta\gamma}^2(\Gamma)\cap L_{\beta\sqrt\gamma}^2(\Gamma),\quad  \text{and}\quad\lim_{R\to\infty}\sup_{z\ge R}z\sqrt{\gamma_m(z)}=0.
\end{align}
The above condition \eqref{eq:convercond} holds for the graph weight function $\gamma$ in both cases (i) and (ii) of Example \ref{ex:gammafunction}, provided that $\rho_3>2$. 
It would be interesting to further study the convergence of the truncation procedure for polynomially decaying weight functions where the decay is at most quadratic (i.e., case (ii) of Example \ref{ex:gammafunction} with $\rho_3\in(1,2]$). We also note that the convergence rate obtained in Theorem~\ref{tho:truncate} may not be optimal. 

Since $\textup{supp}(\eta_R)\subset[0,R+1]$, according to \eqref{eq:uR}, we have $\partial_t u^{R}(t,z,m)=0$  for any $z\ge R+1$,
which means that $u^R(t,z,m)=u^R(0,z,m)=u_0(z,m)$ for all $t\in[0,T]$ and $z\ge R+1$.
It thus suffices to solve \eqref{eq:uR}  on the compact graph $$\Gamma^{R}:=\{(z,k)\in\Gamma: z\le R+1\}.$$ The truncated graph $\Gamma^R$ consists of $m$ edges $\{J_k\}_{k=1}^m$ and ${n}+1$ vertices $\{O_i\}_{i=1}^{{n}+1}$ with the additional vertex  $O_{{n}+1}:=(R+1,m)$.
For $k=1,\ldots,m-1$, the edges are given by $J_k=I_k$, while the last edge is 
$
J_m := \{(z,m)\in I_m : z \le R+1\}\cong [H_0,R+1].
$
Define a differential operator $\check{\mathcal{L}}$ on $\Gamma^R$ as
\begin{equation}\label{eq:checkL}
\check{\mathcal{L}}f(z,k)=\frac{1}{2\beta_k(z)}\frac{\ud}{\ud z}\Big(\alpha_k^R(z)\frac{\ud}{\ud z}f(z,k)\Big),\quad \text{if } (z,k)\in\mathring{J}_k,\quad k=1,\ldots,m
\end{equation}
 endowed with the same gluing conditions as $\mathcal{L}^R$.
According to \cite[Lemma~1.1]{FW12}, for any $(z,k)\in\Gamma$,
\begin{align*}
\frac{\mathrm{d}}{\mathrm{d} z}\alpha_k(z)
&= \frac{\mathrm{d}}{\mathrm{d} z}
   \oint_{\textup{C}_k(z)} |\nabla H(x)| \mathrm{d} l_{z,k}
 = \oint_{\textup{C}_k(z)} \frac{\Delta H(x)}{|\nabla H(x)|} \mathrm{d} l_{z,k} \sim \Delta H (\mathbf{x}_i)\beta_k(z),
\quad \text{as } z \to H (\mathbf{x}_i).
\end{align*}
By the asymptotic behavior \eqref{eq:T} of $\beta$ near the vertices, it follows that as $(z,k)\to O_i$,
\begin{equation}\label{eq:deri-alp}
\frac{\mathrm{d}}{\mathrm{d} z}\alpha_k(z)
\sim
\begin{cases}
C,
& \text{if $O_i$ is an exterior vertex}, \\[0.3em]
C\Delta H (\mathbf{x}_i)\bigl|\log |z - H (\mathbf{x}_i)|\bigr|,
& \text{if $O_i=(H (\mathbf{x}_i),k)$ is an interior vertex}.
\end{cases}
\end{equation}To control the decay rate of $\alpha^R$ near the vertex $O_{{n}+1}=(R+1,m)$, we choose the cut-off function $\eta_R$ such that
\begin{equation}\label{eq:etaR}
    \frac{\ud^-}{\ud z}\eta_R(R+1)=
\lim_{z\to(R+1)^-}\frac{\eta_R(z)-\eta_R(R+1)}{z-(R+1)}=K_0
\end{equation} for some
$K_0\neq 0$ so that
\begin{equation}\label{eq:alphaR}
\frac{\ud^-}{\ud z}\alpha_m^R(z)
= \alpha_m'(R+1)\eta_R(R+1)
+ \alpha_m(R+1)\frac{\ud^-}{\ud z}\eta_R(R+1)
= \alpha_m(R+1)K_0\neq 0.
\end{equation}
In view of \eqref{eq:alphaR}, as $(z,k)\to O_i$,
\begin{equation}\label{eq:deri-alpR}
\frac{\mathrm{d}}{\mathrm{d} z}\alpha_k^R(z)
\sim
\begin{cases}
C, & \text{if $O_i$ is an exterior vertex with $i\neq n+1$,} \\
C(R), & \text{if $O_i$ is the exterior vertex $O_{n+1}$,} \\
C\Delta H (\mathbf{x}_i)\bigl|\log|z - H(\mathbf{x}_i)|\bigr|, & \text{if $O_i=(H(\mathbf{x}_i),k)$ is an interior vertex.}
\end{cases}
\end{equation}
Consequently,
$
\alpha_k^{R}(z)\sim \mathrm{const}\cdot |z-(R+1)|$
as  $(z,k)\to O_{{n}+1}$. In other words, at the exterior vertex $O_{{n}+1}$ of $\Gamma^R$, the function
$\alpha_k^R$ also decays linearly, as at the other exterior vertices.

  Restricted to the compact graph $\Gamma^R$, the stochastic process $\check{u}(t):=u^R(t)\mid_{\Gamma^R}$ satisfies
\begin{align}\label{eq:local}\ud \check{u}(t)=(\check{\mathcal{L}}\check{u}(t)+\check{B}(\check{u}(t)))\ud t+\check{G}(\check{u}(t))\ud W(t),\quad \check{u}(0)=u_0\mid_{\Gamma^R}.
\end{align}
 The nonlinearity $\check B: L^2_{\beta\gamma}(\Gamma^R)\to L^2_{\beta\gamma}(\Gamma^R)$ and $\check{G}:L^2_{\beta\gamma}(\Gamma^R)\to \mathscr{L}_2(\mathcal{U}_0,L^2_{\beta\gamma}(\Gamma^R))$ are, respectively, defined by, 
 $$\check B(f_1)(z,k):=b(f_1(z,k))\eta_R(z),\quad \check G(f_1)(\psi)(z,k)=g(f_1(z,k))\eta_R(z)\psi(z,k),\quad (z,k)\in\Gamma^R,$$
 where $f_1\in L^2_{\beta\gamma}(\Gamma^R)$ and $\psi\in \mathcal{U}_0$. In view of \eqref{eq:A} and \eqref{eq:T}, the operator
$\check{\mathcal{L}}$ fails to be uniformly elliptic on the graph
$\Gamma^R$, since the coefficient $\alpha_k^R/\beta_k$ vanishes at vertices. 
To handle this degeneracy, we introduce a regularization of the truncated problem \eqref{eq:local} in the next subsection.
\subsection{Regularized truncated approximation}
For each edge $J_k$ with $k=1,2,\ldots,m-1$ on $\Gamma^R$, we divide it into two subintervals
$J_k^{(1)}$ and $J_k^{(2)}$ of equal length. For $k=m$, we split
 $J_m \cong [ H _0, R+1]$ into two subintervals
\[
J_m^{(1)} \cong [ H _0,  H _0+1]
\quad \text{and} \quad
J_m^{(2)} \cong [ H _0+1, R+1].
\]
With this decomposition, each subinterval $\{J_k^{(l)}\}_{l=1,2}$ is incident to exactly one
vertex of the compact graph $\Gamma^R$.
We can define the regularization locally depending on whether the adjacent vertex is an exterior or an interior vertex. In view of \eqref{eq:A} and \eqref{eq:T}, it suffices to regularize $\alpha_k^{R}$ near exterior vertices and $\beta_k$ near interior vertices. To make this precise, set $\delta_{\min}:= \frac14\min_{1\le k\le m}|J_k|$, where $|J_k|$ is the  length of the edge $J_k$. For
 $\delta \in(0, \delta_{\min})$, we define
\begin{align}\label{eq:beta1}
\beta_k^\delta(z):=\beta_k(z),~(z,k)\in J_k^{(l)},\quad \text{if} ~J_k^{(l)}~ \text{is incident to an exterior vertex } O_i=(H (\mathbf{x}_i),k),\\\label{eq:alpha1}
\alpha_k^{R,\delta}(z):=\alpha_k^R(z),(z,k)\in J_k^{(l)},\quad \text{if} ~J_k^{(l)}~\text{is incident to an interior vertex } O_i=(H (\mathbf{x}_i),k).
\end{align}
For the remaining cases, the regularization functions can, for instance, be defined as in Example \ref{ex:reg} provided later. More generally, we impose the following assumption on the regularization functions, which is satisfied in particular by Example \ref{ex:reg}.

\begin{assumption}[Regularization functions]\label{asp:reg}
For some $\delta_0>0$ and for all $\delta\in(0,\delta_0)$ , there exist positive constants $\mathfrak{c}_1(\delta)$ and $\mathfrak{c}_4(\delta)$ depending on $\delta$ such that for any $(z,k)\in\mathring{J}_k$ and $k=1,\ldots,m$,
\begin{gather}\label{eq:delta1}
\max\{\mathfrak{c}_0\alpha_k^R(z),\mathfrak{c}_1(\delta)\}\le \alpha_k^{R,\delta}(z)\le \mathfrak{c}_2(\alpha_k+1),\quad\lim_{\delta\to0}\alpha_k^{R,\delta}(z)=\alpha_k^R(z),\\\label{eq:delta2}
0<\mathfrak{c}_3\le \beta_k^{\delta}(z)\le \min\{\mathfrak{c}_4(\delta),\mathfrak{c}_5\beta_k(z)\},\quad\lim_{\delta\to0}\beta_k^{\delta}(z)=\beta_k(z),
\end{gather}
where $\mathfrak{c}_0,\mathfrak{c}_2,\mathfrak{c}_3,$ and $\mathfrak{c}_5$ are some constants independent of $R$ and $\delta$.
\end{assumption}

Accordingly, for $\delta\in(0,\delta_0)$, we define the regularization of $\check{\mathcal{L}}$ as follows
 \begin{equation}\label{eq:Ldelta}
	\check{\mathcal{L}}^\delta f(z,k)=\frac{1}{ 2\beta_k^\delta(z)} \frac{\ud}{\ud z}\left(\alpha_{k}^{R,\delta}(z)\frac{\ud}{\ud z}f(z,k)
	\right),\quad \text{if } (z,k) \in\mathring{J}_k,k=1,\ldots,m,
\end{equation}
 endowed with the following gluing conditions: the function $f$ is continuous at each interior vertex $O_i$ on $\Gamma^R$ and satisfies the following Kirchhoff condition
\begin{equation}\label{eq:gluedelta}
	\sum_{k: I_k \text{ is incident to } O_i} \alpha_{k}^{R,\delta}(H (\mathbf{x}_i))\ud_k f(H (\mathbf{x}_i),k) = 0,
	\qquad \text{for each vertex } O_i \text{ of } \Gamma^R .
\end{equation}
Now we can introduce the regularized truncated approximation of \eqref{eq:SPDE}:
\begin{align}\label{eq:reguarlized}
\ud \check{u}^\delta(t)=(\check{\mathcal{L}}^\delta\check{u}^\delta(t)+\check{B}(\check{u}^\delta(t)))\ud t+\check{G}(\check{u}^\delta(t))\ud W(t),\quad \check{u}^\delta(0)=\check{u}(0).
\end{align}
We impose the following assumption to ensure the well-posedness of the truncated equation \eqref{eq:reguarlized} (see Lemma \ref{lem:analydelta} for more details).
\begin{assumption}[Weight function]\label{asp:reg-gamma}
There exists a constant $\kappa_1>0$ such that for any $\delta\in(0,\delta_0)$,
\begin{equation*}
\alpha_k^{R,\delta}(z)|\gamma^\prime_k(z)|^2\le \kappa_1\beta_k^\delta(z)|\gamma_k(z)|^2\quad\text{for}~ (z,k)\in \Gamma^R.
\end{equation*}
\end{assumption}
In view of Assumption \ref{asp:reg}, Assumption \ref{asp:reg-gamma} holds under the condition \eqref{eq:gammader}. The following theorem identifies the limiting equation of the regularized truncated equation \eqref{eq:reguarlized} with the truncated problem \eqref{eq:uR}, as the regularization parameter $\delta\to 0$. 

\begin{theorem}\label{tho:con-in-pro}
Let \eqref{eq:etaR} and Assumptions \ref{asp:reg}--\ref{asp:reg-gamma} hold. Assume that $1_\Gamma\in L^2_{\beta\gamma}(\Gamma)$, and $u_0\in W^{1,2}_{(\alpha+1)\gamma,\beta\gamma}(\Gamma)$. Then for any $R>H_0$, $p\ge1$ and  $t\in[0,T]$,
\begin{align*}
\lim_{\delta\to 0}\E\bigg[\sup_{t\in[0,T]}\|\check{u}^\delta(t)-\check{u}(t)\|_{L_{\beta\gamma}^2(\Gamma^R)}^p\bigg]=0.
\end{align*}
\end{theorem}
The proof of Theorem \ref{tho:con-in-pro} can be found in Section \ref{S:Reg}.
For a fixed truncation parameter $R >  H _0 $ and a
fixed regularization parameter $\delta \in (0,\delta_0)$, the regularized truncated problem \eqref{eq:reguarlized} becomes a non-degenerate stochastic parabolic PDE posed on the compact graph $\Gamma^R$. Our goal is to construct and analyze a numerical method for approximating the solution $\check{u}^\delta$ to \eqref{eq:reguarlized} for each $R>H_0$ and $\delta>0$, and for approximating the original equation \eqref{eq:SPDE} via \eqref{eq:reguarlized} in the limits $\delta\to 0$ and $R\to\infty$.

\subsection{Regularized truncated finite element approximation}
Inspired by
\cite{AB18,BKKS24}, we construct a finite element discretization for \eqref{eq:reguarlized} based on an equally
spaced partition of the edges of the truncated graph $\Gamma^R$.
More precisely, for each edge $J_k$ we introduce a uniform subdivision into $n_k$
subintervals of length $h_k$. Let $h:=\min_{k=1}^m h_k$, and $\{(z_j,k)\}_{j=1}^{n_k-1}$ denote the interior nodes
of this subdivision. On each edge $J_k$, we define the standard hat basis functions by
\begin{align*}
\xi_j^k(z,k)=
\begin{cases}
1-\dfrac{|z-z_j|}{h_k}, &\quad z\in[z_{j-1},z_{j+1}],\\[0.3em]
0, &\quad \text{otherwise},
\end{cases}
\qquad j=1,2,\ldots,n_k-1,
\end{align*}
where $(z_0,k)$ and $(z_{n_k},k)$ are the vertices of $\Gamma^R$ attached to the
edge $J_k$, with $z_0 < z_{n_k}$. To connect the edges in the finite element approximation, we also introduce basis functions
associated with the vertices of the graph $\Gamma^R$. Specifically, for each vertex
$\{O_i\}_{i=1}^{{n}+1}$ of the graph $\Gamma^R$, 
the hat function centered at $O_i$ is defined by
\begin{align*}
\zeta_{i}(z,k)=
\begin{cases}
1-\dfrac{|z-H (\mathbf{x}_i)|}{h_k}, &\quad (z,k)\in\mathbb{W}_{O_i} ,\\[0.3em]
0, &\quad \text{otherwise},
\end{cases}
\end{align*}
where  $H (\mathbf{x}_i)$ denotes the $ z $-coordinate of $O_i$, and $\mathbb{W}_{O_i}$ is a neighborhood set of $O_i$ given by
\begin{align*}
\mathbb{W}_{O_i} = \bigcup_{k : I_k \text{ is incident to } O_i} \left\{ (z,k) \in J_k: |z - H (\mathbf{x}_i)| \le h_k \right\}.
\end{align*}
The finite dimensional space $\mathbb{V}_h(\Gamma^R)$ is spanned by the functions $\{\xi_j^k,k=1,\ldots,m;j=1,\ldots,n_k-1\}$ and $\{\zeta_i\}_{i=1}^{n+1}$, and its dimension is $\textup{dim}(\mathbb{V}_h(\Gamma^R))=\sum_{k=1}^m(n_k-1)+(n+1).$ 

Denote by $W^{1,2}_{\mathrm{cont}}(\Gamma^R):=W^{1,2,c}_{\alpha^{R,\delta}\gamma,\beta^\delta\gamma}(\Gamma^R)$ the subspace of $W^{1,2}_{\alpha^{R,\delta}\gamma,\beta^\delta\gamma}(\Gamma^R)$ consisting of functions that are continuous at each interior vertex of $\Gamma^R$.
For $\lambda\in\R$, consider
 the bilinear form $a_{\lambda}^{R,\delta}: W^{1,2}_{\mathrm{cont}}(\Gamma^R)\times W^{1,2}_{\mathrm{cont}}(\Gamma^R)\to \C$ defined by 
\begin{equation}\label{eq:app}
a_{\lambda}^{R,\delta}(\phi,\psi)=\lambda \sum_{k=1}^m\int_{J_k}\phi\overline{\psi}\beta_k^\delta\gamma_k\ud z+\frac12\sum_{k=1}^m\int_{J_k}\alpha_{k}^{R,\delta}\frac{\ud }{\ud z}\phi\frac{\ud }{\ud z}\left(\overline{\psi}\gamma_k\right)\ud z,\quad\phi,\psi\in W^{1,2}_{\mathrm{cont}}(\Gamma^R).
\end{equation}
Following the reasoning in the proof of Theorem \ref{lem:operatorL}, one can show that under Assumption \ref{asp:reg-gamma},
for any $\lambda \ge \tfrac{1}{8}\kappa_1$, the form $a_\lambda^{R,\delta}$ is densely defined,
continuous, closed, and accretive. 
Since $\mathbb{V}_h(\Gamma^R)\subset W^{1,2}_{\mathrm{cont}}(\Gamma^R)$, we can define the discrete operator $\check{\mathcal{L}}^\delta_h:\mathbb{V}_h(\Gamma^R)\to \mathbb{V}_h(\Gamma^R)$ by 
\begin{equation}\label{eq:Lhdelta}
	-\langle \check{\mathcal{L}}^\delta_h\varphi, \psi\rangle_{L^2_{\beta^\delta\gamma}(\Gamma^R)}= a_0^{R,\delta}(\varphi,\psi),\quad \varphi,\psi\in \mathbb{V}_h(\Gamma^R).
\end{equation}
We further introduce the Ritz projection $\mathcal{R}_h:W^{1,2}_{\mathrm{cont}}(\Gamma^R)\to \mathbb{V}_h(\Gamma^R)$ and the $L^2_{\beta^\delta\gamma}(\Gamma^R)$-projection $\mathcal{P}_h$ defined, respectively, via
\begin{gather*}
a_0^{R,\delta}(\mathcal{R}_h\varphi,\psi) = a_0^{R,\delta}(\varphi,\psi),\quad \varphi\in  W^{1,2}_{\mathrm{cont}}(\Gamma^R) ,\psi\in \mathbb{V}_h(\Gamma^R),\\
\langle \mathcal{P}_h\varphi,\psi\rangle_{L^2_{\beta^\delta\gamma}(\Gamma^R)}=\langle \varphi,\psi\rangle_{L^2_{\beta^\delta\gamma}(\Gamma^R)},\quad \varphi\in{L^2_{\beta^\delta\gamma}(\Gamma^R)},\psi\in \mathbb{V}_h(\Gamma^R).
\end{gather*}
We propose the following spatial semi-discrete approximation of \eqref{eq:reguarlized}: 
\begin{equation}\label{eq:udeltah}
\ud \check{u}^\delta_h(t)=(\check{\mathcal{L}}^\delta_h \check{u}^\delta_h(t)+\mathcal{P}_h\check{B}(\check{u}^\delta_h(t)))\ud t+\mathcal{P}_h\check{G}(\check{u}^\delta_h(t))\ud W(t),\quad \check{u}^\delta_h(0)=\mathcal{P}_h\check{u}(0).
\end{equation}

 The following assumption will be used in the interpolation error estimates on the weighted Sobolev spaces on the graph $\Gamma^R$ (see Lemma \ref{lem:int-L}).
\begin{assumption}\label{asp:alphader}
       For any fixed $R>H_0$, $\delta\in(0,\delta_0)$, and $k=1,\ldots,m$, the function $\alpha^{R,\delta}_k$ is continuously differentiable on the edge $J_k$, and its derivative $(\alpha_k^{R,\delta})'\in L^4(J_k)$.
       \end{assumption}

We remark that Assumption \ref{asp:alphader} is also satisfied by Example \ref{ex:reg} provided later.
The following theorem provides an error estimate for the finite element approximation \eqref{eq:udeltah} of the regularized equation \eqref{eq:reguarlized} for nonsmooth initial data, whose proof is deferred to Section \ref{S:FEM}.
\begin{theorem}\label{tho:FEM}
Let Assumptions \ref{asp:reg}, \ref{asp:reg-gamma}, and \ref{asp:alphader} hold, and  $\{u_0,1_\Gamma\}\subset L^2_{\beta\gamma}(\Gamma)$. Then for any $R\ge H_0$, $\delta\in(0,\delta_0)$, and $\vartheta\in(0,\frac12)$, there exists $C:=C(R,\delta,\vartheta,u_0)>0$ such that
\begin{align*}
\E\left[\|\check{u}^\delta(t)-\check{u}^\delta_h(t)\|_{L^2_{\beta\gamma}(\Gamma^R)}^2\right]\le Ch^{2\vartheta}(1+t^{-\vartheta}).
\end{align*}	
\end{theorem}
    For sufficiently large $R>H_0$ and small $\delta,h>0$, we define the following regularized truncated finite element approximation
    $$u^{R,\delta}_h(t,z,k):=\begin{cases}
        \check{u}^\delta_h(t,z,k),\quad &\text{if }(z,k)\in \Gamma^R,\\
        u_0(z,k),\quad &\text{if }(z,k)\in \Gamma-\Gamma^R
    \end{cases} $$ 
    for approximating the mild solution $u(t,z,k)$ to the stochastic PDE \eqref{eq:SPDE} on the graph $\Gamma$. Gathering Theorems \ref{tho:truncate}, \ref{tho:con-in-pro}, and \ref{tho:FEM} together, we obtain the following convergence result.
\begin{corollary}\label{tho:mr}
Let Assumptions \ref{asp:aT}, \ref{asp:reg}, \ref{asp:reg-gamma}, and \ref{asp:alphader} hold. Assume that $u_0\in W^{1,2}_{(\alpha+1)\gamma,\beta\gamma}(\Gamma)\cap L^2_{\beta\sqrt\gamma}(\Gamma)$, \eqref{eq:etaR}, and  \eqref{eq:convercond} hold. Then for any $t\in(0,T]$,
\begin{align*}
\lim_{R\to\infty}\lim_{\delta\to0}\lim_{h\to 0}\E\left[\|u(t)-u^{R,\delta}_h(t)\|^2_{L^2_{\beta\gamma}(\Gamma)}\right]=0.
\end{align*}	
\end{corollary}
\begin{proof}
Since $u^R(t)=u_0$ on $\Gamma-\Gamma^R$ and $\check{u}(t)=u^{R}(t)1_{\Gamma^R}$, we have
\begin{align*}
u(t)-u^{R,\delta}_h(t)&=(u(t)-u^{R}(t))1_{\Gamma-\Gamma^R}+(u(t)-\check{u}^\delta_h(t))1_{\Gamma^R}\\
&=u(t)-u^{R}(t)+(\check{u}(t)-\check{u}^\delta_h(t))1_{\Gamma^R}.
\end{align*}
Therefore, by the H\"older inequality,
\begin{align*}
&\E\left[\|u(t)-u^{R,\delta}_h(t)\|^2_{L^2_{\beta\gamma}(\Gamma)}\right]
\le2\E\left[\|u(t)-u^{R}(t)\|^2_{L^2_{\beta\gamma}(\Gamma)}\right]+2\E\left[\|\check{u}(t)-\check{u}^\delta_h(t)\|^2_{L^2_{\beta\gamma}(\Gamma^R)}\right]\\
&\le2\E\left[\|u(t)-u^{R}(t)\|^2_{L^2_{\beta\gamma}(\Gamma)}\right]+4\E\left[\|\check{u}(t)-\check{u}^\delta(t)\|^2_{L^2_{\beta\gamma}(\Gamma^R)}\right]+4\E\left[\|\check{u}^\delta(t)-\check{u}^\delta_h(t)\|^2_{L^2_{\beta\gamma}(\Gamma^R)}\right].
\end{align*}
Finally, by applying Theorems \ref{tho:truncate}, \ref{tho:con-in-pro}, and \ref{tho:FEM}, we can complete the proof. 
\end{proof}
\section{Error analysis of truncated approximation}\label{S:Trun}
In this section, we carry out the error analysis of the truncated approximation and present the proof of Theorem \ref{tho:truncate}. To this end, we first show that the truncated problem \eqref{eq:uR} admits a unique mild solution that is uniformly bounded in $L^p(\Omega;\mathcal{C}([0,T];L^2_{\beta\gamma}(\Gamma)))$ for $p\ge1$, with bounds independent of the truncation parameter $R$.  
Denote by $B^R$ and $G^R$ the truncations of $B$ and $G$ (see \eqref{eq:BG}), defined respectively, by 
$$B^R(f_1)(z,k):=b(f_1(z,k))\eta_R(z),\quad G^R(f_1)\psi(z,k):=g(f_1(z,k))\eta_R(z)\psi(z,k),\quad (z,k)\in\Gamma,$$
where $f_1\in L^2_{\beta\gamma}(\Gamma)$ and $\psi\in \mathcal{U}_0$.
Then by $|\eta_R|\le 1$, for any $R>H_0$ and any $f_1,f_2\in L^2_{\beta\gamma}(\Gamma)$,
\begin{gather}\label{eq:BRLip}
    \|B^R(f_1)-B^R(f_2)\|_{L^2_{\beta\gamma}(\Gamma)}\le  \|B(f_1)-B(f_2)\|_{L^2_{\beta\gamma}(\Gamma)},\\\label{eq:GRLip}
     \|G^R(f_1)-G^R(f_2)\|_{\mathscr{L}_2(\mathcal{U}_0,L^2_{\beta\gamma}(\Gamma))}\le  \|G(f_1)-G(f_2)\|_{\mathscr{L}_2(\mathcal{U}_0,L^2_{\beta\gamma}(\Gamma))}.
\end{gather}

\begin{lemma}\label{lem:wp-trun}
Let Assumption \ref{asp:aT} hold and $\{1_{\Gamma},u_0\}\subset L^2_{\beta\gamma}(\Gamma)$. 
    Then for any fixed $R> H _0$, the truncated equation \eqref{eq:uR} admits a unique mild solution $u^R=\{u^R(t)\}_{t\in[0,T]}$ satisfying
$$u^R(t)=e^{t\mathcal{L}^R}u_0+\int_0^t e^{(t-s)\mathcal{L}^R}B^R(u^R(s))\ud s+\int_0^t e^{(t-s)\mathcal{L}^R} G^R(u^R(s))\ud W(s).$$
Moreover, for any $p\ge1$, there exists $C:=C(T,p,\kappa)>0$ such that 
\begin{align}\label{eq:L2N}
\sup_{R>H_0}\E\bigg[\sup_{t\in[0,T]}\|u^R(t)\|^p_{L^2_{\beta\gamma}(\Gamma)}\bigg]\le C\Big(1+\|u_0\|_{L^2_{\beta\gamma}(\Gamma)}^p\Big).
\end{align}
\end{lemma}
\begin{proof}
Under Assumption \ref{asp:aT}, the condition \eqref{eq:theta} holds with $\alpha_k$ replaced by $\alpha_k^R=\alpha_k\cdot\eta_R$. 
Hence, by Theorem \ref{lem:operatorL}, for any $\lambda\ge\frac18\kappa$ and $R>H_0$, the operator $-\lambda I+\mathcal{L}^R$ generates an analytic contraction semigroup  on $L^2_{\beta\gamma}(\Gamma)$. In particular, the semigroup $\{e^{t\mathcal{L}^R}\}_{t\ge0}$ generated by $\mathcal{L}^R$ satisfies 
$$\|e^{t\mathcal{L}^R}\|_{\mathscr{L}(L^2_{\beta\gamma}(\Gamma))}\le e^{\frac18 \kappa  t}\quad \forall~t\ge 0,\quad R> H _0.$$

A combination of \eqref{eq:BRLip}, \eqref{eq:GRLip}, and \eqref{eq:BLip} implies that
    $B^R:L^2_{\beta\gamma}(\Gamma)\to L^2_{\beta\gamma}(\Gamma)$ and $G^R:L^2_{\beta\gamma}(\Gamma)\to\mathscr{L}_2(\mathcal{U}_0,L^2_{\beta\gamma}(\Gamma))$ are Lipschitz continuous with the Lipschitz constants independent on $R$. Moreover, by the assumption $1_\Gamma\in L^2_{\beta\gamma}(\Gamma)$, for any $f_1\in L^2_{\beta\gamma}(\Gamma)$ we have
    \begin{gather*}
    \|B^R(f_1)\|_{L^2_{\beta\gamma}(\Gamma)}
+\|G^R(f_1)\|_{\mathscr{L}_2(\mathcal{U}_0,L^2_{\beta\gamma}(\Gamma))}\le  C(1+   \|f_1\|_{L^2_{\beta\gamma}(\Gamma)}).
\end{gather*}
The remainder of the proof is similar to that of Lemma \ref{lem:L2}, and we omit the details.
\end{proof}

\textbf{Proof of Theorem \ref{tho:truncate}.}
Subtracting \eqref{eq:uR} from \eqref{eq:SPDE}, the truncation error $\mathcal{E}^R:=u-u^R$ satisfies
\begin{equation*}
\ud \mathcal{E}^R(t)=[\mathcal{L}u(t)-\mathcal{L}^Ru^R(t)+B(u(t))-B^R(u^R(t))]\ud t +[G(u(t))-G^R(u^R(t))]\ud W(t)
\end{equation*}
subject to the initial value $\mathcal{E}^R(0)=0$. 
We partition the graph into two regions $\mathcal{D}_1:=\{(z,m)\in I_m, z> R-1\}$ and $\mathcal{D}_2:=
   \cup_{k=1}^{m-1}I_k\cup\{(z,m)\in I_m, z\le R-1\}$, and treat the truncation error separately. 
   
We observe that, under Assumption \ref{asp:aT}, the inequality \eqref{eq:theta} remains valid with $\gamma$ replaced by $\sqrt{\gamma}$.  By the assumption $\{1_\Gamma,u_0\}\subset L_{\beta\sqrt\gamma}^2(\Gamma)$, it follows that \eqref{eq:L2} and \eqref{eq:L2N} hold with $\gamma$ replaced by $\sqrt{\gamma}$; that is, for any $p\ge1$,
\begin{equation}\label{coro:IkTksqrt}
\E\bigg[\sup_{t\in[0,T]}\|u(t)\|^p_{L^2_{\beta\sqrt{\gamma}}(\Gamma)}\bigg]+\sup_{R>  H _0}\E\bigg[\sup_{t\in[0,T]}\|u^R(t)\|^p_{L^2_{\beta\sqrt{\gamma}}(\Gamma)}\bigg]\le C\left(1+\|u_0\|_{L^2_{\beta\sqrt\gamma}(\Gamma)}^p\right).
\end{equation}
As a consequence of \eqref{coro:IkTksqrt}, for any $R\ge H _0+1$,
\begin{align}\label{eq:Einfty}
\E\int_{\mathcal{D}_1} |\mathcal{E}^R(t)|^2\gamma_k\beta_k\ud z&=\E\int_{R-1}^\infty |\mathcal{E}^R(t,z,m)|^2\gamma_m(z)\beta_m(z)\ud z\\\notag
&\le 2\sup_{y\ge R-1}\sqrt{\gamma_m(y)}\left(\E\big[\|u^R(t)\|_{L^2_{\beta\sqrt\gamma}(\Gamma)}^2\big]+\E\big[\|u(t)\|_{L^2_{\beta\sqrt\gamma}(\Gamma)}^2\big]\right)\\\notag
&\le C\sup_{y\ge R-1}\sqrt{\gamma_m(y)}.
\end{align}

To handle the region 
$\mathcal{D}_2$,
we introduce a smooth function $\zeta:\mathbb{R}\to[0,1]$ such that
$\zeta(z)=1$ for $z\le 0$, $\zeta(z)=0$ for $z\ge 1$, and
$\zeta'(0)=\zeta'(1)=0$.
For instance, one can define
$\zeta(z)=1-\frac{1}{C}\int_0^{z} e^{-\frac{1}{t(1-t)}}\ud t$ for $z\in(0,1)$,
with the normalizing constant $C=\int_0^1 e^{-\frac{1}{t(1-t)}}\ud t>0$.
For $R\ge H_0+1$, we define $\zeta_R(z)=\zeta(z-R+1)$ such that
$\zeta_R(z)=1$ for $z\in[0,R-1]$, $\zeta_R(z)=0$ for $z\ge R$, and
$\zeta_R'(R-1)=\zeta_R'(R)=0$. Introduce another graph weight function $\tilde{\gamma}:\Gamma\to [0,\infty)$ by setting $\tilde{\gamma}_k(z)=\gamma_k(z)\cdot\zeta_R(z)$ for $(z,k)\in\Gamma$.
Together with \eqref{eq:Einfty},  this implies that for any $R\ge H _0+1$, 
\begin{align}\label{eq:ERtH}
 \E\left[\|\mathcal{E}^R(t)\|_{L^2_{\beta\gamma}(\Gamma)}^2\right]
 &=\E\int_{\mathcal{D}_1} |\mathcal{E}^R(t)|^2\gamma_k\beta_k\ud z+\E\int_{\mathcal{D}_2} |\mathcal{E}^R(t)|^2\gamma_k\beta_k\ud z
 \\\notag&\le C\sup_{y\ge R-1}\sqrt{\gamma_m(y)}+\E\Big[\|\mathcal{E}^R(t)\|_{L^2_{\beta\tilde{\gamma}}(\Gamma)}^2\Big],
\end{align}
where we have also used the facts $\zeta_R\ge0$ on $\Gamma$ and $\zeta_R=1$ on $\mathcal{D}_2$ in the last step.

Applying the It\^o formula, for any $t\in(0,T]$ we have
\begin{align}\label{eq:EERt-0}\notag
\ud \|\mathcal{E}^R(t)\|_{L^2_{\beta\tilde{\gamma}}(\Gamma)}^2&=-\sum_{k=1}^m\int_{I_k}\alpha_k\partial_zu(t)\partial_z(\mathcal{E}^R(t)\tilde{\gamma}_k)\ud z+\sum_{k=1}^m\int_{I_k}\alpha_{k}^R\partial_zu^R(t)\partial_z(\mathcal{E}^R(t)\tilde{\gamma}_k)\ud z\\\notag
&\quad+2\langle B(u(t))-B^R(u^R(t)),\mathcal{E}^R(t)\rangle_{L^2_{\beta\tilde{\gamma}}(\Gamma)}\ud t\\\notag
&\quad+2\langle[G(u(t))-G^R(u^R(t))]\ud W(t),\mathcal{E}^R(t)\rangle_{L^2_{\beta\tilde{\gamma}}(\Gamma)}\\
&\quad+\|G(u(t))-G^R(u^R(t))\|_{\mathscr{L}_2(\mathcal{U}_0, L^2_{\beta\tilde{\gamma}}(\Gamma))}^2\ud t.
\end{align}
Invoking the linear growth of $g$ and \eqref{eq:e_i}, one has
\begin{align*}
\|G(u(t))-G^R(u(t))\|_{\mathscr{L}_2(\mathcal{U}_0, L^2_{\beta\tilde{\gamma}}(\Gamma))}^2\le C\int_{I_m}|\eta_R-1|^2(1+|u(t)|^2)\beta_m{\tilde\gamma}_m\ud z=:\mathcal{K}_0^R(t).
\end{align*}
This, along with \eqref{eq:GRLip}, implies that  for any $R\ge H _0+1$ and $t\in[0,T]$,
\begin{equation}\label{eq:J2}
\|G(u(t))-G^R(u^R(t))\|_{\mathscr{L}_2(\mathcal{U}_0, L^2_{\beta\tilde{\gamma}}(\Gamma))}^2
\le C \|\mathcal{E}^R(t)\|_{L^2_{\beta\tilde{\gamma}}(\Gamma)}^2+C\mathcal{K}_0^R(t).
\end{equation}
Similarly, by the Lipschitz continuity of $b$, it also holds that
 \begin{equation}\label{eq:J3}
2\langle B(u(t))-B^R(u^R(t)),\mathcal{E}^R(t)\rangle_{L^2_{\beta\tilde{\gamma}}(\Gamma)}
\le C \|\mathcal{E}^R(t)\|_{L^2_{\beta\tilde{\gamma}}(\Gamma)}^2+C\mathcal{K}_0^R(t),\quad t\in[0,T].
\end{equation}
Since $\textup{supp} (\eta_R-1)\subset[R,\infty)$ and $\textup{supp} (\zeta_{R})\subset[0,R]$, we obtain $(\eta_R(z)-1)\tilde\gamma_k(z)= 0$ for all $(z,k)\in \Gamma$, which implies that the function $\mathcal{K}_0^R(t)$ 
 vanishes identically on $\Gamma$. 
In addition, by $\mathcal{E}^R=u-u^R$ and $\tilde{\gamma}=\gamma\zeta_R$, we can recast
\begin{align}\label{eq:JR}\notag
&-\sum_{k=1}^m\int_{I_k}\alpha_k\partial_zu(t)\partial_z(\mathcal{E}^R(t)\tilde{\gamma}_k)\ud z+\sum_{k=1}^m\int_{I_k}\alpha_{k}^R\partial_zu^R(t)\partial_z(\mathcal{E}^R(t)\tilde{\gamma}_k)\ud z\\\notag
&=\sum_{k=1}^m\int_{I_k}\alpha_k(\eta_R-1)\partial_zu(t)\partial_z(\mathcal{E}^R(t)\tilde{\gamma}_k)\ud z-\sum_{k=1}^m\int_{I_k}\alpha_{k}^R\partial_z\mathcal{E}^R(t)\partial_z(\mathcal{E}^R(t)\tilde{\gamma}_k)\ud z\\
&=-\sum_{k=1}^m\int_{I_k}\alpha_{k}^R\partial_z\mathcal{E}^R(t)\partial_z(\mathcal{E}^R(t)\tilde{\gamma}_k)\ud z=\sum_{i=1}^2\mathcal{K}_i^R(t),
\end{align}
where the remainder terms are 
\begin{align*}
\mathcal{K}_1^R(t)&=-\sum_{k=1}^m\int_{I_k}\alpha_{k}^R|\partial_z\mathcal{E}^R(t)|^2{\tilde\gamma}_k\ud z-\sum_{k=1}^m\int_{I_k}\alpha_{k}^R\partial_z\mathcal{E}^R(t)\mathcal{E}^R(t)\gamma_k^\prime\zeta_{R}\ud z,\\
\mathcal{K}_2^R(t)&=-\sum_{k=1}^m\int_{I_k}\alpha_{k}^R\partial_z\mathcal{E}^R(t)\mathcal{E}^R(t)\gamma_k\zeta^\prime_{R}\ud z.
\end{align*}
Inserting the estimates \eqref{eq:J2}, \eqref{eq:J3}, and \eqref{eq:JR} into \eqref{eq:EERt-0}, we conclude that for any $t\in[0,T]$,
\begin{equation}\label{eq:EERt-1}
\E\left[\|\mathcal{E}^R(t)\|_{L^2_{\beta\tilde{\gamma}}(\Gamma)}^2\right]
\le\|\mathcal{E}^R(0)\|_{L^2_{\beta\tilde{\gamma}}(\Gamma)}^2
+\sum_{i=1}^2\E\int_0^t\mathcal{K}_i^R(s)\ud s+C \E\int_0^t\|\mathcal{E}^R(s)\|_{L^2_{\beta\tilde{\gamma}}(\Gamma)}^2\ud s.
\end{equation}
Here we have also used the fact that $\mathcal{K}_0^R(t)$ 
 vanishes identically on $\Gamma$.
Since $\eta_R\equiv 1$ on the support of $\zeta_R$, it follows from Young's inequality and Assumption \ref{asp:aT} that
\begin{align}\label{eq:K12R}
\mathcal K_1^R(t)
&\le -\frac12\sum_{k=1}^m\int_{I_k}\alpha_{k}^R|\partial_z\mathcal{E}^R(t)|^2\gamma_k\zeta_R\ud z+\frac12\sum_{k=1}^m\int_{I_k}\alpha_{k}^R|\mathcal{E}^R(t)|^2|\gamma_k^\prime|^2\zeta_{R}\gamma^{-1}_k\ud z\\\notag
&\le -\frac12\sum_{k=1}^m\int_{I_k}\alpha_{k}^R|\partial_z\mathcal{E}^R(t)|^2\gamma_k\zeta_R\ud z+\frac12\kappa \sum_{k=1}^m\int_{I_k}\eta_R|\mathcal{E}^R(t)|^2\zeta_{R}\beta_k\gamma_k\ud z\\\notag
&= -\frac12\sum_{k=1}^m\int_{I_k}\alpha_{k}|\partial_z\mathcal{E}^R(t)|^2\gamma_k\zeta_R\ud z+\frac12\kappa \sum_{k=1}^{m}\int_{I_k}|\mathcal{E}^R(t)|^2\beta_k\tilde{\gamma}_k\ud z.
\end{align}
To estimate $\mathcal K_{2}^R(t)$, we note that the function $\eta_R\zeta_{R}^\prime=\zeta_{R}^\prime$ is supported on $[R-1,R]$. Combined with the facts that
$\zeta_{R}^\prime(R-1)=\zeta_{R}^\prime(R)=0$ and by applying the integration by parts formula, we obtain
\begin{align}\label{eq:K3R}
\mathcal K_2^R(t)
&=-\frac12\int_{R-1}^{R}\alpha_{m}(z)\partial_z\left(\mathcal{E}^R(t,z,m)^2\right)\gamma_m(z)\zeta^\prime_{R}(z)\ud z\\\notag
&=\frac12\int_{R-1}^{R}\mathcal{E}^R(t,z,m)^2\frac{\ud}{\ud z}\left(\alpha_{m}(z)\gamma_m(z)\zeta^\prime_{R}(z)\right)\ud z.
\end{align}
From the chain rule, Assumption \ref{asp:aT}, and the uniformly boundedness of $\zeta_R^\prime$ and $\zeta_R^{\prime\prime}$, we infer that
\begin{align*}
    \frac{\ud}{\ud z}\left(\alpha_{m}(z)\gamma_m(z)\zeta^\prime_{R}(z)\right)&=  \alpha_{m}^\prime(z)\gamma_m(z)\zeta^\prime_{R}(z)+\alpha_{m}(z)\gamma_m^\prime(z)\zeta^\prime_{R}(z)+   \alpha_{m}(z)\gamma_m(z)\zeta^{\prime\prime}_{R}(z)\\
    &\le C|\alpha_{m}^\prime(z)|\gamma_m(z)+C\sqrt{\alpha_{m}(z)}\sqrt{\beta_m(z)}\gamma_m(z)+C \alpha_{m}(z)\gamma_m(z).
\end{align*}
Referring to \eqref{eq:T}, the function $\beta_m$ is uniformly bounded from below by a positive constant on the edge $I_m$ connecting $O_\infty$. Hence for any $z\in[R-1,R]$ with $R\ge H _0+1$,
\begin{align*}
    \frac{\ud}{\ud z}(\alpha_{m}(z)\gamma_m(z)\zeta^\prime_{R}(z))
    &\le C(|\alpha_{m}^\prime(z)|+\sqrt{\alpha_{m}(z)}+ \alpha_{m}(z))\gamma_m(z)\beta_m(z).
\end{align*}
Moreover, it follows from \eqref{eq:deri-alp} that $|\alpha_m'(z)|$ is uniformly bounded for $z\in [H_0+1,\infty).$ Together with $\alpha_m(z)\sim Cz$ as $z\to\infty$, this implies that for any $z\in[R-1,R]$ with $R\ge H _0+1$, 
\begin{align}\label{eq:alpham}
    \frac{\ud}{\ud z}(\alpha_{m}(z)\gamma_m(z)\zeta^\prime_{R}(z))
    &\le C(1+\alpha_{m}(z))\gamma_m(z)\beta_m(z)\\\notag
    &\le C\sup_{y\ge R-1}\{\alpha_m(y)\sqrt{\gamma_m(y)}\}\sqrt{\gamma_m(z)\beta_m(z)}.
\end{align}
Taking account of \eqref{eq:K3R}, \eqref{eq:alpham}, and \eqref{coro:IkTksqrt}  yields
\begin{align}\label{eq:K3R-new}
\E\left[\mathcal K_{2}^R(t)\right]&\le C\Big(\E\Big[\|u^R(t)\|_{L^2_{\beta\sqrt\gamma}(\Gamma)}^2\Big]+\E\Big[\|u(t)\|_{L^2_{\beta\sqrt\gamma}(\Gamma)}^2\Big]\Big)\sup_{y\ge R-1}\{\alpha_m(y)\sqrt{\gamma_m(y)}\}\\\notag
&\le C\sup_{y\ge R-1}\{\alpha_m(y)\sqrt{\gamma_m(y)}\}.
\end{align}
Plugging \eqref{eq:K12R} and \eqref{eq:K3R-new} into \eqref{eq:EERt-1}, for any $t\in[0,T]$ we have
\begin{align*}
&\E\left[\|\mathcal{E}^R(t)\|_{L^2_{\beta\tilde{\gamma}}(\Gamma)}^2\right]+\frac12\E\int_0^t\sum_{k=1}^m\int_{I_k}\alpha_{k}|\partial_z\mathcal{E}^R(s)|^2\gamma_k\zeta_R\ud z\ud s\\
&\le\|\mathcal{E}^R(0)\|_{L^2_{\beta\tilde{\gamma}}(\Gamma)}^2+C \E\int_0^t\|\mathcal{E}^R(s)\|_{L^2_{\beta\tilde{\gamma}}(\Gamma)}^2\ud s+C\sup_{y\ge R-1}\{\alpha_m(y)\sqrt{\gamma_m(y)}\}.
\end{align*}
Since $\mathcal{E}^R(0)=0$, applying Gronwall's inequality implies that for any $R\ge H _0+1$,
\begin{align*}
\E\left[\|\mathcal{E}^R(t)\|_{L^2_{\beta\tilde{\gamma}}(\Gamma)}^2\right]\le 
C\sup_{y\ge R-1}\{\alpha_m(y)\sqrt{\gamma_m(y)}\}.
\end{align*}
This, along with \eqref{eq:ERtH} and the fact that $\alpha_m$ is bounded from below by a positive constant, gives 
\begin{align*}
 \E\left[\|\mathcal{E}^R(t)\|_{L^2_{\beta\gamma}(\Gamma)}^2\right]
 \le C\sup_{y\ge R-1}\{(\alpha_m(y)+1)\sqrt{\gamma_m(y)}\}\le C\sup_{y\ge R-1}\{\alpha_m(y)\sqrt{\gamma_m(y)}\},
\end{align*}
as required. The proof of Theorem \ref{tho:truncate} is completed.
\hfill$\square$

\begin{remark}
    If we were to directly apply the It\^o formula to 
$
\|\mathcal{E}^R(t)\|_{L^2_{\beta\gamma}(\Gamma)}^2,
$
it would become necessary to estimate the following term:
$
\sum_{k=1}^m \int_{I_k} \alpha_k (\eta_R - 1)  \partial_z u(t)  \partial_z (\mathcal{E}^R(t) \gamma_k)  dz;
$
see \eqref{eq:JR} with $\tilde{\gamma}$ replaced by $\gamma$.
 The challenge of this approach lies in controlling the term where the derivative acts on the error $\mathcal{E}^R(t)$. Specifically, one would need a suitable estimate for:
\begin{equation}\label{eq:partialR}
    -\int_{R}^{R+1}\alpha_m \partial_z u(t)
\partial_z \mathcal{E}^R(t) \gamma_m \mathrm{d} z=-\int_{R}^{R+1}\alpha_m |\partial_z u(t)|^2\gamma_m \mathrm{d} z+\int_{R}^{R+1}\alpha_m \partial_z u(t)\partial_zu^R(t)
\gamma_m \mathrm{d} z.
\end{equation}
However, for the derivative of the truncated solution $u^R$, we only have the following bound:
\begin{align*}
\E\int_0^T\sum_{k=1}^m\int_{I_k}\alpha_k^R(z)|\frac{\partial}{\partial z}u^R(t,z,k)|^2\gamma_k(z)\ud z \ud t\le C.
\end{align*}
This bound is insufficient to control the second term on the right-hand side of \eqref{eq:partialR} because the coefficient $\alpha_m^R(z)$ degenerates near the vertex $O_{n+1}=(R+1,m)$. To overcome this difficulty, in the proof of Theorem \ref{tho:truncate}, we split $\Gamma=\mathcal{D}_1\cap \mathcal{D}_2$. 
In $\mathcal{D}_1$, the rapid decay of the weight $\gamma$ facilitates a straightforward estimate of the tail. In contrast, for $\mathcal{D}_2$, by introducing an additional cut-off function $\zeta_R$, we can localize the error analysis and effectively manage boundary contributions. 
\end{remark}

\section{Convergence of regularized  truncated approximation}\label{S:Reg}
In this section, we prove Theorem \ref{tho:con-in-pro} on the convergence of the regularized  truncated approximation \eqref{eq:reguarlized} for the truncated problem \eqref{eq:local}. We emphasize that the generic constant $C$ in this section may depend on $R$, but is always independent of $\delta$ and $h$. We begin by presenting an example of the regularization function.

\begin{example}\label{ex:reg}
 Let $\delta \in(0, \delta_{\min})$ and $\varrho:[0,\infty)\to[0,1]$ be a smooth function satisfying \eqref{eq:varphi}. Owing to \eqref{eq:beta1} and \eqref{eq:alpha1}, it remains to specify $\alpha_k^{R,\delta}$ near exterior vertices and $\beta_k^\delta$ near interior vertices.
If a subinterval $J_k^{(l)}\subset J_k$ is incident to an exterior vertex $O_i=(H (\mathbf{x}_i),k)$, we set
\begin{align*}
\alpha_k^{R,\delta}(z)
:= \alpha_k^R(z) + \delta\varrho\left(\delta^{-1}| z-H (\mathbf{x}_i)|\right),
\quad (z,k)\in J_k^{(l)}.
\end{align*}
If a subinterval $J_k^{(l)}\subset J_k$ is incident to an interior vertex $O_i=(H (\mathbf{x}_i),k)$, we define
\begin{align*}
\beta_k^\delta(z)
:= \varrho(\delta^{-1}| z-H (\mathbf{x}_i)|)\beta_k(H (\mathbf{x}_i) \pm \delta)
+ \big(1-\varrho(\delta^{-1}| z-H (\mathbf{x}_i)|)\big)\beta_k(z),\quad (z,k)\in J_k^{(l)},
\end{align*}
where the sign $+$ is taken if $H (\mathbf{x}_i)+\delta\in J_k$, and the  sign $-$ is taken if $H (\mathbf{x}_i)-\delta\in J_k$. In Fig.~\ref{fig:reg}, we plot the regularization of the functions $\alpha_k^R(z)=z$ and $\beta_k(z)=|\log|z||$ near $z=0$, where $\varrho(z)=\frac12[\cos((z-1)\pi)+1]$ for $1<z<2$.
\end{example}
\begin{figure}[htbp]
    \centering
    \begin{subfigure}{0.45\textwidth}
        \centering
        \includegraphics[width=\textwidth]{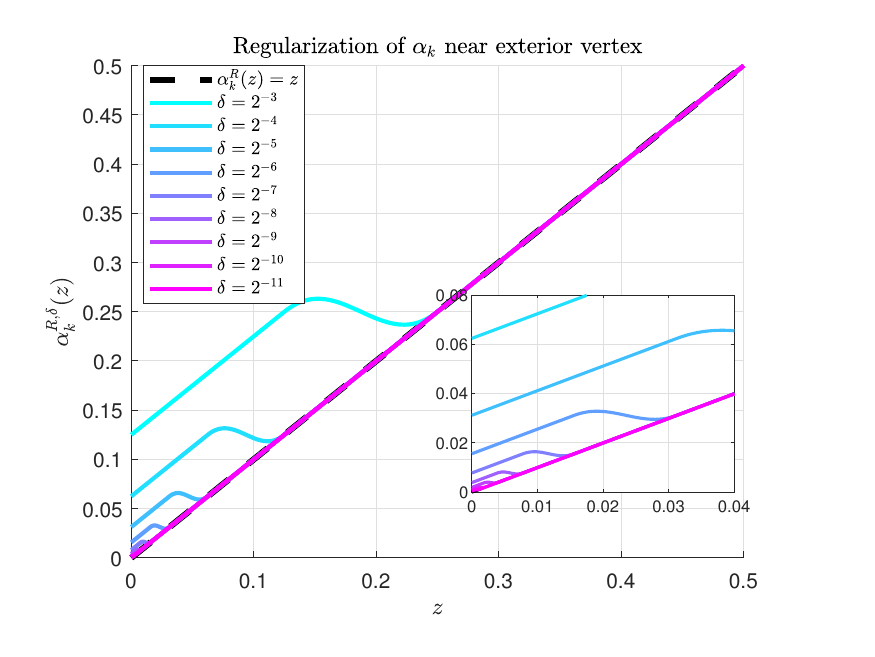}
    \end{subfigure}
    \hfill
    \begin{subfigure}{0.45\textwidth}
        \centering
        \includegraphics[width=\textwidth]{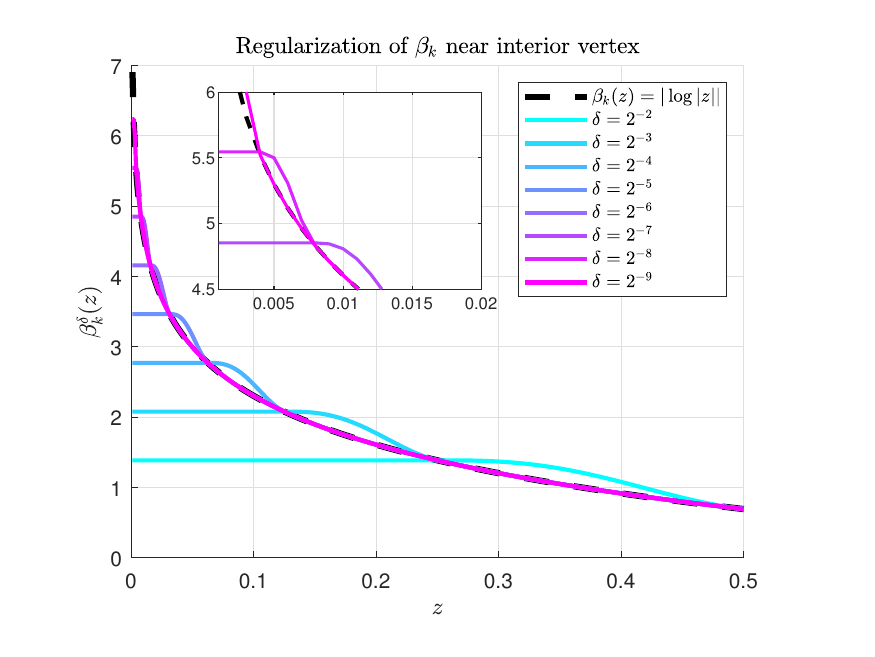}
    \end{subfigure}
    \caption{Regularization function used in Example~\ref{ex:reg}.}
    \label{fig:reg}
\end{figure}

\begin{lemma}
 Assumption~\ref{asp:reg} is satisfied for the regularization function in Example \ref{ex:reg}.
\end{lemma}

\begin{proof}
For \eqref{eq:delta1}, we can take $\mathfrak{c}_0=\mathfrak{c}_2=1$, and $\mathfrak{c}_1(\delta)=\min\{\delta,C_1(\delta))$, where
\begin{align*}
C(\delta)=\inf\{
\alpha_k(z);(z,k)\in \Gamma, |z-H (\mathbf{x}_i)|\ge \delta \text{ for all exterior vertices }O_i=(H (\mathbf{x}_i),k)\}.
\end{align*}
Next, by \eqref{eq:T}, for each interior vertex $O_i=(H (\mathbf{x}_i),k)$,
there exists a $\delta_i\in(0,\delta_{\min}\wedge 1)$ and two positive constants $K_1(i)$ and $K_2(i)$ such that
\[
- K_1(i)\log|z-H (\mathbf{x}_i)|
\le \beta_k(z)
\le - K_2(i)\log|z-H (\mathbf{x}_i)|
\quad\forall~
|z-H (\mathbf{x}_i)|\le \delta_i .
\]
Set
$
\delta_{0}:=\frac14\min_{O_i\in\mathcal{O}_{\mathrm{int}}}\delta_i,
$
$K_1:=\min_{O_i\in\mathcal{O}_{\mathrm{int}}}K_1(i),
$ and
$K_2:=\max_{O_i\in\mathcal{O}_{\mathrm{int}}}K_2(i)$, and let $\delta\in(0,\delta_{0})\subset(0,\frac14)$.  
For any $(z,k)\in I_k$ satisfying $\delta<|z-H (\mathbf{x}_i)|<2\delta$ for some interior vertex $O_i=(H (\mathbf{x}_i),k)$, we have
$$
- K_1\log(2\delta)
\le \beta_k^\delta(z)
\le - K_2\log\delta\le -\frac{\log 4}{\log 2}K_2\log(2\delta)
\le \frac{\log 4}{\log 2}\cdot\frac{K_2}{K_1}\beta_k(z).
$$
Similarly, if $|z-H (\mathbf{x}_i)|\le\delta$, then
$$
- K_1\log\delta
< \beta_k^\delta(z)
\le - K_2\log\delta\le - K_2\log|z-H (\mathbf{x}_i)|
\le \frac{K_2}{K_1}\beta_k(z).
$$
Consequently, condition \eqref{eq:delta2} holds with
$
\mathfrak{c}_3 = \min\{K_1\log 2,\inf_{(z,k)\in\Gamma}\beta_k(z)\}$, $\mathfrak{c}_5 = \frac{\log 4}{\log 2}\cdot \frac{K_2}{K_1}$, and $\mathfrak{c}_4(\delta) =\max\{-K_2\log\delta, C_2(\delta)\}$, where
$$C_2(\delta)=\sup\{\beta_k(z);(z,k)\in\Gamma,|z-H (\mathbf{x}_i)|\ge 2\delta \text{ for all interior vertices }O_i=(H(\mathbf{x}_i),k)\}.$$
The proof is completed.
\end{proof}
Due to \eqref{eq:deri-alpR}, the regularization function $\alpha^{R,\delta}$ in Example \ref{ex:reg} also satisfies Assumption \ref{asp:alphader}.
\subsection{Tightness}
We first establish an energy estimate for \eqref{eq:reguarlized}, which allows us to derive uniform bounds for the regularized solutions $\check{u}^\delta.$ These bounds, in turn, enable us to prove the convergence of $\check{u}^\delta$ by means of a compactness argument.
\begin{lemma}\label{lem:analydelta}
Let Assumptions \ref{asp:reg}--\ref{asp:reg-gamma} hold, $R> H_0$, and $\delta\in(0,\delta_0)$. Then we have the following.
\begin{enumerate}
\item[(1)] For any $\lambda\ge\frac18\kappa_1$, $-\lambda I+\check{\mathcal{L}}^\delta$ generates an analytic contraction semigroup on $L^2_{\beta^\delta\gamma}(\Gamma^R)$, where the domain $D(\check{\mathcal{L}}^{\delta})$ of $\check{\mathcal{L}}^{\delta}$ is given by
\begin{align}\label{eq:DLdel}\notag
D(\check{\mathcal{L}}^{\delta}):=\{f\in W^{1,2}_{\mathrm{cont}}(\Gamma^R): f \text{ satisfies }
\eqref{eq:gluedelta}, & ~\alpha_k^{R,\delta}\tfrac{\ud f}{\ud z}
\text{ is differentiable on each } J_k,\\
&\qquad\qquad\quad \text{ and }
\check{\mathcal{L}}^{\delta} f \in L^2_{\beta^\delta\gamma}(\Gamma^R)\}.
\end{align}

\item[(2)] If $\{1_{\Gamma},u_0\}\subset L^2_{\beta\gamma}(\Gamma)$, then the regularized truncated equation \eqref{eq:reguarlized} admits a unique mild solution $\check{u}^\delta=\{\check{u}^\delta(t)\}_{t\in[0,T]}$. Moreover, there exists $C:=C(\kappa_1,T)>0$ such that for any $R>H_0$ and $\delta\in(0,\delta_0)$,
\begin{equation}\label{eq:L2delta}
	\E\bigg[\sup_{t\in[0,T]}\|\check{u}^\delta(t)\|^2_{L^2_{\beta^\delta\gamma}(\Gamma^R)}\bigg]+\E\int_0^T\sum_{k=1}^m\int_{I_k}\alpha_k^{R,\delta}|\frac{\partial}{\partial z}\check{u}^\delta(t)|^2\gamma_k\ud z \ud t\le C\Big(1+\|u_0\|_{L^2_{\beta\gamma}(\Gamma)}^2\Big).
\end{equation}
\end{enumerate}
\end{lemma}

\begin{proof}
	(1) Fix $\delta\in(0,\delta_0)$ and $\lambda\ge\frac18\kappa_1$. Recall that the bilinear form $a_{\lambda}^{R,\delta}: W^{1,2}_{\mathrm{cont}}(\Gamma^R)\times W^{1,2}_{\mathrm{cont}}(\Gamma^R)\to \C$ defined in \eqref{eq:app}  is densely defined, continuous, closed, and accretive. Specifically, for any $\lambda > \tfrac{1}{8}\kappa_1$ (see \eqref{eq:lower} and \eqref{eq:aconti} for a similar argument),
\begin{gather}\label{eq:coe-delta}
\Re a_{\lambda}^{R,\delta}(\phi,\phi)\ge c(\lambda)\|\phi\|^2_{W^{1,2}_{\alpha^{R,\delta}\gamma,\beta^\delta\gamma}(\Gamma^R)},\\\label{eq:cont-delta}
 |a_{\lambda}^{R,\delta}(\phi,\psi)|\le C(\lambda)\|\phi\|_{W^{1,2}_{\alpha^{R,\delta}\gamma,\beta^\delta\gamma}(\Gamma^R)}\|\psi\|_{W^{1,2}_{\alpha^{R,\delta}\gamma,\beta^\delta\gamma}(\Gamma^R)}.
\end{gather}
 Moreover, the operator associated with the form $a_\lambda^{R,\delta}$ is $-\lambda I + \check{\mathcal{L}}^{\delta}$, where
$
D(\check{\mathcal{L}}^{\delta})
$
is given by \eqref{eq:DLdel}. As a consequence of \cite[Theorem 1.51]{OE05}, we obtain the first claim.

(2)
From the upper bound $\beta_k^\delta\le \mathfrak{c}_5\beta_k$ in \eqref{eq:delta2}, we infer that
$$\|\check{u}^\delta(0)\|_{L^2_{\beta^\delta\gamma}(\Gamma^R)}^2=\sum_{k=1}^{m}\int_{J_k}|\check{u}(0)|^2\beta_k^\delta\gamma_k\ud z\le C\sum_{k=1}^{m}\int_{I_k}|u_0|^2\beta_k\gamma_k\ud z\le C\Big(1+ \|u_0\|_{L^2_{\beta\gamma}(\Gamma)}^2\Big).$$
Finally, by the Lipschitz continuity of $b$ and $g$,
\eqref{eq:L2delta} can be proved in the same way as in
Lemmas~\ref{lem:L2} and~\ref{lem:H1}, and we omit the details.
\end{proof}

We are now state the smoothing properties of the regularized operator $\check{\mathcal{L}}^\delta$.
\begin{lemma}\label{lem:smo}
	Let Assumptions \ref{asp:reg}--\ref{asp:reg-gamma} hold and $\lambda>\frac18\kappa_1$. Then for any $R>H_0$ and $\delta\in(0,\delta_0)$, the operator
 $\mathcal{A}^\delta_\lambda :=\lambda I-\check{\mathcal{L}}^\delta$ satisfies
	\begin{gather}\label{eq:smo1}
		\|(\mathcal{A}^\delta_\lambda )^{\vartheta } e^{-t\mathcal{A}^\delta_\lambda }\|_{\mathscr{L}(L^2_{\beta^\delta\gamma}(\Gamma^R))}\le C(\lambda,\vartheta ) t^{-\vartheta },\quad \vartheta >0,\\\label{eq:smo2}
		\|(\mathcal{A}^\delta_\lambda )^{-\vartheta }(e^{-t\mathcal{A}^\delta_\lambda }-I)\|_{\mathscr{L}(L^2_{\beta^\delta\gamma}(\Gamma^R))}\le C(\lambda,\vartheta )t^{\vartheta },\quad \vartheta \in (0,1).
	\end{gather}
\end{lemma}
\begin{proof}
For a fixed  $\lambda > \tfrac{1}{8}\kappa_1$,
by \eqref{eq:coe-delta} and \eqref{eq:cont-delta}, we obtain   
\begin{align*}
	\big|\langle \mathcal{A}^\delta_\lambda \phi,\phi\rangle_{L^2_{\beta^\delta\gamma}(\Gamma^R)}\big|\le \frac{C(\lambda)}{c(\lambda)}\Re \langle \mathcal{A}^\delta_\lambda \phi,\phi\rangle_{L^2_{\beta^\delta\gamma}(\Gamma^R)},\quad \phi\in D(\mathcal{A}^\delta_\lambda)=D(\check{\mathcal{L}}^\delta).
\end{align*}
Therefore, the numerical range $\mathcal{N}(\mathcal{A}^\delta_\lambda )$ of $\mathcal{A}^\delta_\lambda $ (see \cite[Definition 1.26]{OE05}), defined by
$$\mathcal{N}(\mathcal{A}^\delta_\lambda ):=\Big\{\langle \mathcal{A}^\delta_\lambda \phi,\phi\rangle_{L^2_{\beta^\delta\gamma}(\Gamma^R)}:\phi\in D(\mathcal{A}^\delta_\lambda ), \|\phi\|_{L^2_{\beta^\delta\gamma}(\Gamma^R)}= 1 \Big\},
$$
is a subset of the sector $\Sigma_0(\theta_0):=\{y\in\mathbb{C},y\neq 0,0\le | \arg y|\le \theta_0\}$ where $\theta_0=\arccos(c(\lambda)/C(\lambda))\in(0,\frac{\pi}{2})$. Fix $\theta_1\in (\theta_0,\frac{\pi}{2})$ and denote $\Sigma_\pi(\theta_1):=\{y\in\mathbb{C},\theta_1\le | \arg y|\le \pi\}$. Then by \cite[Proposition C.3.1]{HM06}, the resolvent set $\rho(\mathcal{A}^\delta_\lambda )$ of $\mathcal{A}^\delta_\lambda $ contains $\Sigma_\pi(\theta_1)$, and for any $y\in \Sigma_\pi(\theta_1)$,
\begin{equation}\label{eq:analytic}
	\|(y I-\mathcal{A}^\delta_\lambda )^{-1}\|_{\mathscr{L}(L^2_{\beta^\delta\gamma}(\Gamma^R))}\le \frac{1}{\textup{dist}(y,\overline{\mathcal{N}(\mathcal{A}^\delta_\lambda )})}\le\frac{1}{\textup{dist}(y,\Sigma_0(\theta_0))}\le \frac{1}{\sin(\theta_1-\theta_0)}\frac{1}{|y|}.
\end{equation}
If we introduce the contour $\mathcal{Y}=\{y:y=r e^{\pm \mathbf{i}\theta_1},r\ge0\}$, then it follows from \eqref{eq:analytic} that for any  $\vartheta >0$,
\begin{align*}
	\|(\mathcal{A}^\delta_\lambda )^{\vartheta } e^{-t\mathcal{A}^\delta_\lambda }\phi\|_{L^2_{\beta^\delta\gamma}(\Gamma^R)}&=\left\|\frac{1}{2\pi \mathbf{i}}\int_{\mathcal{Y}} y^\vartheta  e^{-yt}(yI-\mathcal{A}^\delta_\lambda )^{-1}\phi\ud  y\right\|_{L^2_{\beta^\delta\gamma}(\Gamma^R)}\\\notag
	&\le C\|\phi\|_{L^2_{\beta^\delta\gamma}(\Gamma^R)}\int_0^\infty r^{\vartheta -1} e^{-r t\cos\theta_1}\ud r\le C t^{-\vartheta }\|\phi\|_{L^2_{\beta^\delta\gamma}(\Gamma^R)},
\end{align*}
where $C:=C(\vartheta,\lambda)>0$ is independent of $\phi\in L^2_{\beta^\delta\gamma}(\Gamma^R)$. 
This completes the proof of \eqref{eq:smo1}. Finally, \eqref{eq:smo2} comes from 
$(I-e^{-t\mathcal{A}^\delta_\lambda })\phi=\int_0^t\mathcal{A}^\delta_\lambda e^{-s\mathcal{A}^\delta_\lambda }\phi\ud s$ and an application of \eqref{eq:smo1}.
\end{proof}
To facilitate the a-priori estimate of $\check{u}^\delta$, we 
 introduce the following random PDE
\begin{align}\label{eq:vr}
	\partial_tv^\delta(t)=\check{\mathcal{L}}^\delta v^\delta(t)+\check{B}(\check{u}^\delta(t)),\quad  v^\delta(0)=\check{u}(0),
\end{align}
for any $\delta\in(0,\delta_0)$. Then $r^\delta:=\check{u}^\delta-v^\delta$ is the stochastic convolution given by 
\begin{align*}
r^\delta(t)=\int_0^t e^{(t-s)\check{\mathcal{L}}^\delta}\check{G}(\check{u}^\delta(s))\ud W(s),\quad t\in[0,T].
\end{align*}

\begin{lemma}\label{lem:uHolder}
	Let Assumptions \ref{asp:reg}--\ref{asp:reg-gamma} hold, $1_\Gamma\in L^2_{\beta\gamma}(\Gamma)$, and $u_0\in W^{1,2}_{(\alpha+1)\gamma,\beta\gamma}(\Gamma)$. 	
	Then there exists a constant $C>0$ such that for any $R>H_0$ and $\delta\in(0,\delta_0)$,
			\begin{align}\label{eq:vunif}
		\sup_{t\in[0,T]}\|v^\delta(t)\|_{W^{1,2}_{\alpha^{R}\gamma,\beta\gamma}(\Gamma^R)}+ \int_0^T\|\partial_sv^\delta(s)\|_{L^2_{\gamma}(\Gamma^R)}^2 \ud s\le C.
	\end{align}
\end{lemma}
\begin{proof}
	By $\alpha_k^{R,\delta}\le \mathfrak{c}_2(\alpha_k+1)$ and $\beta_k^\delta\le \mathfrak{c}_5\beta_k$ in Assumption \ref{asp:reg}, for all $\delta\in(0,\delta_0)$ we have
	\begin{align}\label{eq:Qv0}
		\|v^\delta(0)\|_{W^{1,2}_{\alpha^{R,\delta}\gamma,\beta^\delta\gamma}(\Gamma^R)}=\|\check{u}(0)\|_{W^{1,2}_{\alpha^{R,\delta}\gamma,\beta^\delta\gamma}(\Gamma^R)}\le C\|u_0\|_{ W^{1,2}_{(\alpha+1)\gamma,\beta\gamma}(\Gamma)}.
	\end{align}
	For $ t\in[0,T]$, we denote $$\mathcal{Q}(v^\delta(t)):=\frac12\sum_{k=1}^m
	\int_{J_k}\alpha_k^{R,\delta}(z)|\frac{\partial}{\partial z}v^\delta(t,z,k)|^2
	\gamma_k(z)\ud z.$$
	From \eqref{eq:vr}, Assumption \ref{asp:reg-gamma}, Young's inequality, and the linear growth of $b$, it holds  
	\begin{align*}
		\frac{\ud}{\ud t}\mathcal{Q}(v^\delta(t))&=-\sum_{k=1}^m
		\int_{J_k}\alpha_k^{R,\delta}(z)\frac{\partial}{\partial z}v^\delta(t)\partial_tv^\delta(t)
		\gamma_k'(z)\ud z
		-2\langle \check{\mathcal{L}}^\delta v^\delta(t),\partial_tv^\delta(t)\rangle_{L^2_{\beta^\delta\gamma}(\Gamma^R)}\\
		&=\frac{1}{2}\kappa_1\mathcal{Q}(v^\delta(t))
		+\|\partial_tv^\delta(t)\|_{L^2_{\beta^\delta\gamma}(\Gamma^R)}^2-2\langle\partial_tv^\delta(t)-\check{B}(\check{u}^\delta(t)),\partial_tv^\delta(t)\rangle_{L^2_{\beta^\delta\gamma}(\Gamma^R)}\\
		&\le \frac{1}{2}\kappa_1\mathcal{Q}(v^\delta(t)) -\frac12\|\partial_tv^\delta(t)\|_{L^2_{\beta^\delta\gamma}(\Gamma^R)}^2+C\Big(1+\|\check{u}^\delta(t)\|_{L^2_{\beta^\delta\gamma}(\Gamma^R)}^2\Big).
	\end{align*}
	As a result of \eqref{eq:Qv0} and Lemma \ref{lem:analydelta}, we can readily apply the Gr\"onwall inequality to obtain 
	\begin{align}\label{eq:vreg}
\mathcal{Q}(v^\delta(t))+ \int_0^T\|\partial_sv^\delta(s)\|_{L^2_{\beta^\delta\gamma}(\Gamma^R)}^2 \ud s\le C\Big(1+\|u_0\|_{ W^{1,2}_{(\alpha+1)\gamma,\beta\gamma}(\Gamma)}^2\Big),\quad t\in[0,T].
	\end{align}
		According to \eqref{eq:vreg}, for any $0\le s\le t\le T$,
	\begin{equation}\label{eq:vHolderreg}
		\|v^\delta(t)-v^\delta(s)\|_{L^2_{\beta^\delta\gamma}(\Gamma^R)}^2\le (t-s)\int_s^t \|\partial_{s_1}v^\delta(s_1)\|_{L^2_{\beta^\delta\gamma}(\Gamma^R)}^2\ud s_1\le C(t-s).
	\end{equation}
	A combination of \eqref{eq:vreg} and \eqref{eq:vHolderreg} with $s=0$ proves
		\begin{align}\label{eq:vdeltadelta}
		\|v^\delta(t)\|_{W^{1,2}_{\alpha^{R,\delta}\gamma,\beta^\delta\gamma}(\Gamma^R)}+ \int_0^T\|\partial_sv^\delta(s)\|_{L^2_{\beta^\delta\gamma}(\Gamma^R)}^2 \ud s\le C.
	\end{align}

To prove \eqref{eq:vunif}, recall that we have split each edge $J_k$ into two parts $J_k^{(1)}$ and $J_k^{(2)}$, where $\{J_k^{(l)}\}_{l=1}^2$ connects exactly one vertex $O_i=(H (\mathbf{x}_i),k)$ of $\Gamma^R$. 
\begin{enumerate}
\item[Case 1.] If $O_i$ is an exterior vertex, then $\beta_k^\delta=\beta_k$ on $J_k^{(l)}$, which implies 
\begin{align}\label{eq:saddlesimL2}
			\|f\|_{L^2_{\beta\gamma}(J_k^{(l)})}^2= \|f\|_{L^2_{\beta^\delta\gamma}(\Gamma^R)},\quad  f\in W^{1,2}_{\alpha^{R,\delta}\gamma,\beta^\delta\gamma}(\Gamma^R).
		\end{align}

\item[Case 2.] If $O_i$ is an interior vertex, then $\alpha_k^{R,\delta}=\alpha_k^R$ is uniformly bounded from below and above by positive constants on $J_k^{(l)}$. 
			Since $\gamma$ is continuous on $\Gamma$,  for any $(z,k)\in\Gamma^R$ with $z\in[0, H _0+1]$,
		\begin{equation}\label{eq:gammalowupp-H}
		C^{-1}\le \gamma_k(z)\le C.
		\end{equation}
		Assume without loss of generality that the left endpoint of $J_k^{(l)}$ corresponds to the interior vertex $O_i$; otherwise the integral domain $[H (\mathbf{x}_i),H (\mathbf{x}_i)+|J_k^{(l)}|]$ in \eqref{eq:saddleL2} is replaced by $[H (\mathbf{x}_i)-|J_k^{(l)}|,H (\mathbf{x}_i)]$.
		Then, by $\beta_k^\delta\ge \mathfrak{c}_3>0$, the Sobolev embedding $ W^{1,2}(J_k^{(l)})\hookrightarrow L^\infty(J_k^{(l)})$, \eqref{eq:T}, and \eqref{eq:gammalowupp-H}, we obtain
		\begin{align}\label{eq:saddleL2}
			&\|f\|_{L^2_{\beta\gamma}(J_k^{(l)})}^2=\int_{H (\mathbf{x}_i)}^{H (\mathbf{x}_i)+|J_k^{(l)}|}|f(z,k)|^2\beta_k\gamma_k\ud z\\\notag
			&\le \|f(\cdot,k)\|_{L^2(J_k^{(l)})}\|f(\cdot,k)\|_{L^\infty(J_k^{(l)})}\bigg(\int_{H (\mathbf{x}_i)}^{H (\mathbf{x}_i)+|J_k^{(l)}|}|\beta_k\gamma_k|^2\ud z\bigg)^{\frac12}\\\notag
			&\le C \|f\|_{L^2_{\beta^\delta\gamma}(J_k^{(l)})}\|f\|_{W^{1,2}_{\alpha^{R,\delta}\gamma,\beta^\delta\gamma}(J_k^{(l)})}\bigg(\int_{H (\mathbf{x}_i)}^{H (\mathbf{x}_i)+|J_k^{(l)}|}(1+|\log|z-H (\mathbf{x}_i)||^2)\gamma_k^2\ud z\bigg)^{\frac12}\\\notag
&\le  C\|f\|_{L^2_{\beta^\delta\gamma}(\Gamma^R)}\|f\|_{W^{1,2}_{\alpha^{R,\delta}\gamma,\beta^\delta\gamma}(\Gamma^R)},\quad  f\in W^{1,2}_{\alpha^{R,\delta}\gamma,\beta^\delta\gamma}(\Gamma^R).
		\end{align}
\end{enumerate}	
Combining \eqref{eq:saddlesimL2} and \eqref{eq:saddleL2} in the above two cases allows us to conclude
	\begin{align}\label{eq:L2beta}
		\|f\|_{L^2_{\beta\gamma}(\Gamma^R)}^2\le C \|f\|_{L^2_{\beta^\delta\gamma}(\Gamma^R)}\|f\|_{W^{1,2}_{\alpha^{R,\delta}\gamma,\beta^\delta\gamma}(\Gamma^R)},
	\end{align}
	for some constant $C>0$ independent of $ f\in W^{1,2}_{\alpha^{R,\delta}\gamma,\beta^\delta\gamma}(\Gamma^R)$.
	Taking further the assumption $\alpha_k^{R,\delta}\ge \mathfrak{c}_0\alpha_k^R$ and \eqref{eq:L2beta} into account, we arrive at
	\begin{equation}\label{eq:fW}
		\|f\|_{W_{\alpha^R\gamma,\beta\gamma}^{1,2}(\Gamma^R)}\le C\|f\|_{W^{1,2}_{\alpha^{R,\delta}\gamma,\beta^\delta\gamma}(\Gamma^R)},\quad f\in W^{1,2}_{\alpha^{R,\delta}\gamma,\beta^\delta\gamma}(\Gamma^R).
	\end{equation}
	Since $\beta_k^\delta\ge \mathfrak{c}_3$, we can combine \eqref{eq:fW} and \eqref{eq:vdeltadelta} to finish the proof of \eqref{eq:vunif}.
\end{proof}

Given a Banach space $(X,\|\cdot\|_X)$, for any $\vartheta\in(0,1)$, we denote the seminorm of $\vartheta$-Hölder continuous functions $x: [0,T] \to X$ by
$$|x|_{\mathcal{C}^\vartheta([0,T];X)}:=\sup_{0\le s<t\le T}\frac{\|x(t)-x(s)\|_X}{|t-s|^\vartheta}.$$

\begin{lemma}\label{lem:rHolder}	
	Under the conditions of Lemma \ref{lem:uHolder}, for any $p>2$, $\vartheta\in(0,\frac14)$, and $R> H _0$, there exists a constant $C:=C(R, p,\vartheta)>0$ such that for any $\delta\in(0,\delta_0)$,
		\begin{equation*}
		\E\left[|r^\delta|_{\mathcal{C}^\vartheta([0,T];L_{\beta\gamma}^2(\Gamma^R))}^p\right]\le C.
	\end{equation*}
	
\end{lemma}
\begin{proof}

	The proof is divided into three steps.

\emph{Step 1}.
Let $\lambda>\frac{1}{8}\kappa_1$. If we denote $ r^\delta_\lambda (t):=e^{-\lambda t}r^\delta(t)$ , then
	\begin{equation}\label{eq:Rdef}
		r^\delta_\lambda (t)=\int_0^t e^{-\mathcal{A}^\delta_\lambda  (t-s)} e^{-\lambda s}\check{G}(\check{u}^\delta(s))\ud W(s),\quad t\in[0,T],
	\end{equation} 
	where $\mathcal{A}^\delta_\lambda=\lambda I-\check{\mathcal{L}}^\delta$.
	For any $0\le s<t\le T$,
	\begin{align*}
		r^\delta_\lambda (t)- r^\delta_\lambda (s)&=\int_0^s (\mathcal{A}^\delta_\lambda )^{\vartheta }e^{-\mathcal{A}^\delta_\lambda  (s-s_1)}(\mathcal{A}^\delta_\lambda )^{-\vartheta }(e^{-\mathcal{A}^\delta_\lambda  (t-s)} -I)e^{-\lambda s_1}\check{G}(\check{u}^\delta(s_1))\ud W(s_1)\\
		&\quad+\int_s^t e^{-\mathcal{A}^\delta_\lambda  (t-s_1)} e^{-\lambda s_1}\check{G}(\check{u}^\delta(s_1))\ud W(s_1).
	\end{align*}
	Then, invoking \eqref{eq:smo1}, \eqref{eq:smo2}, the contraction of $\{e^{-\mathcal{A}^\delta_\lambda  t}\}_{t\ge0}$ on $L^2_{\beta^\delta\gamma}(\Gamma^R)$, and the Burkholder inequality \cite{DZ14}, we infer from the linear growth of $g$ and Lemma \ref{lem:analydelta} that for any $\vartheta \in(0,\frac12)$ and any $p\ge2$, 
	\begin{align}\label{eq:rdeltaHolder}
		\E\left[\|r^\delta_\lambda(t)- r^\delta_\lambda (s)
		\|_{L^2_{\beta^\delta\gamma}(\Gamma^R)}^p\right]
		&\le C(\lambda,\vartheta,p)(t-s)^{\vartheta p}\bigg(1+\E\Big[\sup_{s_1\in[0,T]}
		\|\check{u}^\delta(s_1)\|_{L^p_{\beta^\delta\gamma}(\Gamma^R)}^p\Big]\bigg)\\\notag
		&\le C(\vartheta,\lambda,p)(t-s)^{\vartheta p},\quad0\le s\le t\le T.
	\end{align} 
Because $r^\delta(t)=e^{\lambda t}r^\delta_\lambda(t)$, it can be shown that for any $\vartheta \in(0,\frac12)$ and any $p\ge2$.
	\begin{align}\label{eq:rreg}
		\E\left[\|r^\delta(t)-r^\delta(s)\|_{L^2_{\beta^\delta\gamma}(\Gamma^R)}^p\right]\le C(\vartheta,\lambda,p)(t-s)^{\vartheta p},\quad 0\le s\le t\le T.
	\end{align}
	
	\emph{Step 2.}
 By Lemma \ref{lem:analydelta}, $-\mathcal{A}_\lambda^\delta$ generates a contraction semigroup on $L_{\beta^\delta}^2(\Gamma^R)$, which implies that the assumption of \cite[Proposition A.19]{DZ14} holds (see \cite[Remark A.20]{DZ14}). Therefore, by applying \cite[Proposition A.21]{DZ14}, there exists a constant $C>0$ such that for any $t\in[0,T]$,
	\begin{align}\label{eq:criti}
		\int_0^t\|(\mathcal{A}^\delta_\lambda )^{\frac12}e^{-s\mathcal{A}^\delta_\lambda }\phi\|_{L_{\beta^\delta \gamma}^2(\Gamma^R)}^2\ud s
		\le C\| \phi\|_{L_{\beta^\delta\gamma}^2(\Gamma^R)}^2,\quad \phi\in L_{\beta^\delta \gamma}^2(\Gamma^R).
	\end{align}
	For the stochastic process $ r^\delta_\lambda $ defined in \eqref{eq:Rdef}, applying the Burkholder inequality leads to 
	\begin{align*}
		\E\left[\|(\mathcal{A}^\delta_\lambda )^{\frac12}r^\delta_\lambda(t)\|_{L_{\beta^\delta\gamma}^2(\Gamma^R)}^p\right]\le C\E\left[\left(\int_0^t \|(\mathcal{A}^\delta_\lambda )^{\frac12}e^{-\mathcal{A}^\delta_\lambda  (t-s)} e^{-\lambda s}\check{G}(\check{u}^\delta(s))\|_{\mathscr{L}_2(\mathcal{U}_0,L_{\beta^\delta \gamma}^2(\Gamma^R))}^2\ud s\right)^{\frac{p}{2}}\right].
	\end{align*}
	Moreover,
	owing to \eqref{eq:smo1} and \eqref{eq:criti}, we have
	\begin{align*}
		&\int_0^t\|(\mathcal{A}^\delta_\lambda )^{\frac12}e^{-\mathcal{A}^\delta_\lambda  (t-s)} e^{-\lambda s}\check{G}(\check{u}^\delta(s))\|_{\mathscr{L}_2(\mathcal{U}_0,L_{\beta^\delta\gamma}^2(\Gamma^R))}^2\ud s\\
		&\le 2\int_0^t \big\|(\mathcal{A}^\delta_\lambda )^{\frac12}e^{-\mathcal{A}^\delta_\lambda  (t-s)} e^{-\lambda s}\big(\check{G}(\check{u}^\delta(s))-\check{G}(\check{u}^\delta(t))\big)\big\|_{\mathscr{L}_2(\mathcal{U}_0,L_{\beta^\delta\gamma}^2(\Gamma^R))}^2\ud s\\
		&\quad+2\int_0^t \|(\mathcal{A}^\delta_\lambda )^{\frac12}e^{-\mathcal{A}^\delta_\lambda  (t-s)} e^{-\lambda s}\check{G}(\check{u}^\delta(t))\|_{\mathscr{L}_2(\mathcal{U}_0,L_{\beta^\delta\gamma}^2(\Gamma^R))}^2\ud s\\
		&\le C\int_0^t (t-s)^{-1} \|\check{G}(\check{u}^\delta(s))-\check{G}(\check{u}^\delta(t))\|_{\mathscr{L}_2(\mathcal{U}_0,L_{\beta^\delta\gamma}^2(\Gamma^R))}^2\ud s+C\|\check{G}(\check{u}^\delta(t))\|_{\mathscr{L}_2(\mathcal{U}_0,L_{\beta^\delta\gamma}^2(\Gamma^R))}^2.
	\end{align*}
	Besides, utilizing the Lipschitz continuity of $g$, it holds that 
	\begin{align*}
		&\|\check{G}(\check{u}^\delta(s))-\check{G}(\check{u}^\delta(t))\|_{\mathscr{L}_2(\mathcal{U}_0,L_{\beta^\delta\gamma}^2(\Gamma^R))}\le C\|\check{u}^\delta(s)-\check{u}^\delta(t)\|_{L_{\beta^\delta\gamma}^2(\Gamma^R)}\\&\le \|r^\delta(s)-r^\delta(t)\|_{L_{\beta^\delta\gamma}^2(\Gamma^R)}+\|v^\delta(s)-v^\delta(t)\|_{L_{\beta^\delta\gamma}^2(\Gamma^R)}.
	\end{align*}
	Together with \eqref{eq:vHolderreg} and \eqref{eq:rreg}, we deduce that for any $p\ge2$ and $\lambda>\frac{1}{8}\kappa_1$,
	\begin{equation}\label{eq:rspatialreg}
	 \E\left[\|(\mathcal{A}^\delta_\lambda )^{\frac12}r^\delta_\lambda(t)\|_{L_{\beta^\delta\gamma}^2(\Gamma^R)}^p\right]\le C(p,\lambda),\quad 0\le t\le T.
	\end{equation}
	
	\emph{Step 3.}
		In view of \eqref{eq:L2beta}, for any $0\le s\le t\le T$ we have
	\begin{align*}
		\|r^\delta(t)-r^\delta(s)\|_{L^2_{\beta\gamma}(\Gamma^R)}^2\le C \|r^\delta(t)-r^\delta(s)\|_{L^2_{\beta^\delta\gamma}(\Gamma^R)}\Big(\|r^\delta(t)\|_{W^{1,2}_{\alpha^{R,\delta}\gamma,\beta^\delta\gamma}(\Gamma^R)}+\|r^\delta(s)\|_{W^{1,2}_{\alpha^{R,\delta}\gamma,\beta^\delta\gamma}(\Gamma^R)}\Big).
	\end{align*}
	This, combined with \eqref{eq:rreg} and H\"older's inequality, yields that for any $\vartheta\in(0,\frac14)$ and any $p\ge2$,
		\begin{equation}\label{eq:rdeltaHolderbeta}
		\E\left[\|r^\delta(t)-r^\delta(s)\|_{L^2_{\beta\gamma}(\Gamma^R)}^{p}\right]
		\le C(R)(t-s)^{\vartheta p} \sup_{t\in[0,T]}\left(\E\left[\|r^\delta(t)\|_{W^{1,2}_{\alpha^{R,\delta}\gamma,\beta^\delta\gamma}(\Gamma^R)}^p\right]\right)^{\frac{1}2}.
	\end{equation}
By continuity of $\gamma$, for any $R> H _0$ we have 
	\begin{equation}\label{eq:gammalowupp}
	c_\gamma(R):=\inf_{\Gamma^R}\gamma \le \gamma_k(z)\le 	\sup_{\Gamma^R}\gamma=:C_\gamma(R)\quad \forall~(z,k)\in \Gamma^R.
	\end{equation} 
	In addition, for any $\phi\in W^{1,2}_{\alpha^{R,\delta},\beta^\delta}(\Gamma^R)$,
	\begin{align}\label{eq:self-adj}
		\|(\mathcal{A}^\delta_\lambda )^{\frac12}\phi\|_{L_{\beta^\delta}^2(\Gamma^R)}^2=\langle \mathcal{A}^\delta_\lambda \phi,\phi\rangle_{L_{\beta^\delta}^2(\Gamma^R)}
		=\lambda \sum_{k=1}^m\int_{J_k}|\phi|^2\beta_k^\delta\ud z+\frac12\sum_{k=1}^m\int_{J_k}\alpha_{k}^{R,\delta}|\frac{\ud }{\ud z}\phi|^2\ud z,
	\end{align}
	and thus the fractional norm $\|(\mathcal{A}^\delta_\lambda )^{\frac12}\cdot\|_{L_{\beta^\delta}^2(\Gamma^R)}$ is an equivalent norm of $\|\cdot\|_{W^{1,2}_{\alpha^{R,\delta},\beta^\delta}(\Gamma^R)}$.
	Combining \eqref{eq:gammalowupp}, \eqref{eq:self-adj}, and the relation $ r^\delta_\lambda (t)=e^{-\lambda t}r^\delta(t)$, we arrive at
		\begin{align}\label{eq:rW12}\notag
&\E\left[\|r^\delta(t)\|_{W^{1,2}_{\alpha^{R,\delta}\gamma,\beta^\delta\gamma}(\Gamma^R)}^p\right]\le C(R)\E\left[\|r^\delta(t)\|_{W^{1,2}_{\alpha^{R,\delta},\beta^\delta}(\Gamma^R)}^p\right]\le C(R)\E\left[\|r^\delta_\lambda(t)\|_{W^{1,2}_{\alpha^{R,\delta},\beta^\delta}(\Gamma^R)}^p\right]\\
	&\le C(R)\E\left[\|(\mathcal{A}^\delta_\lambda )^{\frac12}r^\delta_\lambda(t)\|_{L_{\beta^\delta}^2(\Gamma^R)}^p\right]
	\le C(R)\E\left[\|(\mathcal{A}^\delta_\lambda )^{\frac12}r^\delta_\lambda(t)\|_{L_{\beta^\delta\gamma}^2(\Gamma^R)}^p\right]\le C,
	\end{align}
	where the last step is due to \eqref{eq:rspatialreg}.
	Substituting the above inequality \eqref{eq:rW12} into \eqref{eq:rdeltaHolderbeta} yields that for any $\vartheta\in(0,\frac14)$ and any $p\ge2$,
		\begin{equation*}
		\E\left[\|r^\delta(t)-r^\delta(s)\|_{L^2_{\beta\gamma}(\Gamma^R)}^{p}\right]
		\le C(R,p)(t-s)^{\vartheta p}.
	\end{equation*}
	Finally, we apply the Garsia--Rodemich--Rumsey inequality \cite[Theorem 2 \& Corollary 1]{GM18} to complete the proof. 
\end{proof}
On account of Lemma \ref{lem:rHolder} and \eqref{eq:rW12}, for any $p\ge2$ and $\vartheta\in(0,\frac{1}{4})$ we have
\begin{equation}\label{eq:rdeltaee}
\sup_{t\in[0,T]}\E\left[\|r^\delta(t)\|_{W^{1,2}_{\alpha^{R}\gamma,\beta\gamma}(\Gamma^R)}^p\right]+\E\left[|r^\delta|_{\mathcal{C}^\vartheta([0,T];L_{\beta\gamma}^2(\Gamma^R))}^p\right]\le C(R,p),
\end{equation}
which, along with Lemma \ref{lem:uHolder}, implies that for any $p\ge2$,
\begin{equation}\label{eq:udeltaee}
\sup_{t\in[0,T]}\E\left[\|\check{u}^\delta(t)\|_{W^{1,2}_{\alpha^{R}\gamma,\beta\gamma}(\Gamma^R)}^p\right]\le C(R,p).
\end{equation}
\begin{lemma}\label{lem:tight}
		Let \eqref{eq:etaR} and Assumptions \ref{asp:reg}--\ref{asp:reg-gamma} hold. Assume that $1_\Gamma\in L^2_{\beta\gamma}(\Gamma)$ and $u_0\in W^{1,2}_{(\alpha+1)\gamma,\beta\gamma}(\Gamma)$.	
	Then for any fixed $R>H_0$, the law of $\{\check{u}^\delta\}_{\delta\in(0,\delta_0)}$ is tight in $\mathcal{C}([0,T];L^2_{\beta\gamma}
	(\Gamma^R))$.
\end{lemma}
\begin{proof}
Due to \eqref{eq:alphaR}, \eqref{eq:A}, and \eqref{eq:T},
	it follows from \cite[Claim 1 in Lemma 5.6]{CS25} that the embedding $W_{\alpha^R\gamma,\beta\gamma}^{1,2}(\Gamma^R)\hookrightarrow L^2_{\beta\gamma}(\Gamma^R)$ is compact. Then due to $L^2_{\beta\gamma}(\Gamma^R)\subset L^2_{\gamma}(\Gamma^R)$, applying Aubin--Lions theorem implies that for any $M>0$, the set
\begin{align*}
	\mathfrak{V}_M:=\left\{\phi\in \mathcal{C}([0,T];L^2_{\beta\gamma}(\Gamma^R)),\|\phi\|_{L^\infty(0,T;W_{\alpha^R\gamma,\beta\gamma}^{1,2}(\Gamma^R))}+\|\partial_t \phi\|_{L^2(0,T;L^2_{\gamma}(\Gamma^R))}\le M\right\}
\end{align*}
is a pre-compact subset of $\mathcal{C}([0,T];L^2_{\beta\gamma}(\Gamma^R))$.

By \cite[Theorem 1]{SJ87}, a subset $\mathfrak{F}\subset\mathcal{C}([0,T];L^2_{\beta\gamma}
(\Gamma^R))$ is relatively compact in $\mathcal{C}([0,T];L^2_{\beta\gamma}
(\Gamma^R))$ if the following two conditions hold:
\begin{align*}
\left\{	\int_{t_1}^{t_2}f(t)\ud t:f\in\mathfrak{F}\right\}~\text{is relatively compact in 
	$L^2_{\beta\gamma}
	(\Gamma^R)$}\quad\forall~0<t_1<t_2<T,\\
\|f(\cdot+\tau)-f(\cdot)\|_{L^\infty(0,T-\tau;L^2_{\beta\gamma}
		(\Gamma^R))}\to 0\quad\text{as}~\tau\to0,\quad\text{uniformly for } f\in\mathfrak{F}.
\end{align*}
Recall that  $W_{\alpha^R\gamma,\beta\gamma}^{1,2}(\Gamma^R)$ is compactly embedded into $ L^2_{\beta\gamma}(\Gamma^R)$. 
Hence for any $M>0$ and $\vartheta>0$, the set 
\begin{align*}
	\mathfrak{K}_M:=\left\{\phi\in \mathcal{C}([0,T];L^2_{\beta\gamma}
	(\Gamma^R));\|\phi\|_{L^1(0,T;W^{1,2}_{\alpha^{R}\gamma,\beta\gamma}(\Gamma^R))}
	+|\phi|_{\mathcal{C}^\vartheta([0,T];L_{\beta\gamma}^2(\Gamma^R))}\le M\right\}
\end{align*}
is a pre-compact subset of $\mathcal{C}([0,T];L^2_{\beta\gamma}
(\Gamma^R))$.
For any $M>0$, the sum $\mathfrak{K}_M+\mathfrak{V}_M:=\{x+y,x\in \mathfrak{K}_M,y\in \mathfrak{V}_M\}$ is a pre-compact subset of $\mathcal{C}([0,T];L^2_{\beta\gamma}
(\Gamma^R))$. By the Markov inequality, \eqref{eq:rdeltaee} and Lemma \ref{lem:uHolder}, it holds that
\begin{align*}
	&\mathbb{P}\left\{\check{u}^\delta\in\mathfrak{K}_M+\mathfrak{V}_M \right\}\ge 	\mathbb{P}\left\{r^\delta\in\mathfrak{K}_M,v^\delta\in\mathfrak{V}_M \right\}\\
	&\ge1-\mathbb{P}\left\{r^\delta\notin\mathfrak{K}_M\right\}
	-\mathbb{P}\left\{v^\delta\notin\mathfrak{V}_M \right\}\ge 1-C(R)M^{-1},
	\end{align*}
	which proves the tightness of the law of $\check{u}^\delta$ in $\mathcal{C}([0,T];L^2_{\beta\gamma}
	(\Gamma^R))$.
\end{proof}

\subsection{Convergence}
In this subsection, we prove the convergence of the truncated regularized approximation \eqref{eq:reguarlized} to the truncated problem \eqref{eq:uR} as $\delta \to 0$. The convergence analysis is based on the tightness argument in Lemma \ref{lem:tight}, combined with the Gyöngy–Krylov characterization of convergence in probability stated below (see also \cite[Lemma 1.1]{GK96}).
\begin{lemma}[Gyöngy--Krylov]\label{lem:GK} Let $\{Y_n\}_{n\in\mathbb{N}}$ be a sequence of $X$-valued random variables, where $X$ is a complete separable metric space. Then $\{Y_n\}_{n\in\mathbb{N}}$ converges in probability if and only if for every two subsequences $\{Y_{n_i}\}_{i\in\mathbb{N}}$ and $\{Y_{l_i}\}_{i\in\mathbb{N}}$, the joint sequence $\{(Y_{n_i}, Y_{l_i})\}_{i\in\mathbb{N}}$ has a further subsequence whose laws converge weakly to a probability measure $\nu$ supported on the diagonal of $X \times X$:
	$$
	\nu(\{(x, y) \in X \times X: x=y\})=1 .
	$$
\end{lemma}

By \eqref{eq:reguarlized} and the integration by parts formula, for any real valued function $\varphi\in \oplus_{k=1}^m \mathcal{C}_c^\infty(J_k)$,
	\begin{align*}
	\langle \check{u}^\delta(t),\varphi\rangle_{L^2_{\beta^\delta\gamma}(\Gamma^R)}=\langle \check{u}(0),\varphi\rangle_{L^2_{\beta^\delta\gamma}(\Gamma^R)}
	-\frac12\int_0^t\sum_{k=1}^m\int_{J_k} \alpha_k^{R,\delta}\frac{\partial}{\partial z}\check{u}^\delta(s) \frac{\ud}{\ud z}(\varphi\gamma_k)\ud z\ud s\\ + \int_0^t \langle \check{B}(\check{u}^\delta(s)),\varphi\rangle_{L^2_{\beta^\delta\gamma}(\Gamma^R)}\ud s+\int_0^t\langle \check{G}(\check{u}^\delta(s))\ud W(s),\varphi\rangle_{L^2_{\beta^\delta\gamma}(\Gamma^R)}.
	\end{align*}
To identify the limiting equation, we reformulate
	\begin{align}\label{eq:udelta}
		\langle \check{u}^\delta(t),\varphi\rangle_{L_{\beta\gamma}^2(\Gamma^R)}=\langle \check{u}(0),\varphi\rangle_{L_{\beta\gamma}^2(\Gamma^R)}
	-\frac12\int_0^t\sum_{k=1}^m\int_{J_k} \alpha_k^{R}\frac{\partial}{\partial z}\check{u}^\delta(s) \frac{\ud}{\ud z}(\varphi\gamma_k)\ud z\ud s\\ \notag+ \int_0^t \langle \check{B}(\check{u}^\delta(s)),\varphi\rangle_{L_{\beta\gamma}^2(\Gamma^R)}\ud s
	+\int_0^t\langle \check{G}(\check{u}^\delta(s))\ud W(s),\varphi\rangle_{L_{\beta\gamma}^2(\Gamma^R)}+\mathrm{R}^{\varphi,\delta}(t;\check{u}^\delta),
	\end{align}
	where the remainder term is given by
	\begin{align*}
		\mathrm{R}^{\varphi,\delta}(t;\check{u}^\delta)&:=\sum_{k=1}^m\int_{J_k}\big(\check{u}^\delta(t)-\check{u}(0)\big)\varphi(\beta_k-\beta_k^\delta)\gamma_k\ud z\\
		&\quad+\frac12\int_0^t\sum_{k=1}^m\int_{J_k} \big(\alpha_k^R-\alpha_k^{R,\delta}\big)\frac{\partial}{\partial z}\check{u}^\delta(t)\frac{\ud}{\ud z}(\varphi\gamma_k)
		\ud z\ud s\\
		&\quad+\int_0^t\sum_{k=1}^m\int_{J_k}\check{B}(\check{u}^\delta(s))\varphi(\beta_k^\delta-\beta_k)\gamma_k\ud z\ud s\\
		&\quad+\int_0^t\langle \check{G}(\check{u}^\delta(s))\ud W(s),\varphi(\beta^\delta-\beta)\gamma \rangle_{L^2(\Gamma^R)}=:\sum_{i=1}^4\mathrm{R}_i^{\varphi,\delta}(t;\check{u}^\delta).
			\end{align*}
In the following lemma, we show that the remainder term $\mathrm{R}^{\varphi,\delta}(t;\check{u}^\delta)$ becomes negligible in the limit $\delta\to 0$.

\begin{lemma} \label{lem:R}
Let \eqref{eq:etaR} and Assumptions \ref{asp:reg}--\ref{asp:reg-gamma} hold. Assume that $1_\Gamma\in L^2_{\beta\gamma}(\Gamma)$ and $u_0\in W^{1,2}_{(\alpha+1)\gamma,\beta\gamma}(\Gamma)$.	 
Then for any real valued function $\varphi\in \oplus_{k=1}^m \mathcal{C}_c^\infty(J_k)$,
\begin{align*}
\lim_{\delta\to 0}\E\bigg[\sup_{t\in[0,T]}|\mathrm{R}^{\varphi,\delta}(t;\check{u}^\delta)|^2\bigg]=0.
\end{align*}
\end{lemma}
\begin{proof}
Invoking the Cauchy--Schwarz inequality, together with the uniform lower bound $\beta_k^\delta\ge \mathfrak{c}_3>0$, we have
\begin{align*}
|\mathrm{R}_1^{\varphi,\delta}(t;\check{u}^\delta)|^2\le \frac{1}{\mathfrak{c}_3}\left(\|\check{u}^\delta(t)\|_{L^2_{\beta^\delta \gamma}(\Gamma^R)}^2+\|\check{u}(0)\|_{L^2_{\beta^\delta \gamma}(\Gamma^R)}^2\right)\sum_{k=1}^m\int_{J_k}\varphi^2|\beta_k-\beta_k^\delta|^2\gamma_k\ud z
\end{align*}
Similarly, from the linear growth of $\check B: L^2_{\beta\gamma}(\Gamma^R)\to L^2_{\beta\gamma}(\Gamma^R)$, we can readily obtain 
\begin{align*}
|\mathrm{R}_3^{\varphi,\delta}(t;\check{u}^\delta)|^2\le \frac{C}{\mathfrak{c}_3}\int_0^T\Big(1+\|\check{u}^\delta(s)\|_{L^2_{\beta^\delta \gamma}(\Gamma^R)}^2\Big)\ud s\sum_{k=1}^m\int_{J_k}\varphi^2|\beta_k-\beta_k^\delta|^2\gamma_k\ud z.
\end{align*}
Next, by Doob's martingale inequality and the linear growth of $\check G: L^2_{\beta\gamma}(\Gamma^R)\to \mathscr{L}_2(\mathcal{U}_0,L^2_{\beta\gamma}(\Gamma^R))$, 
\begin{align*}
\E\left[\sup_{t\in[0,T]}|\mathrm{R}_4^{\varphi,\delta}(t;\check{u}^\delta)|^2\right]&\le C\E\int_0^T\sum_{i=0}^\infty\langle \check{G}(\check{u}^\delta(s))e_i,\varphi(\beta^\delta-\beta)\gamma \rangle_{L^2(\Gamma^R)}^2 \ud s\\
&\le \frac{C}{\mathfrak{c}_3}\E\int_0^T\Big(1+\|\check{u}^\delta(s)\|_{L^2_{\beta^\delta \gamma}(\Gamma^R)}^2\Big)\ud s\sum_{k=1}^m\int_{J_k}\varphi^2|\beta_k-\beta_k^\delta|^2\gamma_k\ud z.
\end{align*}
Consequently, for $i=1,3,4$ we employ Lemma \ref{lem:analydelta} to infer that
\begin{align*}
\E\left[\sup_{t\in[0,T]}|\mathrm{R}_i^{\varphi,\delta}(t;\check{u}^\delta)|^2\right]
&\le C\E\Big(1+\sup_{s\in[0,T]}\|\check{u}^\delta(s)\|_{L^2_{\beta^\delta \gamma}(\Gamma^R)}^2\Big)\sum_{k=1}^m\int_{J_k}\varphi^2|\beta_k-\beta_k^\delta|^2\gamma_k\ud z\\
&\le C\sum_{k=1}^m\int_{J_k}\varphi^2|\beta_k-\beta_k^\delta|^2\gamma_k\ud z.
\end{align*}
From $\beta_k^\delta \le \mathfrak{c}_5 \beta_k$ it follows that for all $\delta \in (0,\delta_0)$,
$
\varphi^2 \lvert \beta_k - \beta_k^\delta \rvert^2 \gamma_k
\le 2\bigl(1 + \mathfrak{c}_5^2\bigr)\varphi^2 \lvert \beta_k \rvert^2 \gamma_k,
$
where the right-hand side is integrable over $\Gamma^R$. Together with the pointwise convergence of
$\beta_k^\delta$ to $\beta_k$ as $\delta \to 0$ in Assumption \ref{asp:reg}, the Lebesgue dominated convergence theorem implies
\begin{align*}
\lim_{\delta \to 0}
\sum_{k=1}^m \int_{J_k}
\varphi^2 \lvert \beta_k - \beta_k^\delta \rvert^2 \gamma_k \ud z
= 0 .
\end{align*}
As for $\mathrm{R}_2^{\varphi,\delta}$, using the Cauchy--Schwarz inequality again, 
\begin{align*}
|\mathrm{R}_2^{\varphi,\delta}(t;\check{u}^\delta)|^2\le T\int_0^T\big\|\frac{\partial}{\partial z} \check{u}^\delta(s)\big\|_{L^2_{\alpha^{R,\delta}\gamma}(\Gamma^R)}^2\ud s
\left(\sum_{k=1}^m\int_{J_k} |\alpha_k^R-\alpha_k^{R,\delta}|^2|\frac{\ud}{\ud z}(\varphi\gamma_k)|^2\gamma_k^{-1}(\alpha^{R,\delta}_k)^{-1}
		\ud z\right).
\end{align*}
Since $\varphi$ has compact support on each edge $J_k$, by \eqref{eq:delta1}, the function $\alpha^{R,\delta}\ge \mathfrak{c}_0\alpha_k^R$ admits a positive lower bound on $\text{supp}(\varphi)$ uniformly in $\delta\in(0,\delta_0)$. Thus, according to Lemma \ref{lem:analydelta} and \eqref{eq:gammalowupp}, 
\begin{align*}
\E\bigg[\sup_{t\in[0,T]}|\mathrm{R}_2^{\varphi,\delta}(t;\check{u}^\delta)|^2\bigg]&\le C
\sum_{k=1}^m\int_{J_k} |\alpha_k^R-\alpha_k^{R,\delta}|^2|\frac{\ud}{\ud z}(\varphi\gamma_k)|^2\gamma_k^{-1}(\alpha^{R,\delta}_k)^{-1}
		\ud z\\
		&\le C(R,\phi)\sum_{k=1}^m\int_{J_k} |\alpha_k^R-\alpha_k^{R,\delta}|^2
		\ud z.
\end{align*}
The right-hand side converges to zero as $\delta\to0$ by the dominated convergence theorem, and the pointwise convergence of $\alpha_k^\delta$ to $\alpha_k$ stated in Assumption \ref{asp:reg}.
\end{proof}

With Lemma \ref{lem:R} at hand, we are now ready to present the proof of Theorem \ref{tho:con-in-pro}.

\textbf{Proof of Theorem \ref{tho:con-in-pro}.}
By the tightness of the family $\{\mathbb{P}\circ(\check{u}^\delta)^{-1}\}_{\delta\in(0,\delta_0)}$ in $\mathcal{C}([0,T];L_{\beta\gamma}^2(\Gamma^R))$, the Skorokhod theorem yields that for any two sequences $\{\delta_n\}_{n\in\mathbb{N}}\subset(0,\delta_0)$ and $\{\lambda_n\}_{n\in\mathbb{N}}\subset(0,\delta_0)$ converging to $0$, there exist subsequences $\{\delta_{n(i)}\}_{i\in\mathbb{N}}$ and $\{\lambda_{n(i)}\}_{i\in\mathbb{N}}$ and a sequence of random variables
$\rho_i=\{(w_1^i,w_2^i,\mathcal{W}^i)\}_{i\in\mathbb{N}}$ 
 taking values in
 $\mathcal{C}_1:=\mathcal{C}([0,T];L_{\beta\gamma}^2(\Gamma^R))^2\times\mathcal{C}([0,T];\mathcal{D}'(\Gamma))$, defined on some probability space $(\bar{\Omega},\bar{\mathscr{F}},\bar{\mathbb{P}})$ such that the law of $\rho_i$ coincides with that of $(\check{u}^{\delta_{n(i)}},\check{u}^{\lambda_{n(i)}},W)$ for each $i\in\mathbb{N}$, and that $\rho_i$ converges $\bar{\mathbb{P}}$-a.s. to a $\mathcal{C}_1$-valued random element $\rho:=(w_1,w_2,\mathcal{W})$. Here $\mathcal{D}'(\Gamma)$ is the dual space of $\mathcal{C}_0^\infty(\Gamma)$.
  For both $j=1$ and $j=2$, $w_j^i$ solves \eqref{eq:udelta} with $(W,\mathrm{R}^{\varphi,\delta}(t;\check{u}^\delta))$ replaced by $(\mathcal{W}^i,\mathrm{R}^{\varphi,\delta}(t;w_j^i))$, i.e.,
 	\begin{align}\label{eq:udelta-Pdelta}
		\langle w_j^i(t),\varphi\rangle_{L_{\beta\gamma}^2(\Gamma^R)}=\langle \check{u}(0),\varphi\rangle_{L_{\beta\gamma}^2(\Gamma^R)}
	-\frac12\int_0^t\sum_{k=1}^m\int_{J_k}\alpha_k^R \frac{\partial}{\partial z}w_j^i(s)\frac{\ud}{\ud z}(\varphi\gamma_k)\ud z\ud s\\ \notag+ \int_0^t \langle \check{B}(w_j^i(s)),\varphi\rangle_{L_{\beta\gamma}^2(\Gamma^R)}\ud s
	+\int_0^t\langle \check{G}(w_j^i(s))\ud \mathcal{W}^i(s),\varphi\rangle_{L_{\beta\gamma}^2(\Gamma^R)}+\mathrm{R}^{\varphi,\delta}(t;w_j^i).
	\end{align} 
According to Lemma \ref{lem:R}, both $\mathrm{R}^{\varphi,\delta}(\cdot;w_1^i)$ and $\mathrm{R}^{\varphi,\delta}(\cdot;w_2^i)$ converges in $L^2(\bar{\Omega};\mathcal{C}([0,T]))$ to $0$, as $\delta_{n(i)}$ and $\lambda_{n(i)}$ tends to $0$. Hence, we can take further subsequences of $\{\delta_{n(i)}\}_{i\in\mathbb{N}}$ and $\{\lambda_{n(i)}\}_{i\in\mathbb{N}}$, still denoted by $\{\delta_{n(i)}\}_{i\in\mathbb{N}}$ and $\{\lambda_{n(i)}\}_{i\in\mathbb{N}}$ for simplicity, such that the corresponding remainder terms $\mathrm{R}^{\varphi,\delta}(t;w_1^i)$ and $\mathrm{R}^{\varphi,\delta}(t;w_2^i)$ converges $\bar{\mathbb{P}}$-a.s. to $0$ as $i\to\infty$.
Passing to the limit as $i\to \infty$ on both sides of \eqref{eq:udelta-Pdelta},
 we infer that for $j=1,2$,
  	\begin{align}\label{eq:udelta-P}
		\langle w_j(t),\varphi\rangle_{L_{\beta\gamma}^2(\Gamma^R)}=\langle \check{u}(0),\varphi\rangle_{L_{\beta\gamma}^2(\Gamma^R)}
	-\frac12\int_0^t\sum_{k=1}^m\int_{J_k}\alpha_k^R \frac{\partial}{\partial z}w_j(s)\frac{\ud}{\ud z}(\varphi\gamma_k)\ud z\ud s\\ \notag+ \int_0^t \langle \check{B}(w_j(s)),\varphi\rangle_{L_{\beta\gamma}^2(\Gamma^R)}\ud s
	+\int_0^t\langle \check{G}(w_j(s))\ud \mathcal{W}(s),\varphi\rangle_{L_{\beta\gamma}^2(\Gamma^R)}.
	\end{align} 
By the pathwise uniqueness of the solution to \eqref{eq:udelta-P}, we obtain that $w_1=w_2$, $\bar{\mathbb{P}}$-a.s.
	Due to Lemma \ref{lem:GK}, this proves that the family $\{\check{u}^\delta\}_{\delta\in(0,\delta_0)}$ of random variables converges in probability to some random variable $\check{u}^0\in \mathcal{C}([0,T];L_{\beta\gamma}^2(\Gamma^R))$. We can find a subsequence of $\{\check{u}^\delta\}_{\delta\in(0,\delta_0)}$ converging $\mathbb{P}$-a.s. to $\check{u}^0$.
    This, along with \eqref{eq:udelta} and Lemma \ref{lem:R}, yields that $\check{u}^0$ also satisfies \eqref{eq:udelta-P} with $\mathcal{W}$ replaced by $W$. This identifies that $\check{u}^0=\check{u}$, by the pathwise uniqueness of \eqref{eq:local}. In other words, for any $\epsilon>0$, 
\begin{equation}\label{eq:con-in-pro}
\lim_{\delta\to 0}\mathbb{P}\left\{\|\check{u}^\delta-\check{u}\|_{\mathcal{C}([0,T];L_{\beta\gamma}^2(\Gamma^R))}>\epsilon\right\}=0.
\end{equation}

Due to the Vatali convergence theorem, for any fixed $p\ge1$, the family $\{\check{u}^\delta\}_{\delta\in(0,\delta_0)}$ converges in $L^p(\Omega;\mathcal{C}([0,T];L_{\beta\gamma}^2(\Gamma^R)))$ to $\check{u}$ as $\delta\to0$ if and only if $\{\check{u}^\delta\}_{\delta\in(0,\delta_0)}$ converges in probability to $\check{u}$ as $\delta\to0$ and the family $\{\|\check{u}^\delta\|_{\mathcal{C}([0,T];L_{\beta\gamma}^2(\Gamma^R))}^p\}_{\delta\in(0,\delta_0)}$ is uniformly integrable.   
Owing to Lemma \ref{lem:rHolder},  \eqref{eq:gammalowupp}, and $r^\delta(0)=0$, we obtain that for any $q> 2$, 
\begin{align}\label{eq:rdeltaee-new}
\E\bigg[\sup_{t\in[0,T]}\|r^\delta(t)\|_{L_{\beta\gamma}^2(\Gamma^R)}^q\bigg]\le C\E\bigg[\sup_{t\in[0,T]}\|r^\delta(t)-r^\delta(0)\|_{L_{\beta}^2(\Gamma^R)}^q\bigg]\le C(R,q,T).
\end{align}
Since $\check{u}^\delta=v^\delta+r^\delta$, combining \eqref{eq:rdeltaee-new} with Lemma \ref{lem:uHolder} gives that
\begin{align*}
\sup_{t\in[0,T]}\|\check{u}^\delta(t)\|_{L_{\beta\gamma}^2(\Gamma^R)}\le \sup_{t\in[0,T]}\|v^\delta(t)\|_{W^{1,2}_{\alpha^{R}\gamma,\beta\gamma}(\Gamma^R)}+\sup_{t\in[0,T]}\|r^\delta(t)\|_{L_{\beta\gamma}^2(\Gamma^R)}
\end{align*}
is uniformly bounded in $L^{q}(\Omega)$ for any $q>p$ with respect to $\delta\in(0,\delta_0)$. In particular, the family $\{\|\check{u}^\delta\|_{\mathcal{C}([0,T];L_{\beta\gamma}^2(\Gamma^R))}^p\}_{\delta\in(0,\delta_0)}$ is uniformly integrable, which, in combination with \eqref{eq:con-in-pro}, allows us to conclude the proof of Theorem \ref{tho:con-in-pro}.	\hfill$\square$
\section{Error analysis of finite element approximation}\label{S:FEM}
In this section, we prove Theorem \ref{tho:FEM} on the error analysis of the finite element method for the truncated regularized problem \eqref{eq:reguarlized}. 
We emphasize that the generic constant $C$ in this section may depend on $R$ and $\delta$, but is independent of $h$.
To simplify the notation, henceforth we denote 
\begin{align*}
\check{L}_\delta^{2}:=L^2_{\beta^\delta\gamma}(\Gamma^R),\quad \check{W}_\delta^{1,2}:=W^{1,2}_{\alpha^{R,\delta}\gamma,\beta^\delta\gamma}(\Gamma^R).
\end{align*}
By repeating the proof of  Lemma \ref{lem:analydelta}, we have the following lemma.
\begin{lemma}\label{lem:analydeltah}
Let Assumptions \ref{asp:reg}--\ref{asp:reg-gamma} hold, $R> H_0$, and $\delta\in(0,\delta_0)$. Then we have the following.
\begin{enumerate}
\item[(1)] For any $\lambda\ge\frac18\kappa_1$, $-\lambda I+\check{\mathcal{L}}^\delta_h$ generates an analytic contraction semigroup on $\mathbb{V}_h(\Gamma^R)$ such that for any $t>0$,
\begin{equation*}
\|e^{t\check{\mathcal{L}}^\delta_h}\|_{\mathscr{L}(\check{L}_\delta^{2})}\le e^{\frac18\kappa_1t}.
\end{equation*}

\item[(2)] If $\{1_{\Gamma},u_0\}\subset L^2_{\beta\gamma}(\Gamma)$, then the discrete regularized truncated equation \eqref{eq:udeltah} admits a unique mild solution such that
for any $R>H_0$ and $\delta\in(0,\delta_0)$,
\begin{equation}\label{eq:L2delta}
	\E\bigg[\sup_{t\in[0,T]}\|\check{u}_h^\delta(t)\|^2_{\check{L}_\delta^{2}}\bigg]\le C\Big(1+\|u_0\|_{L^2_{\beta\gamma}(\Gamma)}^2\Big),
\end{equation}
where $C:=C(\kappa_1,T)>0$ is independent of $\delta$ and $R$.
\end{enumerate}
\end{lemma}

Fix a $\lambda>\frac{1}{8}\kappa_1$ with $\kappa_1$ as in Assumption \ref{asp:reg-gamma}, and let $0<\theta_0<\theta_1<\frac{\pi}{2}$ be the same as in the proof of Lemma \ref{lem:smo}. As we have already shown in \eqref{eq:analytic}, the resolvent set $\rho(-\check{\mathcal{L}}^\delta)$ of $-\check{\mathcal{L}}^\delta$ contains the sector $\Sigma_\pi^\lambda(\theta_1):=\{y\in\mathbb{C}:\theta_1\le |\arg (y+\lambda)|\le \pi\}$ and  
$$\|(yI+\check{\mathcal{L}}^\delta)^{-1}\|_{\mathscr{L}(\check{L}_\delta^{2})}\le\frac{1}{\sin(\theta_1-\theta_0)} \frac{1}{|y+\lambda|}\quad \forall~y\in \Sigma_{\pi}^\lambda(\theta_1).$$ We shall use the contour $\mathcal{Y}_\lambda:=\{y\in \C:y=-\lambda+r e^{ \pm\mathbf{i}\theta_1},r\ge 0\}$. Then
\begin{equation}\label{eq:L-Lh}
	E_h(t)\varphi:=e^{t\check{\mathcal{L}}^\delta}\varphi-e^{t\check{\mathcal{L}}^\delta_h}\mathcal{P}_h\varphi=\frac{1}{2\pi \mathbf{i}}\int_{\mathcal{Y}_\lambda}e^{-y t}(w-w_h) \ud y,
	\end{equation}
	where $w=(yI+\check{\mathcal{L}}^\delta)^{-1}\varphi$ and $w_h=(yI+\check{\mathcal{L}}^\delta_h)^{-1}\mathcal{P}_h\varphi$.

    \begin{figure}[htbp]
    \centering
    \includegraphics[width=0.5\textwidth]{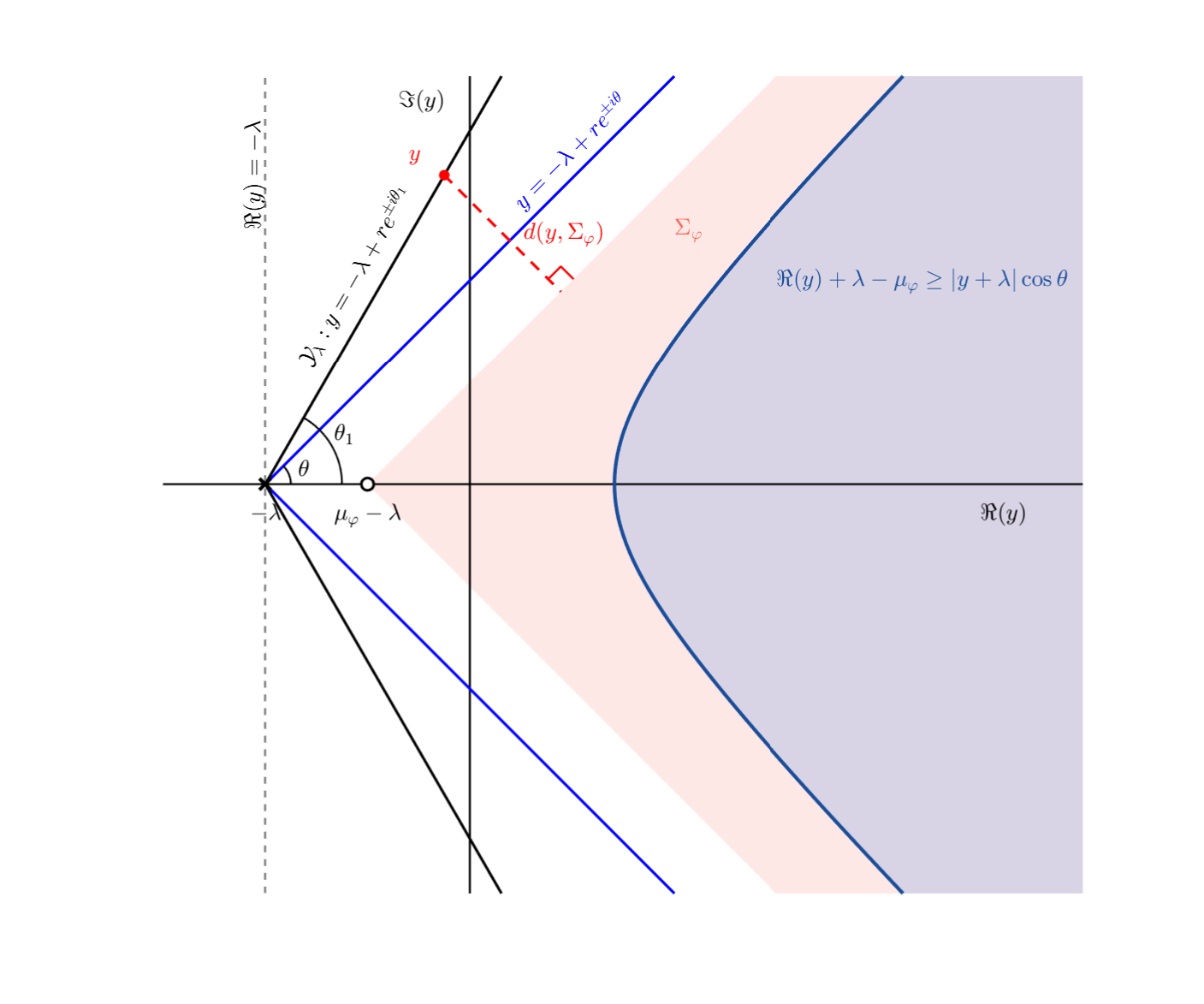}
    \caption{Distance between $y\in\mathcal{Y}_\lambda$ and $\Sigma_\varphi$ in Lemma \ref{lem1}.}
    \label{fig:contour}
\end{figure}
\begin{lemma}\label{lem1}
	Let Assumption \ref{asp:reg-gamma} hold and $\lambda>\frac18\kappa_1$. Then there exists $C:=C(\lambda)>0$ 
	such that for any $R>H_0$, $\delta\in(0,\delta_0)$, $\varphi\in W^{1,2}_{\mathrm{cont}}(\Gamma^R)$, and $y\in \mathcal{Y}_\lambda$,
	\begin{align*}
		|y+\lambda|\|\varphi\|_{\check{L}_\delta^{2}}^2 +\|\varphi\|^2_{\check{W}_\delta^{1,2}}\le C\big|y \|\varphi\|_{\check{L}_\delta^{2}}^2-a_0^{R,\delta}(\varphi,\varphi)\big|.
	\end{align*}
\end{lemma}
\begin{proof}
Recall that $\theta_0=\arccos(c(\lambda)/C(\lambda))$ and $\theta_1\in (\theta_0,\frac{\pi}{2})$, where  $c(\lambda)$ and $C(\lambda)$ are the constants given in \eqref{eq:coe-delta} and \eqref{eq:cont-delta}, respectively.
Choose $\theta\in(\theta_0,\theta_1)$ and set $c_1:=C(\lambda)\cos\theta$. If we set $\ell:=c(\lambda)-c_1>0$, then by \eqref{eq:coe-delta} and \eqref{eq:cont-delta},
\begin{align*}
\Re a_{0}^{R,\delta}(\varphi,\varphi)+\lambda\|\varphi\|_{\check{L}_\delta^{2}}^2-\ell \|\varphi\|_{\check{W}_\delta^{1,2}}^2
\ge c_1\|\varphi\|_{\check{W}_\delta^{1,2}}^2
\ge \cos\theta \big|a_{0}^{R,\delta}(\varphi,\varphi)+\lambda\|\varphi\|_{\check{L}_\delta^{2}}^2\big|.
\end{align*}
Dividing both sides by $\|\varphi\|_{\check{L}_\delta^{2}}$ and by denoting $\mu_\varphi:=\ell\|\varphi\|_{\check{W}_\delta^{1,2}}^2/\|\varphi\|^2_{\check{L}_\delta^{2}}$, we have (see Fig.~\ref{fig:contour}),
\begin{align*}
\frac{a_{0}^{R,\delta}(\varphi,\varphi)}{\|\varphi\|_{\check{L}_\delta^{2}}}\in \Sigma_\varphi:=\left\{y\in\C:|\arg(y+\lambda-\mu_\varphi)|\le \theta\right\}.
\end{align*}
Notice that
for $y\in \mathcal{Y}_\lambda$ (see Fig.~\ref{fig:contour}), 
\begin{equation*}
\textup{dist}(y,\Sigma_\varphi)\ge |y+\lambda|\sin(\theta_1-\theta)+\sin(\theta)\mu_\varphi.
\end{equation*}
Consequently, for any $y\in\mathcal{Y}_\lambda$, 
\begin{align*}
\big|y \|\varphi\|_{\check{L}_\delta^{2}}^2-a_0^{R,\delta}(\varphi,\varphi)\big|\ge \textup{dist}(y,\Sigma_\varphi)\|\varphi\|_{\check{L}_\delta^{2}}^2
\ge |y+\lambda|\sin(\theta_1-\theta)\|\varphi\|_{\check{L}_\delta^{2}}^2+\sin(\theta)\ell\|\varphi\|_{\check{W}_\delta^{1,2}}^2.
\end{align*}
This completes the proof.
\end{proof}

Denote by $\Pi_h \varphi\in \mathbb{V}_h(\Gamma^R)$ the Lagrangian first-order interpolation of a function $\varphi:\Gamma^R\to \R$ in the nodes and the vertices. When restricted to each edge $J_k$, the functions in $\mathbb{V}_h(\Gamma^R)$ are piecewise linear.  Thus, we have the following interpolation error estimate  (see, e.g., \cite[Theorem 4.4.4]{BS08})
	\begin{gather}\label{int:L2}
			\|\Pi_h \varphi-\varphi\|_{L^2(\Gamma^R)}\le Ch\|\varphi'\|_{L^2(\Gamma^R)},\\\label{int:H1}
		\|(\Pi_h \varphi-\varphi)'\|_{L^2(\Gamma^R)}\le Ch\|\varphi''\|_{L^2(\Gamma^R)},
		\end{gather}
		provided that the right-hand sides are finite. 
       The following lemma extends the standard interpolation error estimates \eqref{int:L2} and \eqref{int:H1} to weighted Sobolev spaces on the graph $\Gamma^R$.
\begin{lemma}\label{lem:int-L}
Let Assumptions \ref{asp:reg}, \ref{asp:reg-gamma} and \ref{asp:alphader} hold. Then for any fixed $R> H_0$ and $\delta\in(0,\delta_0)$, there exists $C:=C(R,\delta)>0$ such that
\begin{gather}\label{int:L2wgt}
					\|\Pi_h \varphi-\varphi\|_{\check{L}_\delta^{2}}\le C(R,\delta)h\|\varphi\|_{\check{W}_\delta^{1,2}}\quad \forall~\varphi\in\check{W}_\delta^{1,2}, \\\label{eq:H1err}
			\|\Pi_h \varphi-\varphi\|_{\check{W}_\delta^{1,2}}\le C(R,\delta)h(\|\varphi\|_{\check{W}_\delta^{1,2}}+\|\check{\mathcal{L}}^\delta \varphi\|_{\check{L}_\delta^2})\quad\forall~ \varphi\in D(\check{\mathcal{L}}^\delta).
\end{gather}
\end{lemma}
\begin{proof}

According to Assumption \ref{asp:reg}, for a fixed $\delta\in(0,\delta_0)$ and $R> H _0$,
\begin{gather}\label{eq:equiv}
\mathfrak{c}_1(\delta)\le \alpha_k^{R,\delta}(z)\le CR,\quad
\mathfrak{c}_3\le \beta_k^{\delta}(z)\le \mathfrak{c}_4(\delta)\quad \forall~(z,k)\in\Gamma^R,
\end{gather}
which, along with \eqref{int:L2} and \eqref{eq:gammalowupp}, yields  \eqref{int:L2wgt}.
Similarly, by \eqref{int:H1} and \eqref{eq:equiv} we also have
\begin{equation}\label{int:H1wgt}
	\int_{J_k}\Big|\frac{\ud}{\ud z}(\Pi_h \varphi-\varphi)\Big|^2{\alpha_k^{R,\delta}\gamma_k}\ud z
			\le C(R,\delta)h^2 \int_{J_k}\Big|\frac{\alpha_k^{R,\delta}}{\beta_k^\delta}\frac{\ud^2}{\ud z^2}\varphi\Big|^2\beta^\delta_k\gamma_k\ud z,\quad k=1,\ldots,m.
		\end{equation}	
From \eqref{eq:Ldelta}, Minkowski's inequality, and H\"older's inequality, it follows that
\begin{align}\label{eq:der1}
&\int_{J_k}\Big|\frac{\alpha_k^{R,\delta}}{\beta_k^\delta}\frac{\ud^2}{\ud z^2}\varphi\Big|^2\beta^\delta_k\gamma_k\ud z-8\int_{J_k}|\check{\mathcal{L}}^\delta \varphi|^2\beta^\delta_k\gamma_k\ud z\le2\int_{J_k}\Big|\frac{(\alpha_k^{R,\delta})'}{\beta_k^\delta}\frac{\ud}{\ud z}\varphi\Big|^2\beta^\delta_k\gamma_k\ud z\\\notag
&\le 2\left(\int_{J_k}\frac{|(\alpha_k^{R,\delta})'|^4}{|\beta_k^\delta|^2}\gamma_k^2\ud z\right)^{\frac12}\|\varphi'\|_{L^{4}(J_k)}^2\le C(R,\delta)\|\varphi'\|_{L^{4}(J_k)}^2,
	\end{align}
	where the last step used \eqref{eq:gammalowupp}, \eqref{eq:equiv}, and Assumption \ref{asp:alphader}.
	The Gagliardo--Nirenberg inequality (see, e.g., \cite{GE59}), together with Young's inequality implies that, for any $\epsilon>0$,
\begin{equation}\label{eq:der2}
\|\varphi'\|_{L^{4}(J_k)}\le C\|\varphi''\|_{L^2(J_k)}^{\frac14}\|\varphi'\|_{L^2(J_k)}^{\frac34}+C\|\varphi'\|_{L^2(J_k)}\le \epsilon \|\varphi''\|_{L^2(J_k)}+C(\epsilon)\|\varphi'\|_{L^2(J_k)}.
\end{equation}	
Combining \eqref{eq:der1} and \eqref{eq:der2} yields that for any $\epsilon>0$,
\begin{align*}
&\int_{J_k}\Big|\frac{\alpha_k^{R,\delta}}{\beta_k^\delta}\frac{\ud^2}{\ud z^2}\varphi\Big|^2\beta^\delta_k\gamma_k\ud z-8\int_{J_k}|\check{\mathcal{L}}^\delta \varphi|^2\beta^\delta_k\gamma_k\ud z\\
&\le\epsilon C(R,\delta)  \|\varphi''\|_{L^2(J_k)}^2+C(\epsilon,R,\delta)\|\varphi'\|_{L^2(J_k)}^2\\
&\le \epsilon C_1(R,\delta)\int_{J_k}|\frac{\alpha_k^{R,\delta}}{\beta_k^\delta}\frac{\ud^2}{\ud z^2}\varphi|^2\beta^\delta_k\gamma_k\ud z+C(\epsilon,R,\delta)\|\varphi'\|_{L_{\alpha^{R,\delta}_k\gamma_k}^2(J_k)}^2,
	\end{align*}
	where in the last step we used \eqref{eq:equiv} and \eqref{eq:gammalowupp}.
If we choose $\epsilon>0$ such that $\epsilon C_1(R,\delta)\le \frac12$, then 
	\begin{align}\label{eq:H2Jk}
\int_{J_k}\Big|\frac{\alpha_k^{R,\delta}}{\beta_k^\delta}\frac{\ud^2}{\ud z^2}\varphi\Big|^2\beta^\delta_k\gamma\ud z
&\le 16\int_{J_k}|\check{\mathcal{L}}^\delta \varphi|^2\beta^\delta_k\gamma_k\ud z+C(R,\delta)\|\varphi'\|_{L_{\alpha^{R,\delta}_k\gamma_k}^2(J_k)}^2.
	\end{align}
Inserting 	\eqref{eq:H2Jk} into \eqref{int:H1wgt}, summing over $k=1,\ldots,m$ and then combining \eqref{int:L2wgt} leads to \eqref{eq:H1err}.
\end{proof}

The following lemma provides an estimate of the error between the resolvents of $\check{\mathcal{L}}^\delta$ and its discrete counterpart.

\begin{lemma}\label{lem:res-est}
	Let Assumptions \ref{asp:reg}, \ref{asp:reg-gamma}, and \ref{asp:alphader} hold. Then for any $\lambda>\frac18\kappa_1$, $R> H_0$ and $\delta\in(0,\delta_0)$,  there exists $C:=C(R,\delta,\lambda)>0$ such that for any $y\in \mathcal{Y}_\lambda$ and $\varphi\in \check{L}_\delta^{2}$, 
	\begin{equation*}
		\|w_h-w\|_{\check{L}_\delta^{2}}+ h\|w_h-w\|_{\check{W}_\delta^{1,2}}\le Ch^2\|\varphi\|_{\check{L}_\delta^{2}},
	\end{equation*}
	where $w=(yI+\check{\mathcal{L}}^\delta)^{-1}\varphi$ and $w_h=(yI+\check{\mathcal{L}}^\delta_h)^{-1}\mathcal{P}_h\varphi$.
\end{lemma}
\begin{proof}
	From the definition of $w$ and $w_h$, it follows that
	\begin{align*}
		y\langle w,\psi\rangle_{\check{L}_\delta^{2}}-a_0^{R,\delta}(w,\psi)
		&=\langle\varphi,\psi\rangle_{\check{L}_\delta^{2}}\quad \forall ~\psi\in W^{1,2}_{\mathrm{cont}}(\Gamma^R),\\
		y\langle w_h,\psi\rangle_{\check{L}_\delta^{2}}-a_0^{R,\delta}(w_h,\psi)&=\langle\varphi,\psi\rangle_{\check{L}_\delta^{2}}\quad \forall ~\psi\in \mathbb{V}_h(\Gamma^R).
	\end{align*}
Since $\mathbb{V}_h(\Gamma^R)\subset W^{1,2}_{\mathrm{cont}}(\Gamma^R)$, subtracting the above two equations yields that $\varpi:=w-w_h$ satisfies
\begin{equation}\label{eq:orth}
	y\langle \varpi,\psi\rangle_{\check{L}_\delta^{2}}-a_0^{R,\delta}(\varpi,\psi)=0\quad\forall~\psi\in \mathbb{V}_h(\Gamma^R).
\end{equation}
Together with Lemma \ref{lem1}, this yields that for any $\psi_1 \in \mathbb{V}_h(\Gamma^R)$,
\begin{align*}
		|y+\lambda|\|\varpi\|_{\check{L}_\delta^{2}}^2 +\|\varpi\|^2_{\check{W}_\delta^{1,2}}&\le C\big|y \|\varpi\|_{\check{L}_\delta^{2}}^2-a_0^{R,\delta}(\varpi,\varpi)\big|\\
		&= C\big|y \langle \varpi,\varpi+(w_h-\psi_1)\rangle_{\check{L}_\delta^{2}}-a_0^{R,\delta}(\varpi,\varpi+(w_h-\psi_1))\big|\\
		&=C\big|y \langle \varpi,w-\psi_1\rangle_{\check{L}_\delta^{2}}-a_0^{R,\delta}(\varpi,w-\psi_1)\big|\\
		&=C\big|(y+\lambda) \langle \varpi,w-\psi_1\rangle_{\check{L}_\delta^{2}}-a_\lambda^{R,\delta}(\varpi,w-\psi_1)\big|.
	\end{align*}		
	In particular, if we choose $\psi_1=\Pi_h w$, then 	by Lemma \ref{lem:int-L} and \eqref{eq:cont-delta},
		\begin{align}\label{eq:yla}
			|y+\lambda|\|\varpi\|_{\check{L}_\delta^{2}}^2 +\|\varpi\|^2_{\check{W}_\delta^{1,2}}&\le C(R,\delta)h|y+\lambda|\|\varpi\|_{\check{L}_\delta^{2}} \|w\|_{\check{W}_\delta^{1,2}}\\\notag
			&\quad+C(R,\delta,\lambda)h\|\varpi\|_{\check{W}_\delta^{1,2}}(\|w\|_{\check{W}_\delta^{1,2}}+\|\check{\mathcal{L}}^\delta w\|_{\check{L}_\delta^2}).
		\end{align}
		Invoking Lemma \ref{lem1} again, we obtain 
			\begin{align*}
			|y+\lambda|\|w\|_{\check{L}_\delta^{2}}^2 +\|w\|^2_{\check{W}_\delta^{1,2}}\le C\big|\langle (yI+\check{\mathcal{L}}^\delta)w,w\rangle_{\check{L}_\delta^{2}}\big|
			\le C\|\varphi\|_{\check{L}_\delta^{2}}\|w\|_{\check{L}_\delta^{2}},
		\end{align*} 
		which reveals that 
			\begin{align}\label{eq:wbound}
			\|w\|_{\check{L}_\delta^{2}}\le C|y+\lambda|^{-1}\|\varphi\|_{\check{L}_\delta^{2}},\quad \|w\|_{\check{W}_\delta^{1,2}}
			\le C|y+\lambda|^{-\frac12}\|\varphi\|_{\check{L}_\delta^{2}}.
		\end{align} 
		
On the other hand, by the first inequality of \eqref{eq:wbound}, along with $w=(yI+\check{\mathcal{L}}^\delta)^{-1}\varphi$, we get
$$\|(-\lambda I+\check{\mathcal{L}}^\delta)w\|_{\check{L}_\delta^{2}}=\|(yI+\check{\mathcal{L}}^\delta)w-(y+\lambda)w\|_{\check{L}_\delta^{2}}\le C(\|\varphi\|_{\check{L}_\delta^{2}}+|y+\lambda|\|w\|_{\check{L}_\delta^{2}})\le C\|\varphi\|_{\check{L}_\delta^{2}}.$$
By the coercivity of $a_{\lambda}^{R,\delta}$ in \eqref{eq:coe-delta}, we also have 
\begin{align*}
	c(\lambda)\|w\|_{\check{W}_\delta^{1,2}}^2\le \Re a_{\lambda}^{R,\delta}(w,w)\le  |a_{\lambda}^{R,\delta}(w,w)|
	= |\langle (\lambda I-\check{\mathcal{L}}^\delta)w,w\rangle_{\check{L}_\delta^{2}}|\le \|w\|_{\check{L}_\delta^{2}}\|\varphi\|_{\check{L}_\delta^{2}}.
\end{align*}
Consequently, it holds that $$\|w\|_{\check{W}_\delta^{1,2}}+\|\check{\mathcal{L}}^\delta w\|_{\check{L}_\delta^{2}}\le (\lambda+1)\|w\|_{\check{W}_\delta^{1,2}}+\|(-\lambda I+\check{\mathcal{L}}^\delta) w\|_{\check{L}_\delta^{2}}\le C\|\varphi\|_{\check{L}_\delta^{2}}.$$
Furthermore, from
 \eqref{eq:yla} and the second inequality of  \eqref{eq:wbound}, we can conclude that 
\begin{align*}
	|y+\lambda|\|\varpi\|_{\check{L}_\delta^{2}}^2 +\|\varpi\|^2_{\check{W}_\delta^{1,2}}
	&\le C(R,\delta)h|y+\lambda|^{\frac12}\| \varpi\|_{\check{L}_\delta^{2}} \|\varphi\|_{\check{L}_\delta^{2}}+C(R,\delta,\lambda)h\|\varpi\|_{\check{W}_\delta^{1,2}}\|\varphi\|_{\check{L}_\delta^{2}}\\
	&\le \frac12|y+\lambda|\|\varpi\|_{\check{L}_\delta^{2}}^2 +\frac12\|\varpi\|^2_{\check{W}_\delta^{1,2}}+ C(R,\delta,\lambda)h^2\|\varphi\|_{\check{L}_\delta^{2}}^2.
	\end{align*}
In other words, we have proven 
\begin{equation}\label{eq:H1norm}
	|y+\lambda|\|\varpi\|_{\check{L}_\delta^{2}}^2 +\|\varpi\|^2_{\check{W}_\delta^{1,2}}
\le C(R,\delta,\lambda)h^2\|\varphi\|_{\check{L}_\delta^{2}}^2.
\end{equation}

Next, we study the $\check{L}_\delta^{2}$-norm of the error $\varpi$ by a duality argument. Given $\varphi_1\in \check{L}_\delta^{2}$, we find $v=(yI+(\check{\mathcal{L}}^\delta)^*)^{-1}\varphi_1$ and $v_h=(yI+(\check{\mathcal{L}}^\delta_h)^*)^{-1}\mathcal{P}_h\varphi_1$ such that
\begin{align}\label{eq:dual1}
	y\langle \psi,v\rangle_{\check{L}_\delta^{2}}-a_0^{R,\delta}(\psi,v)&=-\langle\varphi_1, \psi\rangle_{\check{L}_\delta^{2}},\quad \psi\in W^{1,2}_{\mathrm{cont}}(\Gamma^R),\\\label{eq:dual2}
		y\langle \psi,v\rangle_{\check{L}_\delta^{2}}- a_0^{R,\delta}(\psi,v_h)&=-\langle\varphi_1, \psi\rangle_{\check{L}_\delta^{2}},\quad \psi\in \mathbb{V}_h(\Gamma^R).
\end{align}
Due to \eqref{eq:orth} and \eqref{eq:H1norm}, we have 
\begin{align}\label{eq:dualest}
	|y\langle \varpi,v\rangle_{\check{L}_\delta^{2}}-a_0^{R,\delta}(\varpi,v)|&=|y\langle \varpi,v-v_h\rangle_{\check{L}_\delta^{2}}-a_0^{R,\delta}(\varpi,v-v_h)|\\\notag
	&\le |y+\lambda|\|\varpi\|_{\check{L}_\delta^{2}} \|v-v_h\|_{\check{L}_\delta^{2}} +C\|\varpi\|_{\check{W}_\delta^{1,2}}\|v-v_h\|_{\check{W}_\delta^{1,2}}\\\notag
	&\le C(R,\delta)h^2\|\varphi\|_{\check{L}_\delta^{2}}\|\varphi_1\|_{\check{L}_\delta^{2}}.
\end{align}
Finally, by duality, we infer from \eqref{eq:dual1} and \eqref{eq:dualest} that
\begin{align*}
	\|\varpi\|_{\check{L}_\delta^{2}}=\sup_{\|\varphi_1\|_{\check{L}_\delta^{2}}=1}|\langle \varphi_1,\varpi\rangle_{\check{L}_\delta^{2}}|=\sup_{\|\varphi_1\|_{\check{L}_\delta^{2}}=1}|y\langle \varpi,v\rangle_{\check{L}_\delta^{2}}-a_0^{R,\delta}(\varpi,v)|\le C(R,\delta)h^2\|\varphi\|_{\check{L}_\delta^{2}}.
\end{align*}
The proof is finished.
\end{proof}

\begin{lemma}\label{lem:FEM}
Let Assumptions \ref{asp:reg}, \ref{asp:reg-gamma}, and \ref{asp:alphader} hold, and  $\{u_0,1_\Gamma\}\subset L^2_{\beta\gamma}(\Gamma)$. Then for any $R\ge H_0$, $\delta\in(0,\delta_0)$, and $\vartheta\in(0,\frac12)$, there exists $C:=C(R,\delta,\vartheta)>0$ such that
\begin{align}\label{eq:FEMLdelta}
\E\left[\|\check{u}^\delta(t)-\check{u}^\delta_h(t)\|_{\check{L}_\delta^{2}}^2\right]\le Ch^{4\vartheta}(1+t^{-2\vartheta})\|\check{u}(0)\|_{\check{L}_\delta^{2}}^2.
\end{align}	
\end{lemma}
\begin{proof}
As a consequence of \eqref{eq:L-Lh} and Lemma \ref{lem:res-est}, we have
	\begin{align*}
		\|E_h(t)\varphi\|_{\check{L}_\delta^{2}}\le C\int_{0}^\infty e^{-(r \cos\theta_1-\lambda) t}\|w-w_h\|_{\check{L}_\delta^{2}} \ud r
		&\le C(R,\delta,\lambda)h^2 t^{-1}\|\varphi\|_{\check{L}_\delta^{2}},\quad t\in[0,T].
	\end{align*}
Since $\{e^{t\check{\mathcal{L}}^\delta}\}_{t\ge0}$ and $\{e^{t\check{\mathcal{L}}^\delta_h}\}_{t\ge0}$ are strongly continuous semigroups on $\check{L}_\delta^{2}$ (see Lemmas \ref{lem:analydelta} and \ref{lem:analydeltah}), by interpolation, for any $\vartheta\in[0,1]$ it holds
	\begin{align}\label{eq:Sht}
		\|E_h(t)\varphi\|_{\check{L}_\delta^{2}}\le C(R,\delta,\lambda,T)h^{2\vartheta} t^{-\vartheta}\|\varphi\|_{\check{L}_\delta^{2}},\quad t\in[0,T].
	\end{align}
Subtracting  \eqref{eq:udeltah} from
\eqref{eq:reguarlized} gives that for any $t\in[0,T]$,
\begin{align*}
\check{u}^\delta(t)-\check{u}^\delta_h(t)=E_h(t)\check{u}(0)+\int_0^t E_h(t-s)\check{B}(\check{u}^\delta(s))\ud s+\int_0^t e^{(t-s)\check{\mathcal{L}}^\delta_h}\mathcal{P}_h\big(\check{B}(\check{u}^\delta(s))-\check{B}(\check{u}_h^\delta(s))\big)\ud s\\
+\int_0^tE_h(t-s)\check{G}(\check{u}^\delta(s))\ud W(s)+\int_0^te^{(t-s)\check{\mathcal{L}}^\delta_h}\mathcal{P}_h\big(\check{G}(\check{u}^\delta(s))-\check{G}(\check{u}_h^\delta(s))\big)\ud W(s).
 \end{align*}
It follows from the boundedness of the semigroup generated by the discrete operator $\check{\mathcal{L}}^\delta_h$, the contraction of the projection $\mathcal{P}_h$,
the H\"older inequality, the Burkholder inequality, and \eqref{eq:Sht} that for any $\vartheta\in(0,\frac12)$,
\begin{align*}
\E\left[\|\check{u}^\delta(t)-\check{u}^\delta_h(t)\|^2_{\check{L}_\delta^{2}}\right]&\le C(R,\delta)h^{4\vartheta}t^{-2\vartheta}\|\check{u}(0)\|_{\check{L}_\delta^{2}}^2+
C(R,\delta)h^{4\vartheta}\int_0^t  (t-s)^{-\vartheta}\E\left[\|\check{B}(\check{u}^\delta(s))\|_{\check{L}_\delta^{2}}^2\right]\ud s\\
&\quad+ C\int_0^t\E\left[\|\check{B}(\check{u}^\delta(s))-\check{B}(\check{u}_h^\delta(s))\|_{\check{L}_\delta^{2}}^2\right]\ud s\\
&\quad+C(R,\delta)h^{4\vartheta}\int_0^t  (t-s)^{-2\vartheta}\E\left[\|\check{G}(\check{u}^\delta(s))\|_{\mathscr{L}_2(\mathcal{U}_0,\check{L}_\delta^{2})}^2\right]\ud s\\
&\quad +C\int_0^t\E\left[\|\check{G}(\check{u}^\delta(s))-\check{G}(\check{u}_h^\delta(s))\|^2_{\mathscr{L}_2(\mathcal{U}_0,\check{L}_\delta^{2})}\right]\ud s.
\end{align*}	
By the linear growth and Lipschitz continuity of both $\check{B}$ and $\check{G}$, as well as Lemma \ref{lem:analydelta}, we arrive at
\eqref{eq:FEMLdelta}
 by applying the Gr\"onwall inequality.
\end{proof}
Based on Lemma \ref{lem:FEM}, we now present the proof of Theorem \ref{tho:FEM}.

\textbf{Proof of Theorem \ref{tho:FEM}.}
Applying \eqref{eq:L2beta} leads to 
\begin{align}\label{eq:u-uh}
		\|\check{u}^\delta(t)-\check{u}^\delta_h(t)\|_{L^2_{\beta\gamma}(\Gamma^R)}^2\le C \|\check{u}^\delta(t)-\check{u}^\delta_h(t)\|_{\check{L}_\delta^{2}}\left(\|\check{u}^\delta(t)\|_{\check{W}_\delta^{1,2}}+\|\check{u}^\delta_h(t)\|_{\check{W}_\delta^{1,2}}\right).
	\end{align}
Due to $\check{u}^\delta=v^\delta+r^\delta$, it follows from \eqref{eq:vdeltadelta} and \eqref{eq:rW12} that for any $p\ge2$, 
\begin{align}\label{eq:checkudelta}
\E\bigg[\|\check{u}^\delta(t)\|_{\check{W}_\delta^{1,2}}^p\bigg]\le C(R,p).
\end{align}
Similarly, it can be shown that $p\ge2$, 
\begin{align}\label{eq:checkudeltah}
\E\bigg[\|\check{u}^\delta_h(t)\|_{\check{W}_\delta^{1,2}}^p\bigg]\le C(R,p).
\end{align}
Taking expectations on both sides of \eqref{eq:u-uh}, and then using Lemma \ref{lem:FEM}, \eqref{eq:checkudelta}, \eqref{eq:checkudeltah}, we infer from the H\"older inequality that for any $\vartheta\in(0,\frac12)$
\begin{align*}
		\E\left[\|\check{u}^\delta(t)-\check{u}^\delta_h(t)\|_{L^2_{\beta\gamma}(\Gamma^R)}^2\right]\le C \left(\E\left[\|\check{u}^\delta(t)-\check{u}^\delta_h(t)\|_{\check{L}_\delta^{2}}^2\right]\right)^{\frac12}\le Ch^{2\vartheta}(1+t^{-\vartheta}),
	\end{align*}
	which completes the proof of Theorem \ref{tho:FEM}.
\hfill$\square$

\bibliographystyle{abbrv}
\bibliography{references}
\end{document}